\documentclass[a4paper,11pt,oneside]{article}

\usepackage[utf8]{inputenc}
\usepackage[T1]{fontenc}
\usepackage[english]{babel}

\usepackage{answers}
\usepackage{setspace}
\usepackage{graphicx}
\usepackage{enumitem}
\usepackage{multicol}
\usepackage{mathrsfs}
\usepackage{amsmath,amsthm,amssymb,amsfonts}
\usepackage{bbm}
\usepackage{esdiff}
\usepackage{braket}
\usepackage{tikz}
\usepackage{dsfont}
\usepackage{hyperref}
\usepackage{bm}
\usepackage{comment}
\usepackage{dirtytalk}
\usepackage{mathtools,tikz-cd}
\usepackage{color,soul}
\usepackage{cleveref}
\usepackage{nicefrac, xfrac}
\usepackage{titling}

\newtheorem{innercustomgeneric}{\customgenericname}
\providecommand{\customgenericname}{}
\newcommand{\newcustomtheorem}[2]{%
  \newenvironment{#1}[1]
  {%
   \renewcommand\customgenericname{#2}%
   \renewcommand\theinnercustomgeneric{##1}%
   \innercustomgeneric
  }
  {\endinnercustomgeneric}
}

\newcustomtheorem{customassumptions}{Assumptions}
\newcustomtheorem{customassumption}{Assumption}

\makeatletter
\newtheorem*{rep@theorem}{\rep@title}
\newcommand{\newreptheorem}[2]{%
\newenvironment{rep#1}[1]{%
 \def\rep@title{#2 \ref{##1}}%
 \begin{rep@theorem}}%
 {\end{rep@theorem}}}
\makeatother

\newreptheorem{theorem}{Theorem}
\newreptheorem{lemma}{Lemma}

\usepackage[toc,page]{appendix}

\newcommand{\nocontentsline}[3]{}
\newcommand{\tocless}[2]{\bgroup\let\addcontentsline=\nocontentsline#1{#2}\egroup}

\usepackage[a4paper,top=3cm,bottom=3cm,left=2.5cm,right=2.5cm]{geometry}
\numberwithin{equation}{section}

\newcommand{\N}{\mathbb{N}}

\newcommand{\R}{\mathbb{R}}
\newcommand{\Q}{\mathbb{Q}}
\newcommand{\E}{\mathbb{E}}

\newcommand{\attesa}[1]{\mathbb{E}\quadre{{#1}}}
\newcommand{\J}{\mathfrak{J}}
\newcommand{\A}{\mathbb{A}}
\newcommand{\F}{\mathcal{F}}
\newcommand{\prob}{\mathbb{P}}
\newcommand{\wienermeasure}{\mathbb{W}^d}

\newcommand{\probmeasures}[1]{\mathcal{P}(#1)}
\newcommand{\contrd}{\mathcal{C}^d}
\newcommand{\conttraj}[1]{\mathcal{C}\tonde{[0,T];#1}}
\newcommand{\cbounded}[1]{\mathcal{C}_b(#1)}

\newcommand{\contpdue}{\mathcal{C}(\mathcal{P}^2)}

\newcommand{\boreliani}[1]{\mathcal{B}_{#1}}
\newcommand{\pwassspace}[2]{
\mathcal{P}^{#1}\tonde{#2}}
\newcommand{\pwassmetric}[3]{
\mathcal{W}_{#1,#2}^{#3}}

\newcommand{\1}{\mathds{1}}
\newcommand{\tonde}[1]{\left({#1}\right)}
\newcommand{\quadre}[1]{\left[{#1}\right]}
\newcommand{\graffe}[1]{\left\lbrace{#1}\right\rbrace}
\newcommand{\abs}[1]{\left\lvert{#1}\right\rvert}
\newcommand{\norm}[1]{\left\lVert{#1}\right\rVert}

\newcommand{\eps}{\varepsilon}

\newcommand{\insieme}[1]{\left\lbrace{#1}\right\rbrace}

\usepackage{xcolor}
\definecolor{Bittersweet}{HTML}{C04F17}

\theoremstyle{plain}
\newtheorem{theorem}{Theorem}[section]
\newtheorem{lemma}[theorem]{Lemma}
\newtheorem{proposition}[theorem]{Proposition}

\theoremstyle{definition}
\newtheorem{definition}{Definition}
\newtheorem{example}{Example}
\newtheorem{remark}{Remark}

\thanksmarkseries{alph}

\title{Coarse correlated equilibria for continuous time mean field games in open loop strategies}

\author{Luciano Campi\thanks{Department of Mathematics "Federigo Enriques", University of Milan, Via Saldini 50, 20133, Milan, Italy. E-mail address: \href{mailto:luciano.campi@unimi.it}{luciano.campi@unimi.it}} \and Federico Cannerozzi\thanks{Department of Mathematics "Federigo Enriques", University of Milan, Via Saldini 50, 20133, Milan, Italy. \\
Center for Mathematical Economics (IMW), Bielefeld University, Universit\"atsstrasse 25, 33615, Bielefeld, Germany. E-mail address: \href{mailto:federico.cannerozzi@uni-bielefeld.de}{federico.cannerozzi@uni-bielefeld.de}.} \and Markus Fischer\thanks{Department of Mathematics "Tullio Levi-Civita", University of Padua, via Trieste 63, 35121, Padova, Italy. E-mail address: \href{mailto:fischer@math.unipd.it}{fischer@math.unipd.it}.}}

\date{\today}

\begin{document}

\maketitle

\begin{abstract}
In the framework of continuous time symmetric stochastic differential games in open loop strategies, we introduce a generalization of mean field game solution, called coarse correlated solution. This can be seen as the analogue of a coarse correlated equilibrium in the $N$-player game. We justify our definition by showing that a coarse correlated solution for the mean field game induces a sequence of approximate coarse correlated equilibria with vanishing error for the underlying $N$-player games. Existence of coarse correlated solutions for the mean field game is proved by a minimax theorem. An example with explicit solutions is discussed as well.
\end{abstract}


\vspace{0.2cm}

\section{Introduction}

Coarse correlated equilibria are a concept of equilibria for games with many players which allows for correlation between players' strategies, thus generalizing the notion of Nash equilibria.
In this paper, we propose a notion of coarse correlated equilibria for a class of continuous time symmetric stochastic differential games and study the corresponding mean field formulation as the number of players $N$ goes to infinity.

Mean field games (MFGs) have been an active theme of research for almost two decades, started in the mid 2000's from the seminal works of Lasry and Lions \cite{lasry_lions} and of Huang, Malham\'e and Caines \cite{huang_malhame_caines}.
Roughly speaking, MFGs arise as the limit formulation of symmetric stochastic $N$-player games with mean field interactions between the players.
Thanks to the mean field interaction and propagation of chaos type results, one expects that the empirical distribution of players' states converges to the law of some representative player.
In the limit, the concept of Nash equilibrium translates into a fixed point problem in the space of flows of measures.
For a probabilistic approach to MFGs, we refer to the two-volume book by Carmona and Delarue \cite{librone_vol1,librone_vol2}.
The relation between the MFG and the $N$-player game is commonly understood in two ways: on the one hand, a solution of the MFG allows to construct approximate Nash equilibria for the corresponding $N$-player games, if $N$ is sufficiently large, see, e.g., \cite{campi18absorption,car_del_probabilistic,carmona_lacker2015weak,huang_malhame_caines}.
On the other hand, approximate Nash equilibria can be shown to converge to solutions of the corresponding MFG.
The choice of admissible strategies, while always important, is crucial for results of this kind: see \cite{fischer2017connection,lacker2016} for earlier results in open loop strategies and Cardaliaguet et al. \cite{cardaliaguet2019master_equation}, Lacker \cite{lacker2020convergence}, Lacker and Le~Flem \cite{lacker_leflem2022} for convergence in closed loop strategies, and the works by Djete \cite{djete2022extended,djete2023large,djete2023convergence}, for MFGs of controls.

\smallskip
The notion of coarse correlated equilibrium (CCE) makes its first appearances implicitly in Hannan's work \cite{hannan1957} and explicitly in Moulin's and Vial's \cite{moulin_vial1978}.
The idea of CCEs can be summarized as follows: The game includes a correlation device or a mediator, who picks a strategy profile randomly according to some probability distribution over the set of strategy profiles, which is assumed to be common knowledge among the players.
Each player must decide whether to commit or not to the strategies selected for her by the mediator \emph{before} the mediator runs the lottery. If a player deviates, she will do so without any information on the outcome of the lottery.
If a player commits, the mediator informs her of her own recommendation, without revealing the recommendation to any other player.
In equilibrium, it is
best to commit to the anticipated outcome of the lottery if one believes that every other player is doing the same.
This notion of equilibrium is weaker than that of correlated equilibrium (CE) \`a la Aumann (see \cite{aumann1970,aumann1987}), where players decide whether to accept the mediator's recommendation \emph{after} having been informed (in private) of the strategies extracted for them.
When the distribution used by the mediator is a product distribution, CCEs reduce to usual Nash equilibria in mixed strategies, because in this case the mediator's recommendations do not carry any additional information over what is common knowledge. 
Among the nice features of CCEs, we notice the fact that they may lead to higher payoffs than Nash equilibria, even when true CEs do not exist (see Moulin et al. \cite{moulinraysengupta2014,Moulin2014Coarse} for an example in a two-person static linear quadratic game), and they naturally arise from a learning procedure of the players, such as the so called regret-based dynamics (see, e.g., Hart and Mas-Colell \cite{hart_mascolell_regret_based} and Roughgarden \cite[Section 17.4]{roughgarden_2016}).

\smallskip
Recently, correlation between players' strategy choices has been considered in the context of mean field games. Bonesini, Campi and Fischer \cite{bonesiniCE,campifischer2021} establish the existence of symmetric CEs in a class of symmetric games with discrete time and finite state and action spaces, give a definition of CE in the mean field limit and provide both approximation and convergence results.
In a second group of papers by M\"{u}ller et al. \cite{muller2022learningCE,muller2021learning}, notions of both CCEs and CEs are studied for a class of symmetric games with discrete time, finite states and finite actions, in a setting close to the one in \cite{bonesiniCE,campifischer2021}.
In addition, \cite{muller2022learningCE,muller2021learning} contain an extensive discussion of learning algorithms for approximating Nash equilibria, CEs and CCEs in the mean field limit.
Then, \cite{campi2023LQ} introduces the notion of CCEs in a class of continuous time linear quadratic MFGs.
A methodology to compute CCEs in such class of MFGs is provided, and, through the study of a simple yet important example with applications in environmental economics, the authors show that there exist infinitely many CCEs for the MFG which both yield higher payoffs than the classical MFG solutions and are more efficient with respect to the environmental goals, highlighting the benefits provided by CCEs over MFGs solutions.
Finally, building on \cite{campi2023LQ}, \cite{cannerozzi2024cooperation} considers CCEs in ergodic MFGs of singular controls, provides constructive existence results, as well as approximation results and comparison with MFG solutions. Remarkably, it is shown there that CCEs may exist even when MFG solutions do not.

\smallskip
Notions of equilibria other than Nash have already been considered in literature, both for games with finitely many and infinitely many players.
We cite the principal-agent problems and Stackelberg equilibria both with finite and infinitely many players (see, e.g., the book \cite{cvitanic2013contract} or the papers \cite{aurell_carmona2022epidemics,bergault2024stackelberg,djete2023stackelberg,elie2019many} and the references therein).
Correlated and coarse correlated equilibria are essentially different from the aforementioned problems, for one main reason: Differently from these problems, in CEs and CCEs, the mediator is not in principle an optimizer.
The mediator may not even be an actual person or agency, although we have opted for the interpretation of a mediator recommending strategies: it may be just the result of the learning procedure of the players (see again \cite{hart_mascolell_regret_based}), or the result of a pre-play communication protocol among the players (see, among others, \cite{benporath1998correlation}).

\smallskip
In this work, we consider MFGs with general dependence on the flow of measures, and we deal both with existence of CCEs in the MFG and the relation between CCEs in the $N$-player game and in the MFG.
In the $N$-player game, the state dynamics follow stochastic differential equations (SDEs), driven by independent standard Brownian motions representing additive idiosyncratic noise, and the interaction between players is given through the empirical distribution of their states, which appears both in the drift of the SDEs and in each player's payoff functional.
Players' strategies are assumed to be open loop, i.e., the correlation device would recommend the players to use strategies adapted to the filtration generated by noises and initial data. More precisely, the correlation device, or recommendation to the $N$ players, is modeled as a random variable taking values in the set of open loop strategy profiles; we require it to be independent of the random shocks and the initial states which determine players' states' evolution.
We deal very carefully with the measurability properties the recommendation has to fulfill so that the players' states are well-defined and the recommended strategies are implementable by the players.
In the mean field limit, the notion of coarse correlated solution we present corresponds to a pair given by a recommendation with values in the set of open loop strategies for the representative player and a random flow of measures fulfilling the following two properties:
\begin{itemize}[label= --]
    \item \textit{Optimality}: the representative player has no incentive to deviate from the recommended strategy \emph{before} the extraction has happened.
    \item \textit{Consistency}: the flow of measures at any time $t$ equals the marginal law of the representative player's state conditioned on the $\sigma$-algebra generated by the whole flow of measures up to terminal time.
\end{itemize}
Through the study of a simple example, our notion of coarse correlated solution to the MFG is compared to the more usual notion of MFG solution, (as defined, e.g., in \cite{car_del_probabilistic}) and the notion of weak MFG solution of \cite{lacker2016}.
Our main contributions are as follows:
\begin{itemize}[label= --]

    \item We justify our notion of coarse correlated solution for the MFG by showing that any coarse correlated solution for the MFG induces a sequence of approximate CCEs in the $N$-player game, with vanishing error as $N$ goes to infinity.

    \item Under an additional convexity assumption, we prove the existence of a coarse correlated solution for the mean field game.

\end{itemize}
Both results will be established using a genuinely probabilistic approach.
As for the approximation result, we use the limit flow of measures to act as a correlation device between players' strategies in the $N$-player game, in the same spirit of \cite{bonesiniCE,campifischer2021}.
Then, the proof of the inverse convergence relies on propagation of chaos arguments, which are in part reminiscent of \cite{car_del_probabilistic,carmona_lacker2015weak}.
On the other hand, to prove existence, we associate a zero-sum game to the search of a coarse correlated solution for the MFG, inspired by the works of Hart and Schmeilder \cite{hart_schmeidler} (for static games), Nowak \cite{nowak1992correlated,Nowak1993} (for continuous time dynamic games) and Bonesini \cite[Appendix 1.B]{bonesini_thesis} (for mean field games with discrete time and finite states and actions, in the setting of \cite{campifischer2021}), which require us to apply a minimax theorem.
To do so, compactness arguments are exploited, adapting some of the techniques used in Lacker's works \cite{lacker2015martingale,lacker2020convergence}.

\smallskip
Referring to previous works \cite{bonesiniCE,campifischer2021}, the reason for considering CCEs instead of CEs is both theoretical and practical.
First, they are more general than CEs, thus than Nash equilibria, both in mixed and pure strategies.
Secondly, the fact that if a player deviates, she is not informed of the outcome of the moderator's lottery makes the treatment of CCEs easier than the one of CEs, due to the fact that in CCEs deviations are independent of the outcome of moderator's lottery.
On the contrary, in CEs, every player is informed of the outcome of moderator's lottery, and then decides whether to play accordingly or not.
Thus, in CEs, deviations would depend on that outcome, giving rise to delicate measurability issues, which for the moment we do not know how to handle properly.
Moreover, while in \cite{bonesiniCE,campifischer2021} restricted closed loop strategies were considered, here we consider for simplicity stochastic open loop strategies, since we deal with the more challenging problem of continuous time, actions and states.

\smallskip
The rest of the paper is organised as follows: in Section \ref{sezione_notations_assumptions} we collect some notations and state the main assumptions, which will be in force throughout the whole paper.
In Section \ref{sezione_formulazione_N_giocatori}, we define the $N$-player game and present a notion of CCE; coarse correlated solutions for the MFG are defined in Section \ref{sezione_formulazione_mfg}.
The approximation and existence results are presented in Sections \ref{sec:approximation} and \ref{sezione_existence}, respectively.
In Section \ref{sezione_esempio}, we consider a simple class of games, already discussed in the literature (see \cite{bardi_fischer2019,lacker2016,lacker2020convergence}): we show that it has coarse correlated solutions which are different from the classical MFG solution, and we compare them with the notions of solution of the previous literature.
Finally, some auxiliary technical results are gathered in the Appendix.

\section{Notations and standing assumptions}\label{sezione_notations_assumptions}

Here, we collect the most frequent notations that occur in this work and state the assumptions.

For a metric space $(E,d_E)$, we denote by $\boreliani{E}$ the Borel $\sigma$-algebra generated by the topology of $E$. When the context allows, we will drop the dependence upon $E$, and just denote it by $\mathcal{B}$.
We denote by $\cbounded{E}$ the set of continuous bounded function $f:E \to \R$.

When given infinitely many product spaces $(E^i,\mathcal{E}^i,P^i)_{i \geq 1}$, we define the product space $(E,\mathcal{E},P) = \otimes_{i = 1}^\infty (\E^i,\mathcal{E}^i,P^i)$ by setting
\[
E = \bigtimes_{i=1}^\infty E^i, \quad \mathcal{E} = \bigotimes_{i=1}^\infty \mathcal{E}^i, \quad P = \bigotimes_{i=1}^\infty P^i, 
\]
where $\mathcal{E}$ is the $\sigma$-algebra generated by the cylinders $C = A_1 \times \cdots \times A_n \times_{i=n+1}^\infty E^i$, for $n \geq 1$ and $A_i \in \mathcal{E}_i$, and $P$ is the probability measure defined on the cylinders by setting $P(C) = \prod_{i=1}^n P_i(A_i)$.

Given a measure space $(M,\mathcal{M},\mu)$ and $B \subseteq \R$, we denote by $L^2(M,\mathcal{M},\mu;B)$ the set of measurable functions $f:M \to B$ so that $\int_M \vert f(m) \vert ^2 \mu(de) < \infty$, and we define its $L^2$-norm by $\Vert f \Vert_{L^2} = (\int_M \vert f(m) \vert ^2 \mu(dm))^\frac{1}{2}$.
As usual, we identify functions $f_1$ and $f_2$ which are equal $\mu$-a.s.

We will denote by $\mathcal{P}(E)$ the set of probability measures on $(E,\boreliani{E})$.
For $p \geq 1$, we denote by $\mathcal{P}^p(E)$ the set of probability measures $m \in \mathcal{P}(E)$ so that, for some point $x_0 \in E$, and thus for any, it holds $\int_E d^p_E(x,x_0)m(dx) < \infty$.
Let $\mathcal{W}_{p,E}(m_1,m_2)$ denote the $p$-Wasserstein distance on $\mathcal{P}^p(E)$, defined as
\begin{equation*}
    \pwassmetric{p}{E}{p}(m_1,m_2)=\inf\graffe{ \int_{E \times E} d^p_E(x,y)\pi(dx,dy):\; \pi \in \mathcal{P}(E\times E), \text{ $\pi$ has marginals $m_1$, $m_2$}}.
\end{equation*}
Any time we will be given two metric spaces $(E,d_E)$ and $(E',d_{E'})$, we will regard $E \times E'$ as a metric space itself, with the distance $d((e,f),(e',f'))=d_E(e,f) + d_{E'}(e',f')$.
The $p$-Wasserstein distance on $\mathcal{P}^{p}(E \times E')$ will always be meant with respect to such distance on $E \times E'$.
For $T>0$ fixed, we denote by $\contrd$ the set of continuous functions from $[0,T]$ in $\R^d$, $d \in \N$, i.e. $\contrd=\conttraj{\R^d}$. We endow $\contrd$ with the norm $\lVert x \rVert_{\contrd}=\sup_{ s \in [0,T]}\vert x_s \vert $.
Occasionally, we will use the semi-norm $\lVert x \rVert_{t,\contrd}=\sup_{s \in [0,t]}\vert x_s \vert$, for $x \in \contrd$.
We will denote as $\wienermeasure \in \mathcal{P}(\contrd)$ the law of a standard $d$-dimensional Brownian motion, and by $\contpdue$ the set of continuous functions from $[0,T]$ in $\mathcal{P}^2(\R^d)$, i.e. $\contpdue=\mathcal{C}([0,T];\mathcal{P}^2(\R^d))$, where $\mathcal{P}^2(\R^d)$ is endowed with the $2$-Wasserstein distance.
We endow $\contpdue$ with the supremum distance $\sup_{t \in [0,T]}\pwassmetric{2}{\mathcal{P}^2(\R^d)}{}(m^1_t,m^2_t)$, for any $m^1=(m^1_t)_{t \in [0,T]}$ and $m^2=(m^1_t)_{t \in [0,T]}$ in $\contpdue$.

When given a filtered probability space $(\Omega,\F,(\mathcal{G}_t)_{t},\prob)$, we regard as the $\prob$-aug\-men\-ta\-tion of the filtration $(\mathcal{G}_t)_{t}$ the filtration $\mathbb{F}=(\F_t)_{t}$, where $\F_t=\cap_{\eps > 0} \sigma(\mathcal{G}_{t + \eps},\mathcal{N})$ and $\mathcal{N}$ stands for the $\prob$-null sets of $\Omega$. Such a filtration satisfies the usual assumptions.

\smallskip
We end this section by stating our standing assumptions on the state dynamics and on the costs of the players in both the $N$-player game and the limit game.
We are given a finite time horizon $T >0$, a control actions space $A$, an initial state distribution $\nu \in \probmeasures{\R^d}$, and the following functions:
\begin{equation*}
\begin{aligned}
    (b,f):[0,T]\times\R^d\times\mathcal{P}^2(\R^d)\times A\to\R^d \times \R, \\
    g:\R^d\times\mathcal{P}^2(\R^d)\to \R,
\end{aligned}
\end{equation*}
which will be referred to, respectively, as the drift function, the running cost and the terminal cost.
The following Assumptions \ref{standing_assumptions} will be in force throughout the whole manuscript.

\begin{customassumptions}{\textbf{A}}\label{standing_assumptions}
\begin{enumerate}[label=\normalfont(A.\arabic*)]
    \item[]
    \item $A\subseteq \R^l$, for some $l \geq 1$, is a compact set.
    \item $\nu \in \mathcal{P}^{\overline{p}}(\R^d)$, for some $\overline{p} > 4$.
    \item The functions $b$, $f$ and $g$ are jointly measurable in $(t,x,m,a)$.
    \item $b(t,x,m,a)$ is Lipschitz in $a \in A$, $m \in \mathcal{P}^2(\R^d)$ and $x \in \R^d$, uniformly in $t$:
    \begin{equation*}
        \vert b(t,x,m,a)-b(t,x',m',a') \vert \leq L \tonde{\vert a-a' \vert  + \vert x-x' \vert  + \pwassmetric{2}{\R^d}{}(m,m')}
    \end{equation*}
    for every $t \in [0,T]$, $(x,m,a)$ and $(x',m',a')$ in $\R^d \times \mathcal{P}^2(\R^d)\times A$.
    \item The functions $[0,T] \ni t \mapsto (b,f)(t,0,\delta_0,a_0)$ are bounded, for some $a_0 \in A$ and $\delta_0 \in \mathcal{P}^2(\R^d)$.
    \item $f$ and $g$ are locally Lipschitz in $(x,m,a)$ for every fixed $t \in [0,T]$ with at most quadratic growth, i.e., there exists a positive constant $L>0$ so that
    \begin{equation*}
    \begin{aligned}
        \big\vert (f,g) & (t,x,m,a)-(f,g)(t,x',m',a') \big\vert  \\
        \leq & \, L\left( 1 + \abs{x}+\abs{x'} + \tonde{\int_{\R^d} \abs{y}^2 m(dy)}^\frac{1}{2} + \tonde{\int_{\R^d} \abs{y}^2 m'(dy)}^\frac{1}{2} +\abs{a}+\abs{a'}\right) \\
        & \cdot \tonde{\vert x-x' \vert +\pwassmetric{2}{\R^d}{}(m,m')+\vert a-a' \vert },
    \end{aligned}
    \end{equation*}
    for every $t \in [0,T]$, $(x,m,a)$ and $(x',m',a')$ in $\R^d \times \mathcal{P}^2(\R^d)\times A$.
\end{enumerate}
\end{customassumptions}
Assumptions (A.3-6) are fairly standard in MFG literature, see e.g. \cite{carmona_del_fbsdes}. In particular, they yield existence and uniqueness of strong solutions to the equations governing the dynamics, as well as finiteness and sufficient regularity of the cost functionals. Assumptions (A.1-2) enable us to apply compactness arguments in the proof of our main results, and they are particularly relevant for the existence in result in Section \ref{sezione_existence}.

\section{Formulation of the $N$-player game}\label{sezione_formulazione_N_giocatori}
Consider the following canonical space
\begin{equation}\label{canonical_setup}
\begin{aligned}
    \Omega^1=\bigtimes_{1}^\infty \left( \R^d \times \contrd \right), \quad \F^1=\bigotimes_{1}^\infty \left(  \boreliani{\R^d} \otimes \boreliani{\contrd} \right), \quad \prob^1=\bigotimes_{1}^\infty \left( \nu \otimes \wienermeasure \right).
\end{aligned}
\end{equation}
We define a sequence of random variables $(\xi^i)_{i \geq 1}$ and of Brownian motions $(W^i)_{i \geq 1}$, by taking the projections:
\begin{equation}\label{canonical_browian_motions}
    \xi^i(\omega_1)=\xi^i((x^j,w^j)_{j \geq 1})=x^i, \quad W^i_t(\omega_1)=W^i_t((x^j,w^j)_{j \geq 1})=w^i_t, \; t \in [0,T].
\end{equation}
By definition of $\prob^1$, $(\xi^i)_{i\geq 1}$ and $(W^i)_{i \geq 1}$ are mutually independent, $(\xi^i)_{i\geq 1}$ are independent and identically distributed with law $\nu \in \mathcal{P}^{\overline{p}}(\R^d)$ and $(W^i)_{i \geq 1}$ are independent $d$-dimensional standard Brownian motions.

\medskip
Let $N\in\N$, $N\geq 2$, be the number of players.
We define the filtration $\mathbb{F}^{1,N}$ as the $\prob^1$-augmentation of the filtration generated by the first $N$ random variables $(\xi^i)_{i=1}^N$ and Brownian motions $(W^i)_{i=1}^N$.
Therefore, for the $N$-player game, we work on the space
\begin{equation}\label{finite_players:canonical_space}
    (\Omega^1,\F^1,\mathbb{F}^{1,N},\prob^1).
\end{equation}
We stress that, for every $N \geq 2$, we keep the probability space $(\Omega^1,\F^1,\prob^1)$ fixed while the filtration $\mathbb{F}^{1,N}$ varies.

\medskip
Consider the set $\A_N$ of $\mathbb{F}^{1,N}$-progressively measurable processes taking values in $A$:
\begin{equation}\label{finite_players:strategie_open_loop}
    \A_N = \insieme{\alpha:[0,T]\times\Omega^1\to A \; \Big\vert \;  \text{$\alpha$ is $\mathbb{F}^{1,N}$-progressively measurable } }.
\end{equation}
Provided that we identify processes which are equal $Leb_{[0,T]}\otimes\prob^1$-a.e., we can regard $\A_N$ as
\begin{equation*}
    \A_N=L^2 \tonde{[0,T]\times\Omega^1,\mathcal{R}^{1,N},Leb_{[0,T]} \otimes \prob^1;A},
    \end{equation*}
where $\mathcal{R}^{1,N}$ stands for the progressive $\sigma$-algebra on $[0,T]\times\Omega^1$, using the filtration $\mathbb{F}^{1,N}$.
We call any element $\alpha\in\A_N$ an open loop strategy for the $N$-player game.
We regard a vector $(\alpha^1,\dots,\alpha^N) \in \A_N^N=\bigtimes_1^N \A_N$ as an open loop strategy profile for the $N$ players, which will be
occasionally denoted by $\bm{\alpha}$.
We endow such a space $\A_N$ with the norm
\begin{equation}\label{finite_players:semi_norm}
    \norm{\alpha}_{L^2}=\E^{\prob^1}\quadre{\int_0^T \abs{\alpha_t}^2 dt}^\frac{1}{2}
\end{equation}
and consider the Borel $\sigma$-algebra $\boreliani{\A_N}$ associated to that.
We observe that, since $([0,T]\times\Omega^1,\boreliani{[0,T]\times\Omega^1})$ is Polish and $A$ is closed, $\A_N$ is a separable Banach space.
In the following, we will make no distinction between an $\mathbb{F}^{1,N}$-progressively measurable process $\alpha$ and any other process $\alpha'$ which is equal to it $Leb_{[0,T]}\otimes\prob^1$-almost everywhere.

\begin{definition}[Recommendation profile]
We call \emph{recommendation profile} to the $N$ players a pair $((\Omega^0,\F^{0-},\prob^0),\Lambda)$ so that the following holds:
\begin{enumerate}
    
    \item $(\Omega^0,\mathcal{F}^{0-}$, $\prob^0)$ is a complete probability space; $\Omega^0$ is a Polish space and $\F^{0-}$ is its corresponding Borel $\sigma$-algebra.

    \item $\Lambda =(\Lambda^1,\dots,\Lambda^N)$ is a random vector with values in $\A_N^N$:
    \begin{equation}\label{raccomandazione}
    \begin{aligned}
    \Lambda: (\Omega^0,\mathcal{F}^{0-},\prob^0) & \longrightarrow (\A_N^N,\mathcal{B}_{\A_N^N}) \\
    \omega_0 & \longmapsto \Lambda(\omega_0)=(\alpha^1,\dots,\alpha^N):[0,T]\times\Omega^1\to A^N.
    \end{aligned}
    \end{equation}
\end{enumerate}
\end{definition}
We interpret the recommendation profile as follows: A correlation device or a mediator runs a lottery over open loop strategy profiles according to some publicly known distribution $\prob^0$ and communicates privately to each player a strategy according to the selected profile.
The extraction of the strategy profile happens before the game starts and it is independent of the idiosyncratic shocks that determine the random evolution of players' states.

\smallskip
For a given recommendation profile $((\Omega^0,\F^{0-},\prob^0),\Lambda)$, we build a probability space large enough to support both the moderator's recommendation profile $\Lambda$ and the random shocks and initial data $(\xi^j,W^j)_{j=1}^N$, in such a way that they are independent.
On this new probability space, we then consider the strategy profile associated to the recommendation $\Lambda$, that is, the strategies that the committing players will use.
These strategies will depend both on the realisation of moderator's lottery $\Lambda^i(\omega_0)$ and the initial data and idiosyncratic noise that determine the evolution of players' states.
Let $(\Omega,\mathcal{F},\prob)$ be defined by
\begin{equation}\label{finite_players:condizione_ammissibilita}
    (\Omega,\mathcal{F},\prob)= (\Omega^0 \times \Omega^1,\mathcal{F}^{0-}\otimes\F^1,\prob^0\otimes\prob^1).
\end{equation}
We complete the $\sigma$-algebra $\F$ with the $\prob$-null sets and endow the product probability space with the $\prob$-augmentation of the filtration
\begin{equation*}
    \mathbb{F}= \F^{0-} \otimes \mathbb{F}^{1,N} = (\F^{0-}\otimes\F^{1,N}_t)_{t\in[0,T]}.
\end{equation*}
On the filtered probability space $(\Omega,\F,\mathbb{F},\prob)$ we would like to consider the strategy profile associated to a recommendation profile $\Lambda$, by setting
\begin{equation}\label{controllo_indotto}
    \lambda^i_t(\omega)= \lambda^i_t(\omega_0,\omega_1)=\Lambda^i(\omega_0)_t(\omega_1), \quad i=1,\dots,N.
\end{equation}
A priori, the process $\lambda$ constructed may not be progressively measurable, for instance when a recommendation profile $\Lambda$ takes uncountably many values.
The essential reason is that we cannot deduce the measurability of a set in the product $\sigma$-algebra from the measurability of its sections, as shown, e.g., in \cite[p.\,5]{stoyanov}.
For this reason, we have the following admissibility definition:
\begin{definition}[Admissible recommendation profile]\label{def_correlated_strategy}
A recommendation profile \linebreak $((\Omega^0,\F^{0-},\prob^0),\Lambda)$ is \emph{admissible} if there exists a process $\lambda=(\lambda^1_t,\dots,\lambda^N_t)_{t\in[0,T]}$ with values in $A^N$, defined on the product space $(\Omega,\F,\prob)$ and $\mathbb{F}$-progressively measurable, so that, for every $i=1,\dots,N$, it holds
\begin{equation}\label{finite_players:uguaglianza_ammissibilita}
            \norm{(\lambda^i_t(\omega_0,\cdot))_{t \in [0,T]}-\Lambda^i(\omega_0)}_{L^2([0,T] \times \Omega^1)}=0, \quad \prob^0\text{-a.s.,} \: i=1,\dots,N.
\end{equation}
Any such process $\lambda$ will be called \emph{strategy profile associated to the admissible recommendation profile} $((\Omega^0,\F^{0-},\prob^0),\Lambda)$.
\end{definition}
Observe that the strategy profile $\lambda$ is the result of both moderator's recommendation and the noises, and it will appear in the dynamics and the cost of the committing players. We remark that, by Proposition \ref{esempi:unicita_strategia_associata}, given any admissible recommendation to the $N$ players $((\Omega^0,\F^{0-},\prob^0),\Lambda)$, the strategy profile $\lambda$ associated to it is unique $Leb_{[0,T]}\otimes\prob$-almost everywhere. We give some examples of admissible recommendations in Example \ref{mf:example:admissible_recommendations} in the following Section \ref{sezione_formulazione_mfg}.

\begin{remark}\label{finite_players:remark_estensioni}
As usual, we can extend random variables defined on $\Omega^1$ to random variables defined on $\Omega$. Indeed, suppose $X:(\Omega^1,\mathcal{F}^1)\to (E,\mathcal{E})$ is a random variable with values in some measurable space $(E,\mathcal{E})$. We can then regard  $X$ as defined on the space $(\Omega,\mathcal{F})$ via the identification $\Tilde{X}(\omega_0,\omega_1)=X(\omega_1)$, and analogously for $\Omega^0$.
In this sense, via the identification $(\Tilde{\xi}^i,\Tilde{W}^i)(\omega_0,\omega_1)=(\xi^i,W^i)(\omega_1)$ for every $i=1,\dots,N$, we can regard the Brownian motions and initial data as defined on $\Omega$; we observe that $(W^i)_{i=1}^N$ are independent standard Brownian motions with respect to the filtration $\mathbb{F}$ as well.
Moreover, we can identify each process $\alpha\in\A_N$, which is defined on $\Omega^1$, with a process $\Tilde{\alpha}$ defined on $\Omega$ via the identification $\Tilde{\alpha}(\omega_0,\omega_1)=\alpha(\omega_1)$.
Such a process is progressively measurable with respect to the filtration $\mathbb{F}$ and independent of $\F^{0-}$.
\end{remark}
Observe that, by construction, we have $\F^{0-} \subseteq \F_t$ for every $t \in [0,T]$, and $\F^{0-}$ is independent of the filtration of noises $\mathbb{F}^1$.
This models the fact that the extraction of the strategy profile happens before the game starts and is independent of the idiosyncratic shocks that determine the random evolution of players' states.
We stress that, by definition, the realization $\Lambda^i(\omega_0)$ is an $\mathbb{F}^{1,N}$-progressively measurable process in $\A_N$, for any scenario $\omega_0 \in \Omega^0$ and $i=1,\dots,N$.
Observe that, even though $\Lambda$ and $(\xi^j,W^j)_{j=1}^N$ are independent, the strategy profile associated to the recommendation $\lambda=(\lambda_t)_{t \in [0,T]}$ is in general not independent of either of them, since it is the result of both the recommendation profile and the random shocks and initial data.

\medskip
Let $((\Omega^0,\F^{0-},\prob^0),\Lambda)$ be an admissible recommendation profile.
On the probability space $(\Omega,\F,\mathbb{F},\prob)$ defined in \eqref{finite_players:condizione_ammissibilita}, we assign players state dynamics and define the cost functionals.
If all players follow the recommendation $\Lambda$, players' state dynamics are given by the following system of stochastic differential equations:
\begin{equation}\label{dinamiche_no_deviazione}
    \begin{cases}
    dX^j_t = b(t,X^j_t,\mu^N_t,\lambda^j_t)dt +dW^j_t, \qquad 0\leq t \leq T, \\
    X^j_0 = \xi^j,
    \end{cases}
\end{equation}
for every $j\in\insieme{1,\dots,N}$, where $\mu^N_t$ is the empirical measure of the state processes of all players at time $t$:
\begin{equation}\label{misura_empirica}
    \mu^N_t=\frac{1}{N}\sum_{\substack{j=1}}^N \delta_{X^j_t}.
\end{equation}

Suppose player $i$ deviates, while the other players follow the recommendations they receive from the mediator.
The deviating player will pick instead an open loop strategy $\beta\in\A_N$.
In other words, at every time $t$ and for every scenario $\omega$, player $i$ plays the action $\Tilde{\beta}_t(\omega)=\beta_t(\omega_1)$ instead of playing the recommended action $\lambda^i_t(\omega)=\Lambda^i(\omega_0)_t(\omega_1)$.
Then, players' state dynamics are given by the following system of stochastic differential equations:
\begin{equation}\label{dinamiche_deviazione}
    \begin{cases}
    dX^j_t = b(t,X^j_t,\mu^N_t,\lambda^{j}_t)dt +dW^j_t, \qquad   0\leq t \leq T, \quad X^j_0=\xi^j, \quad j\neq i \\
    dX^i_t = b(t,X^i_t,\mu^N_t,\beta_t)dt +dW^i_t, \qquad 0\leq t \leq T, \quad X^i_0=\xi^i, \end{cases}
\end{equation}
where $\mu^N_t$ is defined as in \eqref{misura_empirica}. Assumptions \ref{standing_assumptions} ensure that there always exists an $\mathbb{F}$-adapted continuous solution to both equations \eqref{dinamiche_no_deviazione} and \eqref{dinamiche_deviazione} so that
\[
\E[\sup_{t \in [0,T]}\max_{1 \leq j \leq N}\vert X^j_t\vert^2] < \infty.
\]
Moreover, pathwise uniqueness holds so that, by Theorem \ref{teorema_di_unicita_legge}, uniqueness in law holds as well.

\begin{remark}\label{finite_players:remark_deviazioni}
We notice that there is an asymmetry between the information available to the mediator and the deviating player.
If a player deviates, she will use an open loop strategy $\beta \in \A_N$, therefore the information available to the deviating player is just given by the smaller Brownian filtration $\mathbb{F}^{1,N}$.
In particular, she will use a strategy which is independent of $\F^{0-}$, thus of the admissible recommendation profile $\Lambda$.
This models the fact that the players must decide whether to follow or not the recommendation without knowing the outcome of the mediator's lottery.
\end{remark}

As for the cost functional, let $((\Omega^0,\F^{0-},\prob^0),\Lambda)$ be an admissible recommendation profile.
If all players follow the recommendation, then the cost functional of player $i=1,\dots,N$ is given by
\begin{equation*}
    \J^N_i(\Lambda) =\E\quadre{\int_0^T f(t,X^i_t,\mu^N_t,\lambda^i_t)dt + g(X^i_T,\mu^N_T)},
\end{equation*}
with dynamics given by \eqref{dinamiche_no_deviazione}.
If instead player $i$ does not play according to the recommendation $\Lambda^i$ and plays a different strategy $\beta \in \A_N$, while the other players stick to the recommendation profile $\Lambda^{-i}=(\Lambda^1,\dots,\Lambda^{i-1},\Lambda^{i+1},\dots,\Lambda^N)$, we define the cost functional of player $i$ as
\begin{equation*}
\begin{aligned}
    \J^N_i(\Lambda^{-i},\beta) & =\E\quadre{\int_0^T f(t,X^i_t,\mu^N_t,\beta_t)dt + g(X^i_T,\mu^N_T)},
\end{aligned}
\end{equation*}
where the dynamics are given by \eqref{dinamiche_deviazione}.
We stress that the expectation in the cost functional is taken with respect to the product probability measure $\prob = \prob^0 \otimes \prob^1$, although we omit this dependence for conciseness.
Finally, we give the notion of $\eps$-coarse correlated equilibrium:
\begin{definition}[$\eps$-coarse correlated equilibrium]\label{def_CCE}
Let $\eps\geq0$. An admissible recommendation profile $((\Omega^0,\F^{0-},\prob^0),\Lambda)$ is an \emph{$\eps$-coarse correlated equilibrium} for the $N$-player game ($\eps$-CCE) if
\begin{equation}\label{def_CCE:optimality}
    \J^N_i(\Lambda)\leq \J^N_i(\Lambda^{-i},\beta) + \eps
\end{equation}
for all open loop strategies $\beta \in \A_N$ and all players $i=1,\dots,N$.
We call an admissible recommendation profile $((\Omega^0,\F^{0-},\prob^0),\Lambda)$ a \emph{coarse correlated equilibrium} for the $N$-player game if it is an $\eps$-coarse correlated equilibrium with $\eps=0$.
\end{definition}

Notice that even if player $i$ deviates, she can actually compute her cost functional $\J^N_i(\Lambda^{-i},\beta)$ by only knowing the joint law of the admissible recommendation profile $\Lambda$, for any open loop strategy $\beta$. Indeed, it is possible to express the equilibrium property \eqref{def_CCE:optimality} in terms of the law of the recommendation profile $\Lambda$ on the right-hand side, and in terms of the marginal law of $\Lambda^{-i}$ on the left-hand side.

\smallskip
The usual notion of Nash equilibrium in open loop strategies is consistent with the definition of coarse correlated equilibrium: Suppose we are given an $\eps$-Nash equilibrium $(\alpha^1,\dots,\alpha^N)$ in open loop strategies.
We choose $(\Omega^0,\F^{0-},\prob^0)$ as the trivial probability space and $\Lambda$ as constant and equal to $(\alpha^1,\dots,\alpha^N)$.
It is then straightforward to see that the triple $((\Omega^0,\F^{0-},\prob^0),\Lambda)$ is an $\eps$-CCE according to Definition \ref{def_CCE}.

Observe that a Nash equilibrium in open loop strategies $(\alpha^1,\dots,\alpha^N)$ is progressively measurable with respect to the filtration $\mathbb{F}^{1,N}$, while a strategy profile $\lambda$ associated to an admissible recommendation $\Lambda$ contains the information carried by $\Lambda$ itself, which is the information the mediator uses to randomize players' strategies.
Moreover, when dealing with Nash equilibria, the deviating player has access to the same information as the other players, since they all use $\mathbb{F}^{1,N}$-progressively measurable strategies.
On the contrary, CCEs present a certain asymmetry between the information available to the committing and the deviating players, as pointed out in Remark \ref{finite_players:remark_deviazioni}.

\begin{remark}[Role of the probability space $(\Omega^0,\mathcal{F}^{0-},\prob^0)$]
According to Definition \ref{def_correlated_strategy}, the probability space $(\Omega^0,\mathcal{F}^{0-},\prob^0)$ is part of the definition of admissible recommendation.
The natural interpretation is that the mediator chooses the auxiliary space he uses to correlate players' strategies.
Moreover, according to equations \eqref{dinamiche_no_deviazione} and \eqref{dinamiche_deviazione}, it determines the probability space on which state processes are defined.
In order to keep the notation as simple as possible, by abuse of notation, we mostly refer only to $\Lambda$ as the admissible recommendation instead of the pair $((\Omega^0,\F^{0-},\prob^0),\Lambda)$.
\end{remark}

\begin{remark}[Relationship with correlated equilibria of \cite{bonesiniCE, campifischer2021}]
It is worth to briefly compare our notion of coarse correlated equilibria with the notion of correlated equilibria of \cite{campifischer2021} and \cite{bonesiniCE}.
In these works, the authors deal with discrete time models with finite sets of individual states and control actions, and consider restricted closed-loop strategies.
Most importantly, in their framework, correlated equilibria are considered, and not coarse correlated equilibria: there, each player observes the  outcome of moderator's lottery and then decides whether to play it or not.
Thus, in that context, if player $i$ deviates, she would use a strategy $\beta$ which is not a priori independent of the outcome of moderator's lottery, but could depend on the realization $\Lambda^i(\omega_0)$, and she could use that information to choose her deviation. On the contrary, in our model, the deviating player chooses a strategy $\beta$ which is independent of the admissible recommendation profile $\Lambda$, as the deviating player has no access to the recommended strategy.
\end{remark}

\section{Formulation of the mean field game}\label{sezione_formulazione_mfg}

Consider the following canonical space
\begin{equation}\label{mf:canonical_space}
    \Omega^*=\R^d \times \contrd, \quad \F^*=\boreliani{\R^d} \otimes \boreliani{\contrd},\quad \prob^*=\nu\otimes\wienermeasure.
\end{equation}
Define $\xi$ and $W=(W_t)_{t \in [0,T]}$ as
\begin{equation}
    \xi(\omega_*)=\xi(x,w)=x, \quad W_t(\omega_*)=W_t(x,w)=w_t.
\end{equation}
By definition of $\prob^*$, $\xi$ and $W$ are independent, $\xi$ is an $\R^d$-valued random variable with law $\nu$ and $W$ is a standard Brownian motion.
Define the filtration $\mathbb{F}^*$ as the $\prob^*$-augmentation of the filtration generated by $\xi$ and $W$.

\medskip
Consider the set $\A$ of $\mathbb{F}^*$-progressively measurable processes taking values in $A$:
\begin{equation}\label{mf:strategie_open_loop}
    \A = \insieme{\alpha:[0,T]\times\Omega^*\to A \; \Big\vert \;  \text{$\alpha$ is $\mathbb{F}^*$-progressively measurable } }.
\end{equation}
Provided that we identify processes which are equal $Leb_{[0,T]}\otimes\prob^*$-a.e., we can regard $\A$ as
\begin{equation*}
    \A=L^2 \tonde{[0,T]\times\Omega^*,\mathcal{R}^*,Leb_{[0,T]} \otimes \prob^*;A},
    \end{equation*}
where $\mathcal{R}^*$ stands for the progressive $\sigma$-algebra on $[0,T]\times\Omega^*$, using the filtration $\mathbb{F}^*$.
We call any element $\alpha\in\A$ an open loop strategy for the mean field game.
We endow such a space $\A$ with the norm
\begin{equation}\label{mf:semi_norm}
    \norm{\alpha}_{L^2}=\E^{\prob^*}\quadre{\int_0^T \abs{\alpha_t}^2 dt}^\frac{1}{2}
\end{equation}
and consider the Borel $\sigma$-algebra $\boreliani{\A}$ associated to that.
We observe that, since $([0,T]\times\Omega^*,\boreliani{[0,T]\times\Omega^*})$ is Polish and $A$ is closed, $\A$ is a separable Banach space.
Finally, we will make no distinction between an $\mathbb{F}^*$-progressively measurable process $\alpha$ and any other process $\alpha'$ which is equal to it $Leb_{[0,T]}\otimes\prob^*$-almost everywhere.

As in the $N$-player game, we define the recommendation to the representative player, and then the admissibility requirement.
\begin{definition}[Recommendation for the mean field game]
We call \emph{recommendation} a pair $((\Omega^0,\F^{0-},\prob^0),\Lambda)$ where:
\begin{enumerate}
    
    \item $(\Omega^0,\mathcal{F}^{0-}$, $\prob^0)$ is a complete probability space; $\Omega^0$ is a Polish space and $\F^{0-}$ is its corresponding Borel $\sigma$-algebra.

    \item $\Lambda$ is a random variable with values in $\A$:
    \begin{equation}\label{mf:raccomandazione}
    \begin{aligned}
    \Lambda: (\Omega^0,\mathcal{F}^{0-},\prob^0) & \longrightarrow (\A,\boreliani{\A}) \\
    \omega_0 & \longmapsto \Lambda(\omega_0)=\alpha:[0,T]\times\Omega^*\to A.
    \end{aligned}
    \end{equation}
\end{enumerate}
\end{definition}
For a given recommendation $((\Omega^0,\F^{0-},\prob^0),\Lambda)$, we build a probability space large enough to support both the moderator's recommendation $\Lambda$ and the random shocks and initial datum $(\xi^*,W^*)$, in such a way that they are independent.
Let $(\Omega,\mathcal{F},\prob)$ be defined by
\begin{equation}\label{mf:condizione_ammissibilita}
    (\Omega,\F,\prob)= (\Omega^0 \times \Omega^*,\F^{0-}\otimes\F^*,\prob^0\otimes\prob^*).
\end{equation}
We complete the $\sigma$-algebra $\F$ with the $\prob$-null sets and endow the product probability space with the $\prob$-augmentation of the filtration
\begin{equation*}
    \mathbb{F}= \F^{0-} \otimes \mathbb{F}^* = (\F^{0-}\otimes\F^*_t)_{t\in[0,T]}.
\end{equation*}
On the filtered probability space $(\Omega,\F,\mathbb{F},\prob)$ we would like to consider the strategy profile associated to a recommendation profile $\Lambda$, by setting
\begin{equation}\label{mf:controllo_indotto}
    \lambda_t(\omega)= \lambda_t(\omega_0,\omega_*)=\Lambda(\omega_0)_t(\omega_*),
\end{equation}
As in Section \ref{sezione_formulazione_N_giocatori}, in order for such process $\lambda$ to be progressively measurable, we have the following admissibility definition:
\begin{definition}[Admissible recommendation for the mean field game]\label{mf:admissible_recommendation}
A recommendation for the mean field game $((\Omega^0,\F^{0-},\prob^0),\Lambda)$ is admissible if  there exists an $A$-valued process $\lambda=(\lambda_t)_{t\in[0,T]}$, defined on $(\Omega,\F,\prob)$ and $\mathbb{F}$-progressively measurable, so that it holds
\begin{equation}\label{mf:uguaglianza_ammissibilita}
        \norm{(\lambda_t(\omega_0,\cdot))_{t \in [0,T]}-\Lambda(\omega_0)}_{L^2([0,T] \times \Omega^*)}=0, \quad \prob^0\text{-a.s.}
\end{equation}
Any such process process $\lambda$ will be called \emph{strategy associated to the admissible recommendation} $((\Omega^0,\F^{0-},\prob^0),\Lambda)$. 
\end{definition}
We remark that, by Proposition \ref{esempi:unicita_strategia_associata}, given any admissible recommendation $((\Omega^0,\F^{0-},$ $\prob^0),\Lambda)$, the strategy $\lambda$ associated to it is unique $Leb_{[0,T]}\otimes\prob$-almost everywhere.

\begin{definition}[Correlated flow]\label{def_correlated_flow}
A \emph{correlated flow} is a triple $((\Omega^0,\F^{0-},\prob^0),\Lambda,\mu)$ where:
\begin{enumerate}
    \item $((\Omega^0,\F^{0-},\prob^0),\Lambda)$ is an admissible recommendation.

    \item $\mu:(\Omega^0,\F^{0-},\prob^0)\to (\contpdue,\boreliani{\contpdue})$ is a random continuous flow of measures in $\mathcal{P}^2(\R^d)$.
\end{enumerate}
\end{definition}
The same considerations as in Remark \ref{finite_players:remark_estensioni} about the extension of random variables on the product space $(\Omega,\F,\prob)$ hold for correlated flows as well.

\medskip
Let $((\Omega^0,\F^{0-},\prob^0),\Lambda,\mu)$ be a correlated flow.
On the product probability space $(\Omega,\mathcal{F},\prob)$ defined in \eqref{mf:condizione_ammissibilita}, we assign state dynamics.
If the representative player decides to play according to the admissible recommendation $\Lambda$, the dynamics is given by the following SDE:
\begin{equation}\label{dinamica_MF_no_deviazione}
    \begin{cases}
    dX_t = b(t,X_t,\mu_t,\lambda_t)dt +dW_t, \qquad  0\leq t \leq T, \\
    X_0 = \xi.
    \end{cases}
\end{equation}
If instead the representative player decides to ignore the mediator's recommendation and to use a possibly different strategy $\beta\in\A$, the dynamics is given by the following SDE:
\begin{equation}\label{dinamica_MF_deviation}
    \begin{cases}
    dX_t = b(t,X_t,\mu_t,\beta_t)dt +dW_t, \qquad 0\leq t \leq T, \\
    X_0 = \xi.
    \end{cases}
\end{equation}
By Assumptions \ref{standing_assumptions}, on any space $(\Omega,\F,\mathbb{F},\prob)$ there exists a solution to equation \eqref{dinamica_MF_no_deviazione} and pathwise uniqueness holds.
By Theorem \ref{teorema_di_unicita_legge}, uniqueness in law holds.
Analogous considerations apply to equation \eqref{dinamica_MF_deviation}.

\medskip
Let $((\Omega^0,\F^{0-},\prob^0),\Lambda,\mu)$ be a correlated flow.
The cost functionals for the representative player and the deviating player, whose state dynamics follow \eqref{dinamica_MF_no_deviazione} and \eqref{dinamica_MF_deviation}, respectively, are given by:
\begin{equation}\label{costi_mfg}
\begin{aligned}
    & \J(\Lambda,\mu)=\E\quadre{\int_0^T f(t,X_t,\mu_t,\lambda_t)dt + g(X_T,\mu_T)}, \\
    & \J(\beta,\mu)=\E\quadre{\int_0^T f(t,X_t,\mu_t,\beta_t)dt + g(X_T,\mu_T)}.
\end{aligned}
\end{equation}
As in the $N$-player game, the expectation in the cost functional is taken with respect to the product probability measure $\prob = \prob^0 \otimes \prob^*$; in particular, it depends on the mediator's randomization $\prob^0$ also when the representative player deviates. Finally, we give the definition of coarse correlated solution of the mean field game:
\begin{definition}[Coarse correlated solution]\label{def_mean_field_sol}
A correlated flow $((\Omega^0,\F^{0-},\prob^0),\Lambda,\mu)$ is a \emph{coarse correlated solution} of the mean field game if the following properties hold:
\begin{enumerate}[label=(\roman*)]
    \item Optimality: for every deviation $\beta \in \A$, it holds
    \begin{equation}\label{def_mean_field_sol:opt}
        \J(\Lambda,\mu)\leq \J(\beta,\mu).
    \end{equation}
    \item Consistency: for every time $t \in [0,T]$, $\mu_t$ is a version of the conditional law of $X_t$ given $\mu$, that is,
    \begin{equation}\label{def_mean_field_sol:cons}
        \mu_t(\cdot)=\prob(X_t \in\cdot \; \vert \; \mu ) \quad \prob\text{-a.s.} \;\; \forall t \in [0,T].
    \end{equation}
\end{enumerate}
We will refer to coarse correlated solutions of the mean field game as coarse correlated mean field solutions and mean field coarse correlated equilibria (CCE) as well.
\end{definition}

\begin{remark}[Role of $(\Omega^0,\F^{0-},\prob^0)$]
Analogously as in the $N$-player game, although the probability space $(\Omega^0,\mathcal{F}^{0-},\prob^0)$ is part of the definitions of admissible recommendation and correlated flow, when it is clear from the context we refer to $\Lambda$ and $(\Lambda,\mu)$, instead of the pair $((\Omega^0,\F^{0-},\prob^0),\Lambda)$ and the triple $((\Omega^0,\F^{0-},\prob^0),\Lambda,\mu)$, as admissible recommendation and correlated flow, respectively.
\end{remark}

As in \cite{bonesiniCE,campifischer2021}, the consistency condition \eqref{def_mean_field_sol:cons} should be read in the following way: the mediator imagines what the flow of measures will be, up to the terminal horizon $T$, before the game starts, and gives a recommendation to each player according to his idea. Since the flow of measures is expected to be stochastic as a result of the mediator's randomization only, we request it to be measurable with respect to $\F^{0-}$, and, since the randomization is performed before the game starts, we have $ \F^{0-} \subseteq \F_t$ for any $t \geq 0$. If all players commit to the mediator's lottery for generating recommendations, then the flow of measures should arise from aggregation of the individual behaviors, consistently with what imagined by the mediator. Since the generation of the recommendation is performed on the basis of the whole flow of measures, we formulate consistency condition \eqref{def_mean_field_sol:cons} with respect to conditioning on the whole flow.
Regarding the strategy of the deviating player, as in the $N$-player game, if the player deviates, she chooses her strategy on her own, without using any of the information carried by $\Lambda$ or $\mu$:
the only information she has about $\Lambda$ or $\mu$ comes from the knowledge of their joint law, which is assumed to be known by the representative player, in analogy to the $N$-player game.

\begin{remark}[Relation with MFGs with common noise]\label{mfg:rmk:common_noise}
Given the consistency condition \eqref{def_mean_field_sol:cons}, it is worth comparing coarse correlated solutions to the MFG and solutions to MFGs with common noise (see, e.g., \cite{carmona2016commonnoise,lacker2016,lacker_leflem2022} and in \cite{librone_vol2}).
In the latter, the flow of measures is stochastic due to a common noise that equally impacts the state dynamics of all players in the underlying $N$-player game.
As a consequence, the flow of measures is expected to be adapted to the filtration generated by the common noise (the so called \emph{strong solutions}); if this is not the case, compatibility conditions between the noises and the flow of measures itself are needed in order to guarantee that the flow $\mu$ picks into the future in a minimal way (the so called \emph{weak solutions}).
In the case of a coarse correlated solution to the MFG, on the other hand, the flow of measures is expected to be stochastic as a result of the mediator's randomization only, which is generated before the beginning of the game. More formally, this implies that the flow of measure is $\F^{0-}$-measurable with $\F^{0-} \subseteq \F_t$ for any $t \geq 0$. Recommendations to the representative player are given according to the mediator's idea of the whole flow, which leads to the consistency condition with conditioning with respect to the whole flow up to terminal time. In this sense, the mediator sees into the future, and consequently no compatibility condition is needed.

One might be tempted to regard the randomness driving the mediator's lottery for selecting recommendations as a common noise that affects the state dynamics only through the control. There are at least two major differences though: First, such a common noise will have no impact on the controls of a deviating player. To put it differently, only the pre-committing players' dynamics are directly affected by the mediator's lottery over strategy profiles.
Second, such a common noise would not be exogenous; instead, it is built into the correlation device used by the mediator, as represented by the auxiliary probability space $(\Omega^0,\F^{0-},\prob^0)$, and as such is part of the solution.
\end{remark}

\begin{example}[Admissible recommendations]\label{mf:example:admissible_recommendations}
Fix a complete probability space $(\Omega^0,\F^{0-},$ $\prob^0)$.
We provide some simple examples of random variables $\Lambda:(\Omega^0,\F^{0-},\prob^0) \to (\A,\boreliani{\A})$ which are admissible recommendations in the sense of Definition \ref{mf:admissible_recommendation}.

\begin{enumerate}[wide]
\item \label{mf:example:finitely_many_values}
Suppose that $\Lambda$ takes only finitely many values, say $\alpha^1,\dots,\alpha^k \in \A$, $k \geq 1$, i.e. $\prob^0(\Lambda=\alpha^i)=p_i$, with $p_i \geq 0$ for every $i=1,\dots,k$, $\sum_{i=1}^k p_i=1$.
We can easily define the associated strategy $(\lambda_t)_{t \in [0,T]}$ as
\begin{equation*}
    \lambda_t(\omega_0,\omega_*)=\sum_{i=1}^k \1_{\insieme{\Lambda=\alpha^i}}(\omega_0)\alpha^i_t(\omega_*).
\end{equation*}
We explicit the dependence upon the scenario $\omega=(\omega_0,\omega_*)$:
\begin{equation*}
    \Lambda(\omega_0)_t(\omega_*) = \sum_{i=1}^k \1_{\insieme{\alpha^i}}(\Lambda(\omega_0))\alpha^i_t(\omega_*).
\end{equation*}
By the same line of reasoning of Remark \ref{finite_players:remark_estensioni}, we have that this process is $\mathbb{F}$-progressively measurable, since the processes $\alpha^i$, $i=1,\dots,k$, are $\mathbb{F}$-progressively measurable and the $\F^{0-}$-measurable real-valued random variables $\1_{\{\alpha^1\}}(\Lambda(\omega_0))$ can be regarded as defined on the product space $\Omega^0\times\Omega^*$ and $\F^{0-}\otimes \{\emptyset, \Omega^*\}$-measurable, therefore $\mathbb{F}$-progressively measurable.
Finally, condition \eqref{mf:uguaglianza_ammissibilita} is satisfied by $\lambda$ itself.

\item Suppose $\Lambda$ takes at most countably many values.
We can define $\lambda$ as
\begin{equation*}
    \lambda_t(\omega_0,\omega_*)= \sum_{i=1}^\infty\1_{\insieme{\Lambda=\alpha^i}}(\omega_0)\alpha^i_t(\omega_*).
\end{equation*}
Set $\lambda^n_t(\omega_0,\omega_*)=\sum_{i=1}^n \1_{\{\Lambda=\alpha^i\}}(\omega_0)\alpha^i_t(\omega_*)$ and observe that, by the same argument of the previous point, $\lambda^n_t$ is an $\mathbb{F}$-progressively measurable process for each $n\geq 1$.
Furthermore, for each $(t,\omega_0,\omega_*) \in [0,T]\times\Omega^0\times\Omega^*$, the sequence $\lambda^n_t(\omega_0,\omega_*)$ is eventually constant, being $(\{\Lambda=\alpha^i\})_{i\geq 1}$ a partition of $\Omega^0$.
Therefore, the sequence $\lambda^n$ converges pointwise to $\lambda=(\lambda_t)_{t \in [0,T]}$.
Being $\lambda$ the pointwise limit of $\lambda^n$, we deduce that $\lambda$ is a progressively measurable process with values in $A$ which satisfies \eqref{mf:uguaglianza_ammissibilita}, so that $\Lambda$ is admissible.

\item Let $(\Omega^0,\F^{0-},\prob^0)$ be a complete probability space, with $\Omega^0$ Polish and $\F^{0-}$ the corresponding Borel $\sigma$-algebra, and let $(\lambda_t)_{t \in [0,T]}$ be an $A$-valued process defined on the $\prob^0\otimes\prob^*$-completion of the product space $(\Omega^0\times\Omega^*,\F^{0-}\otimes\F,\prob^0\otimes\prob^*)$ with values in $A$.
Assume that it is progressively measurable with respect to the $\prob^0\otimes\prob^*$-augmentation of the filtration $\mathbb{F}=(\F^{0-}\otimes\F^*_t)_{t \in [0,T]}$.
We can define a function $\Lambda:\Omega^0\to \A$ by setting
\begin{equation}\label{raccomandazione_indotta}
\begin{aligned}
    \Lambda(\omega_0)=\left \{ \: \begin{aligned}
        & \begin{aligned}
            (\lambda_{t}(\omega_0,\cdot))_{t \in [0,T]}  :\space [0,T] \times \Omega^* & \to A \\
        (t,\omega_*) & \to \lambda_t(\omega_0,\omega_*), 
        \end{aligned} &&  \omega^0 \in \Omega^0\setminus N, \\ 
        & a_0 && \omega_0 \in N.
    \end{aligned} \right.
\end{aligned}
\end{equation}
where $N\subset \Omega^0$ is a $\prob^0$-null set and $a_0$ is an arbitrary point in $A$.
By Lemma \ref{esempi:lemma_misurabile} in Appendix \ref{appendix_recommendations}, the pair $((\Omega^0,\F^{0-},\prob^0),\Lambda)$ is an admissible recommendation, with strategy associated to the recommendation $\Lambda$ given by the process $\lambda$ itself.
\end{enumerate}
\end{example}

\section{Approximate $N$-player coarse correlated equilibria}\label{sec:approximation}

The next result shows how to construct a sequence of approximate $N$-player coarse correlated equilibria with approximation error tending to zero as $N \to \infty$, provided we have a coarse correlated solution to the mean field game.

\begin{theorem}\label{thm_approssimazione}
Let $((\Omega^0,\F^{0-}$, $\prob^0),\Lambda^*,\mu^*)$ be a coarse correlated solution of the mean field game.
For each $N \geq 2$, there exist:
\begin{enumerate}[label=(\roman*)]
    \item an admissible recommendation to the $N$ players $((\Omega^{0,N},\F^{0-,N},\prob^{0,N}),\Lambda^N)$;

    \item a real valued $\eps_N \geq 0$, with $\eps_N \to 0$ as $N \to \infty$,
\end{enumerate}
so that $((\Omega^{0,N},\F^{0-,N},\prob^{0,N}),\Lambda^N)$ is an $\eps_N$-coarse correlated equilibrium for the $N$-player game.
\end{theorem}

The proof of Theorem \ref{thm_approssimazione} has two main steps:
First, starting from a coarse correlated solution to the MFG, we build a probability space $(\overline{\Omega},\overline{\F},\overline{\prob})$ large enough to carry any sequence $(\Lambda^i)_{i \geq 1}$ of admissible recommendations such that
\begin{enumerate}
    \item for every $i$, $\Lambda^i$ is supported on the set of open-loop strategies progressively measurable with respect to player $i$'s private noise;
    \item for every $i$, $\Lambda^i$ has the same distribution as $\Lambda^*$;
    \item for every $N \geq 2$, $(\Lambda^1,\dots,\Lambda^N)$ is exchangeable.
\end{enumerate}
Then, we define the probability space $(\Omega^{0,N},\F^{0-,N},\prob^{0,N})$ as $(\overline{\Omega},\F^{0-,N},\overline{\prob})$, where $\F^{0-,N}$ is the $\overline{\prob}$-completion of $\sigma(\Lambda^1,\dots,\Lambda^N)$.
The construction of the space $(\overline{\Omega},\overline{\F},\overline{\prob})$ strictly depends on the coarse correlated solution we need to approximate.
This is accomplished in Section \ref{sec:approximation:sezione_costruzione_raccomandazioni}.
Then, in Section \ref{sec:approximation:proof_thm}, we prove Theorem \ref{thm_approssimazione}.
In order to exploit the optimality property \eqref{def_mean_field_sol:opt} of the coarse correlated solution, great care is taken in comparing the cost functional associated to open-loop strategies in the $N$-player game with the payoff of the coarse correlated solution. 
We conclude by a propagation of chaos result, which is specific to our situation but quite standard. For completeness, the statement and the proof of such result is deferred to Section \ref{sec:propagation_of_chaos} in the Appendix.

\subsection{Construction of the admissible recommendation profiles to the $N$-player game}\label{sec:approximation:sezione_costruzione_raccomandazioni}

With respect to the probability space $(\Omega^1,\F^1,\prob^1)$ defined in \eqref{finite_players:canonical_space}, let us denote by $\mathbb{F}^{(i)}$ the $\prob^1$-augmentation of the filtration generated by $(\xi^i,W^i)$.
Let us introduce the following set of strategies:
\begin{equation}\label{approximation:strategy_sets}
\begin{aligned}
    & \A_{(i)}=\insieme{\alpha \in \A_N \space \; \vert \; \alpha \text{ is $\mathbb{F}^{(i)}$ progressively measurable}}.
\end{aligned}
\end{equation}
We stress that, by construction, for each $N \geq 2$, open loop strategies for the $N$-player game are defined on the same probability space $(\Omega^1,\F^1,\prob^1)$ and we have the inclusions $\A_N\subseteq \A_{N+1}$ and $\A_{(i)}\subseteq \A_N$ for every $i \leq N$.
Let us denote by $\rho \in \probmeasures{\contpdue}$ the distribution of $\mu^*$.
Let
\begin{equation*}
    K:\F^{0-} \times\contpdue \to [0,1]
\end{equation*}
be the regular conditional probability of $\prob^0$ given $\mu^*$, which exists and is unique since both $(\Omega^0,\F^{0-},\prob^0)$ and $(\contpdue,\mathcal{B})$ are Polish spaces.
Here and in the following, $\mathcal{B}$ stands for the Borel $\sigma$-algebra on $\contpdue$.
Let $\gamma$ denote the joint law of $(\Lambda^*,\mu^*)$ under $\prob^0$, let $\kappa$ be a version of the regular conditional probability of $\gamma$ given $\mu^*$, that is, the stochastic kernel $\kappa:\boreliani{\A}\times\contpdue \to [0,1]$ so that it holds
\begin{equation}\label{approximation:legge_soluzione_disintegrate}
    \prob^0\tonde{(\Lambda^*,\mu^*) \in C \times B}=\int_B \kappa(C,m)\rho(dm) \quad \forall C \in \boreliani{\A}, \; \forall B \in \mathcal{B}.
\end{equation}
Define the probability space $\tonde{\overline{\Omega},\overline{\F}}$ in the following way:
\begin{equation}\label{approximation:spazio_raccomandazioni}
\begin{aligned}
    & \overline{\Omega}= \tonde{ \bigtimes_{1}^\infty \Omega^0} \times \contpdue,
    && \overline{\F}= \tonde{ \bigotimes_{1}^\infty \F^{0-} } \otimes \mathcal{B},
\end{aligned}
\end{equation}
and define $\overline{\prob}$ so that, for every cylinder $R$ with basis $A_1 \times \cdots \times A_N \times B $, with $A_i \in \F^{0-}$ for every $i = 1, \dots, N$, $N \geq 2$, $B \in \mathcal{B}$, it holds
\begin{equation}\label{approximation:legge_spazio_raccomandazioni}
    \overline{\prob}\tonde{ R }=\int_{B} \prod_{i=1}^N K\tonde{A_i,m}\rho(dm).
\end{equation}
We complete the space $(\overline{\Omega},\overline{\F},\overline{\prob})$ with the $\overline{\prob}$-null sets.
Let $\overline{\omega}=((\omega_0^i)_{i \geq 1},m)$ denote a scenario in $\overline{\Omega}$.
Let $\mu: (\overline{\Omega},\overline{\F},\overline{\prob}) \to (\contpdue,\mathcal{B})$ be the projection on $\contpdue$, that is
\begin{equation}\label{approximation:approx_misura}
    \mu\tonde{\overline{\omega}}=m.
\end{equation}

\begin{lemma}\label{approximation:proprieta_raccomandazioni_N_player}
There exists a sequence of recommendations $(\Lambda^i)_{i \geq 1}$ from $(\overline{\Omega},\overline{\F},\overline{\prob})$ to $\bigtimes_{N=1} ^\infty \A_N$ so that, for each $i \geq 1$, the following holds:
\begin{enumerate}[label=(\alph*)]
    \item \label{approximation:proprieta_raccomandazioni_N_player:ammissibilita} $\Lambda^i$ is an admissible recommendation, and it takes values in $\A_{(i)}$.
    \item \label{approximation:proprieta_raccomandazioni_N_player:legge_congiunta_profilo} The joint law of $(\Lambda^{1},\dots,\Lambda^{N})$ under $\overline{\prob}$ is supported on $\bigtimes_{i=1}^N\A_{(i)}\subseteq \A_N^N$ and it is given by
    \begin{equation}\label{legge_eps_equilibrio}
        \gamma_N\tonde{d\alpha^1,\dots,d\alpha^N}= \int_{\contpdue} \bigotimes_{i=1}^N \kappa\tonde{d\alpha^i,m}\rho(dm).
     \end{equation}
     As a consequence, for every $i \geq 1$, $(\Lambda^i,\mu)$ has the same distribution as $(\Lambda^*,\mu^*)$ and  $(\Lambda^i)_{i \geq 1}$ are conditionally independent given $\mu$.
\end{enumerate}
\end{lemma}

\begin{proof}
Recall from \eqref{finite_players:canonical_space} and \eqref{mf:canonical_space} the definitions of the spaces $(\Omega^1,\F^1,\prob^1)$ and $(\Omega^*,\F^*, \linebreak \prob^*)$.
Observe that, up to completion, it holds
\begin{equation}\label{approximation:spazio_rumori}
    \tonde{\Omega^1,\F^1,\prob^1}= \bigotimes_{1}^\infty \tonde{\Omega^*,\F^*,\prob^*}
\end{equation}
so that a scenario $\omega_1 \in \Omega^1$ can be written as $\omega_1=(\omega^j_*)_{j \geq 1}$.
Moreover, by definition of $(\xi^j,W^j)_{j \geq 1}$ in \eqref{canonical_browian_motions}, for every $i \geq 1$ it holds
\begin{equation*}
    (\xi^i,W^i)(\omega_1)=(x^i,w^i)=(\xi^*,W^*)(\omega^i_*),
\end{equation*}
so that $(\xi^i,W^i)_{i\geq 1}$ can be seen as a sequence of independent copies of $(\xi^*,W^*)$.
Define the filtered probability space $(\Omega,\F,\mathbb{F},\prob)$ as in  \eqref{finite_players:condizione_ammissibilita}.
Let $\lambda^i=(\lambda^i_t)_{t \in [0,T]}$, $i \geq 1$, be independent copies of $\lambda^*=(\lambda^*_t)_{t \in [0,T]}$, the strategy associated to the admissible recommendation $\Lambda^*$ according to \eqref{mf:controllo_indotto}, so that
\begin{equation*}
    \lambda^i_t(\overline{\omega},\omega_1)= \lambda^i_t((\omega^j_0)_{j \geq 1},m,(\omega^j_*)_{j \geq 1})=\lambda^*_t\tonde{\omega_0^i,\omega^i_*}.
\end{equation*}
For every $i$, $\lambda^i$ is $\mathbb{F}$-progressively measurable: indeed, since by definition the measures $\prob$ and $\prob^0 \otimes \prob^*$ coincide on the cylinders $A_i$ of the form
\begin{equation}\label{approximation:cilindro_iesimo}
    A_i=\insieme{ (\overline{\omega},\omega_1)=((\omega_0^j)_{j \geq 1} ,m,(\omega_*^j)_{j \geq 1}) \in \overline{\Omega}\times\Omega^1 \; \vert \; (\omega_0^i,\omega_*^i) \in G},
\end{equation}
for any $G \in \F^{0-} \otimes \F^1$, every $\prob^0 \otimes \prob^*$-null set $N$ can be identified with a $\prob$-null cylinder $A_i$ of the form \eqref{approximation:cilindro_iesimo} with basis $N$.
Therefore, for every $t \in [0,T]$, the $\prob$-augmentation of the filtration $\F^{0-}\otimes\F^i_t$ contains all the cylinders with basis 
\begin{equation*}
    A_i=\insieme{ (\overline{\omega},\omega_1)=((\omega_0^j)_{j \geq 1} ,m,(\omega_*^j)_{j \geq 1}) \in \overline{\Omega}\times\Omega^1 \; \vert \; (\omega_0^i,\omega_*^i) \in G},
\end{equation*}
for any $G$ in the $\prob^0 \otimes \prob^*$-augmentation of $\F^{0-} \otimes \F^*_t$. 
This is enough to conclude that $\lambda^i$ is progressively measurable with respect to the $\prob$-augmentation of $\F^{0-} \otimes \F^i_t$, and so with respect to the filtration $\mathbb{F}$ as well.
We define $\Lambda^i$ as in \eqref{raccomandazione_indotta}, that is
\begin{equation}\label{approximation:approx_raccomandazione}
\begin{aligned}
    \Lambda^i(\overline{\omega})=\left \{ \: \begin{aligned}
        & \begin{aligned}
            (\lambda^i_{t}(\overline{\omega},\cdot))_{t \in [0,T]} &  :\space [0,T] \times \Omega^1 \to A \\
        & (t,\omega_1) \to \lambda^i_t(\overline{\omega},\omega_1), 
        \end{aligned} &&  \overline{\omega} \in \overline{\Omega}\setminus \mathcal{N}, \\
        & a_0 && \overline{\omega} \in \mathcal{N},
    \end{aligned} \right.
\end{aligned}
\end{equation}
where $\mathcal{N} \subseteq \overline{\Omega}$ is a $\overline{\prob}$-null set and $a_0$ is an arbitrary point in $A$.
By Lemma \ref{esempi:lemma_misurabile}, $\Lambda^i$ is an admissible recommendation from $(\overline{\Omega},\overline{\F},\overline{\prob})$ to $(\A_N,\boreliani{\A_N})$, for every $N \geq i$.
Since the associated strategies coincide pointwise, it holds $\Lambda^i(\overline{\omega})=\Lambda^*(\omega_0^i)$ 
$\overline{\prob}$-a.s., as ensured by Proposition \ref{esempi:unicita_strategia_associata}.
In particular, this implies that $\Lambda^i$ only takes values in $\A_{(i)}$, since for every fixed $\overline{\omega}$ the control process $(\lambda^i_{t}(\overline{\omega},\cdot))_{t \in [0,T]}$ is $\mathbb{F}^{(i)}$-progressively measurable.
This proves point \ref{approximation:proprieta_raccomandazioni_N_player:ammissibilita}.

As for point \ref{approximation:proprieta_raccomandazioni_N_player:legge_congiunta_profilo}, for every $N \geq 2$, $(\Lambda^1,\dots,\Lambda^N)$ takes values in $\bigtimes_{j=1}^N\A_{\tonde{j}}$ by construction.
Hence, we may restrict the attention to Borel sets $C_j \subseteq \A_{\tonde{j}}$, for every $j=1,\dots,N$.
Let $B \in \mathcal{B}$.
Since  $\Lambda^j(\overline{\omega})=\Lambda^*(\omega^j_0)$ for every $j=1,\dots,N$ $\overline{\prob}$-a.s., by definition of $\overline{\prob}$, we have
\begin{equation*}
\begin{aligned}
    \overline{\prob} & \tonde{\Lambda^{1} \in C_1, \dots, \Lambda^N \in C_N, \mu \in B} = \overline{\prob}\tonde{\bigtimes_{j=1}^N\insieme{\omega^j_0: \; \Lambda^*(\omega_0^j) \in C_j} \; \times \bigtimes_{j=N+1}^\infty \Omega^0 \; \times \;  \contpdue } \\
    & = \int_B \prod_{j=1}^N K\tonde{\insieme{\omega^j_0:\;\Lambda^* (\omega_0^j) \in C_j},m} \rho(dm) = \int_B \prod_{j=1}^N K\tonde{\insieme{\omega_0:\;\Lambda^*\tonde{\omega_0} \in C_j},m} \rho(dm)  \\
    & = \int_{B} \prod_{j=1}^N \kappa\tonde{C_j, m} \rho(dm).
\end{aligned}
\end{equation*}
This shows also that $(\Lambda^i,\mu)$ are identically distributed as $(\Lambda^*,\mu^*)$ and that $(\Lambda^i)_{i\geq 1}$ are conditionally i.i.d. given $\mu$.
\end{proof}
For each $N\geq 2$, set $\F^{0-,N}=\sigma(\Lambda^1,\dots,\Lambda^N)$ and $(\Omega^{0,N},\F^{0-,N},\prob^{0,N})=(\overline{\Omega},\F^{0-,N},\overline{\prob})$.
Then,  
\begin{equation*}
    \tonde{\Lambda^1,\dots,\Lambda^N}:\tonde{\overline{\Omega},\F^{0-,N},\overline{\prob}} \to \tonde{\A_N^N, \boreliani{\A_N^N}}
\end{equation*}
is the candidate $\eps_N$-coarse correlated equilibrium to the $N$-player game, with $\eps_N$ to be determined.

\begin{remark}
The construction of the probability spaces $(\Omega^{0,N},\F^{0-,N},\prob^{0,N})$ is rather involved, but has the advantage of making the admissible recommendation to the $N$-players $(\Lambda^1,\dots,\Lambda^N)$ easy to define.
Besides this technical reason, we notice that, both in the $N$-player game and in the mean field game, the mediator may choose the space $(\Omega^0,\F^{0-},\prob^0)$ he uses to randomize players' strategies, as already pointed out in Sections \ref{sezione_formulazione_N_giocatori} and \ref{sezione_formulazione_mfg}.
Then, it is natural to use the same space $(\Omega^0,\F^{0-},\prob^0)$ on which the coarse correlated solution to the mean field game is defined to randomize players' strategies in the $N$-player game as well.
\end{remark}

\subsection{Proof of Theorem \ref{thm_approssimazione}}\label{sec:approximation:proof_thm}
By symmetry, let us consider only possible deviations of player $i=1$.
For every $N\geq 2,$ let $\eps_N$ be given by
\begin{equation}
    \eps_N :=\sup_{\beta \in \A_N}\tonde{\J^N_1(\Lambda^N) - \J^N_1(\Lambda^{N,-1},\beta)}=\J^N_1(\Lambda^N) - \inf_{\beta \in \A_N}\J_N^1(\Lambda^{N,-1},\beta).
\end{equation}
By definition of $\eps_N$, $\Lambda^N$ is an $\eps_N$-coarse correlated equilibrium for every $N \geq 2$.
In order to conclude the proof, we only need to prove that $\eps_N \to 0$ as $N \to \infty$.
For each $N \geq 2$, choose $\beta^N \in \A_N$ so that
\begin{equation*}
    \J^N_1(\Lambda^{N,-1},\beta^N) \leq  \inf_{\beta \in \A_N}\J^N_1(\Lambda^{N,-1},\beta) - \frac{1}{N}.
\end{equation*}
Let $Z=(Z_t)_{t \in [0,T]}$ be the solution of
\begin{equation}\label{approximation:equazione_Z}
    dZ_t=b(t,Z_t,\mu_t,\beta^N_t)dt + dW^1_t, \quad Z_0=\xi^1,
\end{equation}
and define the corresponding cost as
\begin{equation*}
    \J(\beta^N,\mu)=\E\quadre{\int_0^T f(s,Z_s,\mu_s,\beta^N_s)ds + g(Z_T,\mu_T)}.
\end{equation*}
Let $X$ be the solution of 
\begin{equation}\label{equazione_ottima}
    dX_t=b(t,X_t,\mu_t,\lambda^1_t)dt + dW^1_t, \quad X_0=\xi^1,
\end{equation}
with associated cost
\begin{equation*}
    \J(\Lambda^1,\mu)=\E\quadre{\int_0^T f(s,X_s,\mu_s,\lambda^1_s)ds + g(X_T,\mu_T)}.
\end{equation*}
Observe that, by construction,  $(\xi^1,W^1,\mu,\lambda^1)$ under $\prob$ is distributed as $(\xi,W,\mu^*,\lambda^*)$ under $\prob^*$.
Therefore, by Theorem \ref{teorema_di_unicita_legge} the joint distribution of $(X,\lambda^1,\mu)$ under $\prob$ is the same as $(X^*,\lambda^*,\mu^*)$ under $\prob^*$, where $X^*$ denotes the state process resulting from the mean field CCE $(\Lambda^*,\mu^*)$.
Moreover, note that by construction $\lambda^1$ is $\mathbb{F}^{\mu,\xi^1,W^1}$-progressively measurable, where
\begin{equation*}
    \F^{\mu,\xi^1,W^1}_t= \sigma(\mu)\vee\sigma(\xi^1)\vee\sigma(W^1_s: \; s \leq t), \; t \in [0,T].
\end{equation*}
By the Lipschitz continuity of $b$, $X$ may be taken to be $\mathbb{F}^{\mu,\xi^1,W^1}$-adapted as well.

To prove the theorem, it is enough to show the following:
\begin{subequations}
\label{conv_costi:optim}
\begin{align}
    & \J(\Lambda^1,\mu) \leq \J(\beta^N,\mu) \label{conv_costi:minimo_limite}, \\
    & \lim_{N \to \infty} \J_N^1(\Lambda^N)=\J(\Lambda^1,\mu),  \label{conv_costi:convergenza_lambda} \\
    & \lim_{N \to \infty} \abs{\J_N^1(\Lambda^{N,-1},\beta^N) - \J(\beta^N,\mu)} = 0. \label{conv_costi:convergenza_beta}
\end{align}
\end{subequations}
These imply the conclusion, by noticing that
\begin{equation*}
    \eps_N \leq  \J^1_N(\Lambda^N) - \J(\Lambda^1,\mu) +\J(\Lambda^1,\mu) - \J(\beta^N,\mu) + \J(\beta^N,\mu) - \J^1_N(\Lambda^{N,-1},\beta^N) - \frac{1}{N}. 
\end{equation*}

We start by proving \eqref{conv_costi:minimo_limite}.
We observe that we cannot just deduce it from the optimality property \eqref{def_mean_field_sol:opt} of $(\Lambda^*,\mu^*)$: since $\beta^N$ may belong to $\A_N \setminus \A_{(1)}$, it may not be identifiable with an open loop strategy for the MFG, for which inequality \eqref{conv_costi:minimo_limite} would hold.
Instead, we prove it by using the regular conditional probability of $\prob$ given $(\xi^i,W^i)_{i=2}^N$.
Denote by $(\bm{x},\bm{w})$ a point $(x^i,w^i)_{i=2}^N \in (\R^d \times \contrd)^{N-1}$, and let $P_\nu=\bigotimes_1^{N-1}(\nu \otimes \wienermeasure)$ denote the joint law of $(\xi^i,W^i)_{i=2}^N$ under $\prob$.
Let $\prob^{\bm{x},\bm{w}}$ be a version of the regular conditional probability of $\prob$ given $(\xi^i,W^i)_{i=2}^N=(x^i,w^i)_{i=2}^N$.
We rewrite \eqref{conv_costi:minimo_limite} as
\begin{equation}\label{conv_costi:int_disintegrato}
\begin{aligned}
    \J & (\Lambda^1,\mu)-\J(\beta^N,\mu) =\\
     & = \E \quadre{\tonde{\int_0^T f(s,X_s,\mu_s,\lambda^1_s)ds + g(X_T,\mu_T)} - \tonde{\int_0^T f(s,Z_s,\mu_s,\beta^N_s)ds + g(Z_T,\mu_T) }} \\
    & =  \E  \left[ \E\left[ \tonde{\int_0^T f(s,X_s,\mu_s,\lambda^1_s)ds + g(X_T,\mu_T)} \right. \right. \\
    & \quad \left. \left. - \tonde{\int_0^T f(s,Z_s,\mu_s,\beta^N_s)ds + g(Z_T,\mu_T)} \Big \vert (\xi^i,W^i)_{i=2}^N \right] \right] \\
    & = \int_{(\R^d \times \contrd)^{N-1}} \left( \E^{\prob^{\bm{x},\bm{w}}}\left[\int_0^T f(s,X_s,\mu_s,\lambda^1_s)ds + g(X_T,\mu_T) \right] \right. \\
    & \quad - \left. \E^{\prob^{\bm{x},\bm{w}}}\left[ \int_0^T f(s,Z_s,\mu_s,\beta^N_s)ds + g(Z_T,\mu_T) \right] \right) P_\nu(d\bm{x},d\bm{w}).
\end{aligned}
\end{equation}
We analyse separately the two terms in the last equality.
Let us start with the term depending upon $\lambda^1$.
Since $\mu$, $\lambda^1$, $W^1$ and $\xi^1$ are independent of $(\xi^i,W^i)_{i=2}^N$  under $\prob$ and $X$ is $\mathbb{F}^{\mu,\xi^1,W^1}$-adapted, $X$ is independent of $(\xi^i,W^i)_{i=2}^N$ as well.
We deduce that, under $\prob^{\bm{x},\bm{w}}$, $W^1$ is an $\mathbb{F}$-Brownian motion, $X$ solves equation \eqref{equazione_ottima} and 
$\prob^{\bm{x},\bm{w}}\circ(X,\lambda^1,\mu)^{-1}=\prob\circ(X,\lambda^1,\mu)^{-1}=\prob^*\circ(X^*,\lambda^*,\mu^*)^{-1}$, for $P_\nu$-a.e. $(\bm{x},\bm{w}) \in (\R^d \times \contrd)^{N-1}$.
In particular, this implies that
\begin{equation}\label{funzionale_condizionato1}
\begin{aligned}
    \E^{\prob^{\bm{x},\bm{w}}}& \quadre{\int_0^T f(s,X_s,\mu_s,\lambda^1_s)ds + g(X_T,\mu_T)}=\E^{\prob^*}\quadre{\int_0^T f(s,X^*_s,\mu^*_s,\lambda^*_s)ds + g(X^*_T,\mu^*_T)}\\
    & =\J(\Lambda^*,\mu^*)
\end{aligned}
\end{equation}
for $P_\nu$-a.e. $(\bm{x},\bm{w})\in (\R^d \times \contrd)^{N-1}$.

\smallskip
As for the term depending upon $\beta^N$, we note that, since $\beta^N \in \A_N$, there exists a progressively measurable functional $\hat{\beta}:[0,T] \times (\R^d \times \contrd)^N \to A$ so that
\begin{equation*}
    \beta^N_t=\hat{\beta}\tonde{t,\xi^1,W^1,\dots,\xi^N,W^N}.
\end{equation*}
Under $\prob^{\bm{x},\bm{w}}$, it holds
\begin{equation}\label{strategia_condizionata1}
    \beta^N_t= \hat{\beta}\tonde{t,\xi^1,W^1,x^2,w^2,\dots,x^N,w^N} \; \forall t \in [0,T] \quad \prob^{\bm{x},\bm{w}}\text{-a.s.},
\end{equation}
since $(\xi^i,W^i)_{i=1}^2$ are almost surely constant under $\prob^{\bm{x},\bm{w}}$.
Since the joint law of $\mu$, $W^1$ and $\xi^1$ is the same under both $\prob$ and $\prob^{\bm{x},\bm{w}}$, \eqref{strategia_condizionata1} implies that $Z$ satisfies
\begin{equation}\label{conv_costi:differenziale_zeta}
    dZ_t=b(t,Z_t,\mu_t,\hat{\beta}^N_t(\bm{x},\bm{w}))dt + dW^1_t, \quad Z_0=\xi^1,
\end{equation}
under $\prob^{\bm{x},\bm{w}}$, with $\hat{\beta}^N(\bm{x},\bm{w})$ $\mathbb{F}^{(1)}$-progressively measurable.
For every $(\bm{x},\bm{w}) \in (\R^d \times \contrd)^{N-1}$, define the strategy
\begin{equation}
    \Tilde{\beta}\tonde{\bm{x},\bm{w}}=\tonde{ \hat{\beta}\tonde{t,\xi,W,x^2,w^2,\dots,x^N,w^N}}_{t \in [0,T]}.
\end{equation}
Then $\Tilde{\beta}(\bm{x},\bm{w})$ belongs to $\A$  for every $(\bm{x},\bm{w})\in(\R^d\times\contrd)^{N-1}$, and it depends measurably upon $(\bm{x},\bm{w})$.
For every $(\bm{x},\bm{w})$, let $\Tilde{Z}$ be the solution of 
\begin{equation*}
    d\Tilde{Z}_t=b(t,\Tilde{Z}_t,\mu^*_t,\Tilde{\beta}_t(\bm{x},\bm{w}))dt + dW_t, \quad \Tilde{Z}_0=\xi.
\end{equation*}
Since $\prob^{\bm{x},\bm{w}}\circ(Z,\beta,\mu)^{-1}=\prob^{\bm{x},\bm{w}}\circ(Z,\hat{\beta}^N(\bm{x},\bm{w}),\mu)^{-1}=\prob^*\circ(\Tilde{Z},\Tilde{\beta}(\bm{x},\bm{w}),\mu^*)^{-1}$, it follows that 
\begin{equation}\label{funzionale_condizionato2}
\begin{aligned}
    \E^{\prob^{\bm{x},\bm{w}}} & \bigg[ \int_0^T f(s,Z_s,\mu_s,\beta^N_s)ds + g(Z_T,\mu_T) \bigg]  \\
    & = \E^{\prob^{\bm{x},\bm{w}}}\bigg[\int_0^T f(s,Z_s,\mu_s,\hat{\beta}^N_s(\bm{x},\bm{w})) ds + g(Z_T,\mu_T)\bigg]\\
    & = \E^{\prob^*} \bigg[\int_0^T f(s,\hat{Z}_s,\mu^*_s,\Tilde{\beta}_s(\bm{x},\bm{w}))ds + g(\hat{Z}_T,\mu^*_T) \bigg] =\J(\Tilde{\beta}(\bm{x},\bm{w}),\mu^*).
\end{aligned}
\end{equation}
We note that the left-hand side of \eqref{funzionale_condizionato1} depends measurably upon $(\bm{x},\bm{w})$ due to a monotone class argument.
Being $(\Lambda^*,\mu^*)$ a mean field CCE by assumption, \eqref{funzionale_condizionato1} and \eqref{funzionale_condizionato2} imply that
\begin{equation*}
\begin{aligned}
    \J(\Lambda^1,\mu)-\J(\beta^N,\mu) = & \int  \left( \E^{\prob^{\bm{x},\bm{w}}}\left[\int_0^T f(s,X_s,\mu_s,\lambda^1_s)ds + g(X_T,\mu_T) \right] \right. \\
    & - \left. \E^{\prob^{\bm{x},\bm{w}}}\left[ \int_0^T f(s,Z_s,\mu_s,\beta^N_s)ds + g(Z_T,\mu_T) \right] \right) P_\nu(d\bm{x},d\bm{w}) \\
    = & \int  \left( \J(\Lambda^*,\mu^*) -\J(\Tilde{\beta}(\bm{x},\bm{w}),\mu^*) \right) P_\nu(d\bm{x},d\bm{w}) \leq 0,
\end{aligned}
\end{equation*}
which yields \eqref{conv_costi:minimo_limite}.

\smallskip
As for \eqref{conv_costi:convergenza_lambda} and \eqref{conv_costi:convergenza_beta}, they must be handled by continuity arguments on the cost functions and propagation of chaos as stated in Lemma \ref{lemma_poc}.
We give the details only for \eqref{conv_costi:convergenza_lambda}, since \eqref{conv_costi:convergenza_beta} is analogous.
We have:
\begin{equation*}
\begin{aligned}
    \abs{\J^1_N(\Lambda^N) - \J(\Lambda^1,\mu)} & \leq \E\left[\int_0^T \abs{f(t,X^{1,N}_t,\mu^N_t,\lambda^1_t) - f(t,X_t,\mu_t,\lambda^1_t) } dt \right. \\
     & \quad +  \abs{g(X^N_T,\mu_T^N) - g(X_T,\mu_T)} \bigg ] = \attesa{\Delta f + \Delta g}.
\end{aligned}
\end{equation*}
For $\Delta f$, Assumptions \ref{standing_assumptions} ensure that $f$ is locally Lipschitz with at most quadratic growth.
Therefore, by straightforward estimates, we have:
\begin{equation*}
\begin{aligned}
    & \E\quadre{\Delta f} \leq C \E \left[ 1 + \norm{X^{1,N}}_{\contrd}^2 + \norm{X}_{\contrd}^2 +  \int_0^T \int_{\R^d}\abs{y}^2 \mu_t(dy)dt + \int_0^T \int_{\R^d}\abs{y}^2 \mu^N_t(dy)dt \right. \\
    & \quad \left. + 2 \int_0^T \abs{\lambda^1_t} dt \right]^\frac{1}{2} \cdot \E\left[  \tonde{\norm{X^{1,N} - X}^2_{\contrd}  + \int_0^T\pwassmetric{2}{\R^d}{2}(\mu^N_t,\mu_t)dt}  \right]^\frac{1}{2} \\
    & \leq  C \tonde{1 + \max_{k=1,\dots,N}\E \quadre{ \norm{X^{k,N}}_{\contrd}^2}^\frac{1}{2} + \E \quadre{\norm{X}_{\contrd}^2}^\frac{1}{2}  }\\
    & \quad \cdot \tonde{ \E\quadre{ \norm{X^{1,N} - X}_{\contrd}^2 }^\frac{1}{2}+ \sup_{t \in [0,T]}\attesa{\pwassmetric{2}{\R^d}{2}(\mu^N_t,\mu_t)}^{\frac{1}{2}} }. \\
\end{aligned}
\end{equation*}
By Lemma \ref{lemma_poc}, the right-hand side tends to $0$ as $N$ goes to infinity. The convergence of $\E[\Delta g]$ is shown analogously.

\section{Existence of a coarse correlated solution of the mean field game}\label{sezione_existence}

The main result of this section regards the existence of a coarse correlated solution of the MFG, which requires the following additional assumption:
\begin{customassumption}{\textbf{B}}\label{convexity_assumptions}
For every $(t,x,m) \in [0,T] \times \R^d \times \pwassspace{2}{\R^d}$, the set
\begin{equation}\label{convexity_assumptions:insieme_K}
    K(t,x,m)=\insieme{ (b(t,x,m,a),z): \; a \in A, \; f(t,x,m,a) \leq z} \subseteq \R^d \times \R
\end{equation}
is closed and convex.
\end{customassumption}
Assumption \ref{convexity_assumptions} is standard when dealing with relaxed controls (see, e.g., \cite{el_karoui_compactification} and \cite{haussmann1990sicom} for relaxed controls in control theory, and the series of works \cite{carmona2016commonnoise,lacker2015martingale,lacker2016,lacker2020convergence}, or \cite{campi18absorption}, in mean field games literature).
Similarly to these works, we will use relaxed controls to apply compactness arguments, and we will use Assumption \ref{convexity_assumptions} to come back to strict controls.
We refer to Section \ref{existence:sezione_crtl_rilassati} in the Appendix for a brief recap on relaxed controls.

\smallskip
The main result of this section is the following:
\begin{theorem}[Existence of a coarse correlated solution of the MFG]\label{existence:main_theorem}
In addition to Assumptions \ref{standing_assumptions}, suppose that Assumption \ref{convexity_assumptions} holds.
Then there exists a coarse correlated solution of the mean field game.
\end{theorem}
The road map to prove Theorem \ref{existence:main_theorem} is as follows:
Taking inspiration from \cite{hart_schmeidler, nowak1992correlated, Nowak1993} and \cite[Appendix 1.B]{bonesini_thesis}, we associate a zero-sum game to the search of a mean field CCE.
Loosely speaking, the game should be of the following type: player A, the maximizer, chooses a correlated flow $(\Lambda,\mu)$, while player B chooses a deviating strategy $\beta \in \A$.
The payoff functional is the following:
\begin{equation}\label{existence:payoff_informale}
    F[(\Lambda,\mu),\beta] =\J(\beta,\mu)-\J(\Lambda,\mu).
\end{equation}
Player A aims at maximizing $F$, while player B chooses her strategy in order to minimize $F$.
In order to get an equilibrium, one should restrict to correlated flows $(\Lambda,\mu)$ so that the consistency condition \eqref{def_mean_field_sol:cons} is satisfied.
If we could show that the game has a positive value and player A has an optimal strategy $(\Lambda^*,\mu^*)$ , then we would have established that such a strategy would satisfy the optimality property \eqref{def_mean_field_sol:opt} as well, and therefore $(\Lambda^*,\mu^*)$ would be a mean field CCE.
In order to get a convenient structure for the sets of strategies and good continuity and convexity properties of the payoff functionals, in Section \ref{sec_existence:sec_auxiliary} we define a more general zero-sum game, in which we embed our auxiliary problem.
As shown in Proposition \ref{existence:prop_relazione_mfg}, the auxiliary zero-sum game extends the payoff functional in equation \eqref{existence:payoff_informale}.
Particular care is needed in dealing with the term depending both on $\beta$ and $\mu$, since it must reflect independent strategy choices of the opponents.
Using Fan's minimax theorem, we will show that the auxiliary game has positive value and admits an optimal strategy for the maximizing player. This is the content of Theorem \ref{existence:thm_esistenza_strategia_ottima}.
Finally, in Section \ref{sec_existence:sec_proof} we prove Theorem \ref{existence:main_theorem}, by using such an optimal strategy to induce a coarse correlated solution of the mean field game.
Many technical lemmata are needed to prove this existence result. For reader's convenience, most of the proofs and some statements are deferred to Appendix \ref{appendix_existence}.

\begin{remark}\label{existence:rmk_existence_result}
Under similar assumptions to Assumptions \ref{standing_assumptions}, the existence of weak or also strong MFG solutions has been established in the literature; see, for instance, \cite{lacker2015martingale,lacker2016,lacker_leflem2022}.
As already noticed, strong MFG solutions are coarse correlated solutions to the MFG as well, and, at least in some cases, this is also true for weak MFG solutions (see upcoming Section \ref{sec:comparions_weak_mfg}).
Nevertheless, we think that the existence result of Theorem \ref{existence:main_theorem} is of independent interest, for two main reasons.
Firstly, the proof shows directly the existence of a coarse correlated solutions without relying on existence results for stronger notions of equilibria, which are usually based on fixed point theorems. Such a direct proof is instead based on a minimax theorem, which, to the best of our knowledge, is used for the first time in continuous time MFG literature.
Secondly, since we reduce the search of a coarse correlated solution to the MFG to finding an equilibrium for an auxiliary linear zero-sum game on spaces of measures, our formulation paves the way to the use of linear programming methods for computing mean field CCEs. We further discuss this point in the following Remark \ref{existence:rmk_linear_programming}.
\end{remark}

\subsection{The auxiliary zero-sum game}\label{sec_existence:sec_auxiliary}

We now formally define the auxiliary zero-sum game.
\begin{definition}[Strategies for player A]\label{existence:strategie_max}
A \emph{strategy for player A} is a probability measure $\Gamma \in \mathcal{P}(\contrd \times \mathcal{V} \times \contpdue)$ so that there exists a tuple $\mathfrak{U}=((\Omega,\F,\mathbb{F},\prob),\xi,W,\mu,\mathfrak{r})$ with the following properties:
\begin{enumerate}[label=(\roman*)]
    \item $(\Omega,\F,\mathbb{F},\prob)$ is a filtered probability space satisfying the usual assumptions; $\Omega$ is Polish and $\F$ is its corresponding Borel $\sigma$-algebra.
    \item $W$ is an $\mathbb{F}$-Brownian motion and $\xi$ is an $\F_0$-measurable independent $\R^d$-valued random variable with law $\nu$.
    \item $\mu$ is an $\F_0$-measurable random variable with values in $\contpdue$; it is independent of both $\xi$ and $W$.
    \item $\mathfrak{r}$ is an $\mathbb{F}$-progressively measurable relaxed control $\mathfrak{r}=(\mathfrak{r}_t)_{t \in [0,T]}$ with values in $A$.
    \item Let $X$ be the solution of
    \begin{equation}\label{existence:eq_processo_K}
        dX_t=\int_A b(t,X_t,\mu_t,a)\mathfrak{r}_t(da)dt + dW_t, \; t \in [0,T], \quad X_0=\xi.
    \end{equation}
    Then $\mu_t(\cdot)=\prob(X_t \in \cdot \; \vert \; \mu)$ $\prob$-a.s for every $t \in [0,T]$.
    
    \item $\Gamma$ is the joint law under $\prob$ of  $X$, $\mu$ and $\mathfrak{r}$: $\Gamma=\prob\circ(X,\mathfrak{r},\mu)^{-1}$.
\end{enumerate}
We denote by $\mathcal{K}$ the set of strategies for player A.
\end{definition}
We observe that, by Assumptions \ref{standing_assumptions}, there exists a unique solution to equation \eqref{existence:eq_processo_K} for every tuple $\mathfrak{U}$ satisfying properties (i-iv).

\begin{definition}[Strategies for player B]\label{existence:strategie_min}
A stochastic kernel $\Sigma$ from $\contpdue$ to $\contrd \times \mathcal{V}$ is a \emph{strategy for player B} if there exists a tuple $\mathfrak{U}=((\Omega,\F,\mathbb{F},\prob),\xi,W,\mathfrak{r})$ so that
\begin{enumerate}[label=(\roman*)]
    \item $(\Omega,\F,\mathbb{F},\prob)$ is a filtered probability space satisfying the usual assumptions; $\Omega$ is Polish and $\F$ is its corresponding Borel $\sigma$-algebra.
    \item $W$ is an $\mathbb{F}$-Brownian motion and $\xi$ is an $\F_0$-measurable independent $\R^d$-valued random variable with law $\nu$.
    \item $\mathfrak{r}$ is an $\mathbb{F}$-progressively measurable relaxed control $\mathfrak{r}=(\mathfrak{r}_t)_{t \in [0,T]}$ with values in $A$.
    \item For every $m \in \contpdue$, $\Sigma(\cdot, m) \in \mathcal{P}(\contrd \times \mathcal{V})$ is the joint law under $\prob$ of $(X^m,\mathfrak{r})$, where $X^m$ is the solution to
    \begin{equation}\label{existence:eq_processo_Q}
    dX^m_t=\int_A b(t,X^m_t,m_t,a)\mathfrak{r}_t(da)dt + dW_t, \quad X_0=\xi,
    \end{equation}
    that is:
    \begin{equation}
        \Sigma(B,m)=\prob((X^m,\mathfrak{r}) \in B) \quad \forall m \in \contpdue, B \in \boreliani{\contrd} \otimes \boreliani{\mathcal{V}}.
    \end{equation}
\end{enumerate}
We denote by $\mathcal{Q}$ the set of strategies for player B.
\end{definition}
By Lemma \ref{existence:lemma_kernels}, the set of strategies $\mathcal{Q}$ for player B is well defined in the sense that
the map $\Sigma$ is truly a stochastic kernel.

We now define the payoff functional $\mathfrak{p}$ for the zero-sum game.
Let us introduce the function $\mathfrak{F}: \contrd \times \mathcal{V}\times \contpdue \to  \R$ defined by
\begin{equation}\label{existence:funzione_f}
\begin{aligned}
    \mathfrak{F} (y,q,m)=\int_0^T\int_A f(t,y_t,m_t,a)q_t(da)dt + g(y_T,m_T).
\end{aligned}
\end{equation}
\begin{definition}[Auxiliary zero-sum game]\label{existence:def_zerosum}
The \emph{auxiliary zero-sum game} is a zero-sum game where:
\begin{itemize}
    \item The set of strategies for player A, the maximizer, is the set $\mathcal{K}$ introduced in Definition \ref{existence:strategie_max}.
    \item The set of strategies for player B, the minimizer, is the set $\mathcal{Q}$ introduced in Definition \ref{existence:strategie_min}.
    \item The payoff functional is the function $\mathfrak{p}:\mathcal{K}\times\mathcal{Q}\to \R$ defined as
    \begin{equation}\label{existence:payoff_functional}
    \begin{aligned}
        \mathfrak{p}(\Gamma,\Sigma) = & \int_{\contrd \times \mathcal{V} \times \contpdue} \mathfrak{F} (y,q,m) \Sigma(dy,dq,m)\rho(dm)\\
        & - \int_{\contrd \times \mathcal{V} \times \contpdue} \mathfrak{F} (y,q,m) \Gamma(dy,dq,dm),
    \end{aligned}
    \end{equation}
    where $\rho$ denotes the marginal of $\Gamma$ on $\contpdue$.
\end{itemize}
We denote the lower and upper values of the game as, respectively, $v^A$ and $v^B$: 
\begin{equation*}
\begin{aligned}
    & v^A=\sup_{\Gamma \in \mathcal{K}} \inf_{\Sigma \in \mathcal{Q}} \mathfrak{p}(\Gamma,\Sigma), && \quad 
    v^B= \inf_{\Sigma \in \mathcal{Q}} \sup_{\Gamma \in \mathcal{K}} \mathfrak{p}(\Gamma,\Sigma).
\end{aligned}
\end{equation*}
If the lower and upper values of the game are equal, we set $v=v^A=v^B$ and call $v$ the value of the game.
We say that a strategy $\Gamma^* \in \mathcal{K}$ is optimal for player A if
\begin{equation*}
    \inf_{\Sigma \in \mathcal{Q}}\mathfrak{p}(\Gamma^*,\Sigma)=\max_{\Gamma \in \mathcal{K}}\inf_{\Sigma \in \mathcal{Q}}\mathfrak{p}(\Gamma,\Sigma).
\end{equation*}
\end{definition}

The next result ensures existence of an optimal strategy for the maximizing player:
\begin{theorem}[Existence of the value of the game and of an optimal strategy for the maximizing player]\label{existence:thm_esistenza_strategia_ottima}
Consider the game described in Definition \ref{existence:def_zerosum}.
The following holds:
\begin{enumerate}[label=(\roman*)]
    \item \label{existence:thm_esistenza_strategia_ottima:existence_value} The game has a value, i.e. $v^A=v^B$.
    
    \item \label{existence:thm_esistenza_strategia_ottima:optimal_strategy} There exists a strategy $\Gamma^* \in \mathcal{K}$ which is optimal for player $A$.

    \item \label{existence:thm_esistenza_strategia_ottima:positive_value} The value $v$ of the game is non negative: $v \geq 0$.

\end{enumerate}
\end{theorem}
The proof of this theorem is deferred to Section \ref{existence:sezione_opt_strat_auxiliary} in Appendix \ref{appendix_existence}.
The last result of this section shows that, for every correlated flow $(\Lambda,\mu)$ so that consistency condition \eqref{def_mean_field_sol:cons} is satisfied and every deviation $\beta \in \A$, there exists a pair of strategies $(\Gamma,\Sigma) \in \mathcal{K}\times\mathcal{Q}$ so that the following equality holds:
\begin{equation}\label{existence:prop_relazione_mfg:equazione_payoffs}
        \mathfrak{p}(\Gamma,\Sigma)=\J(\Lambda,\mu)-\J(\beta,\mu)=F[(\Lambda,\mu),\beta].
\end{equation}

\begin{proposition}\label{existence:prop_relazione_mfg}
Let $((\Omega^0,\F^{0-}$, $\prob^0),\Lambda,\mu)$ be a correlated flow.
Denote by $\rho$ the law of $\mu$.
Let $\lambda=(\lambda_t)_{t \in [0,T]}$ be the strategy associated to the admissible recommendation $\Lambda$ and let $\beta \in \A$.
\begin{enumerate}[label=(\roman*)]
    \item     Let $X$ be the solution to \eqref{dinamica_MF_no_deviazione}.
    Suppose that consistency condition \eqref{def_mean_field_sol:cons} is satisfied.
    For every $t \in [0,T]$, set $\mathfrak{r}_t(da)dt=\delta_{\lambda_t}(da)dt$.
    Then, the probability measure $\Gamma=\prob\circ(X,\mathfrak{r},\mu)^{-1}$ belongs to the set $\mathcal{K}$.

    \item \label{existence:prop_relazione_mfg:deviazioni}
    For every $t \in [0,T]$, set $\mathfrak{b}_t(da)dt=\delta_{\beta_t}(da)dt$.
    Denote by $Y$ the solution to \eqref{dinamica_MF_deviation}.
    Then, there exists $\Sigma \in \mathcal{Q}$ so that
    \begin{equation}\label{existence:prop_relazione_mfg:decomposizione}
        \prob((Y,\mathfrak{b},\mu)\in B \times S)=\int_S \Sigma(B,m)\rho(dm), \quad \forall B \in \boreliani{\contrd \times \mathcal{V}}, \; S \in \boreliani{\contpdue}.
    \end{equation}
    
    \item \label{existence:prop_relazione_mfg:relazione_payoffs}
    The pair of strategies $(\Gamma,\Sigma)$ satisfies equation \eqref{existence:prop_relazione_mfg:equazione_payoffs}.
\end{enumerate}
\end{proposition}
\begin{proof}
In the following, we work on the probability space $(\Omega,\F,\mathbb{F},\prob)$ defined in \eqref{mf:condizione_ammissibilita}.
Recall that, as pointed out in Remark \ref{finite_players:remark_estensioni}, we can think of $W$, $\xi$ and $\mu$ as independent random variables, each of them defined on the same probability space $(\Omega,\F,\mathbb{F},\prob)$.
Observe that the $\mathcal{P}(A)$-valued process $\mathfrak{r}=(\delta_{\lambda_t})_{t \in [0,T]}$ is $\mathbb{F}$-progressively measurable since $\Lambda$ is admissible by assumption.
Let $X$ be the solution to equation \eqref{dinamica_MF_no_deviazione}.
Since $X$ obviously satisfies \eqref{existence:eq_processo_K} for such a process $\mathfrak{r}$ and the condition $\mu_t(\cdot)=\prob(X_t \in \cdot \; \vert \; \mu)$ holds by assumption, $\Gamma=\prob\circ(X,\mu,\mathfrak{r})^{-1}$ belongs to $\mathcal{K}$.

As for point \ref{existence:prop_relazione_mfg:deviazioni}, recall from Remark \ref{finite_players:remark_estensioni} that we can regard $\beta$ as defined on the product space $(\Omega,\F,\prob)$, and that $\beta$ and $\mu$ are mutually independent by construction.
Therefore, the $\mathcal{P}(A)$-valued process $\mathfrak{b}=(\delta_{\beta_t}(da))_{t \in [0,T]}$ is independent of $\mu$.
Let $Y$ be the solution of equation \eqref{dinamica_MF_deviation}.
By Lemma \ref{existence:lemma_decomposizione} in Appendix \ref{appendix_existence}, equation \eqref{existence:prop_relazione_mfg:decomposizione} holds.

Finally, since $X$ and $Y$ are defined on the same filtered probability space $(\Omega,\F,\mathbb{F},\prob)$, we can write the integrals in $\mathfrak{p}$ as expectations:
\begin{equation*}
\begin{aligned}
    & \int \mathfrak{F} (y,q,m) \Gamma(dy,dq,dm) = \E\quadre{\int_0^T f(t,X_t,\mu_t,\lambda_t)dt + g(X_T,\mu_T)} = \J(\Lambda,\mu), \\
    & \int \mathfrak{F} (y,q,m) \Sigma(dy,dq,m)\rho(dm) = \E\quadre{\int_0^T  f(t,Y_t,\mu_t,\beta_t)dt + g(Y_T,\mu_T)}=\J(\beta,\mu).
\end{aligned}
\end{equation*}
This proves \eqref{existence:prop_relazione_mfg:equazione_payoffs}.
\end{proof}

\subsection{Proof of Theorem \ref{existence:main_theorem}}\label{sec_existence:sec_proof}
Let $\Gamma^* \in \mathcal{K}$ be an optimal strategy for player A, which exists by Theorem \ref{existence:thm_esistenza_strategia_ottima}.
By applying a mimicking argument, stated and proven in Lemma \ref{existence:lemma_mimicking} in Appendix \ref{appendix_auxiliary:technical}, there exists 
a strategy $\hat{\Gamma}^*\in \mathcal{K}$ so that the following holds:
\begin{itemize}
    \item The marginal distributions of $\Gamma^*$ and $\hat{\Gamma}^*$ on $\contpdue$ are the same: $\Gamma^*(\contrd \times \mathcal{V} \times \cdot)=\hat{\Gamma}^*(\contrd \times \mathcal{V} \times \cdot)$.
    
    \item Let $(X,\mathfrak{r},\mu)$ be such that $\hat{\Gamma}^*=\prob\circ(X,\mathfrak{r},\mu)^{-1}$.
    Then $\mathfrak{r}$ is of the form $\mathfrak{r}_t=\hat{q}_t(X_t,\mu)$, where $\hat{q}:[0,T]\times\R^d\times\contpdue \to \mathcal{P}(A)$ is a measurable function.
    \item For every $\Sigma \in \mathcal{Q}$, it holds
    \begin{equation*}
        \mathfrak{p}(\Gamma^*,\Sigma)=\mathfrak{p}(\hat{\Gamma}^*,\Sigma).
    \end{equation*}
\end{itemize}
In particular, it holds
\begin{equation}\label{existence:main_thm:condizione_min_max}
    \inf_{\Sigma \in \mathcal{Q}}\mathfrak{p}(\hat{\Gamma}^*,\Sigma)=\inf_{\Sigma \in \mathcal{Q}}\mathfrak{p}(\Gamma^*,\Sigma)=\max_{\Gamma \in \mathcal{K}}\inf_{\Sigma \in \mathcal{Q}}\mathfrak{p}(\Gamma,\Sigma) \geq 0.
\end{equation}
Let $\mathfrak{U}=((\hat{\Omega},\hat{\F},\hat{\mathbb{F}},\hat{\prob}),\hat{\xi},\hat{W},\hat{\mu},\hat{\mathfrak{r}})$ 
be as in Definition \ref{existence:strategie_max}, so that $\hat{\Gamma}^*=\prob\circ(\hat{X},\hat{\mathfrak{r}},\hat{\mu})^{-1}$.
Recall that, by Lemma \ref{existence:lemma_mimicking}, $\hat{\mathfrak{r}}_t(da)dt=\hat{q}_t(\hat{X}_t,\hat{\mu})(da)dt$ $Leb_{[0,T]} \otimes \hat{\prob}$-a.s.
By Assumption \ref{convexity_assumptions}, the set $K(t,x,m_t)$ defined by \eqref{convexity_assumptions:insieme_K} is convex for every $(t,x,m) \in [0,T] \times \R^d \times \contpdue$.
Therefore, by a well known measurable selection argument (see, e.g., \cite[Lemma A.9]{haussmann1990sicom}) there exists a measurable function $\hat{\alpha}:[0,T]\times\R^d\times \contpdue \to A$ so that
\begin{equation}\label{existence:main_theorem:strict_ctrl}
\begin{aligned}
    \int_A b(t,x,m_t,a)\hat{q}_t(x,m)(da)=b(t,x,m_t,\hat{\alpha}(t,x,m)), \\
    f(t,x,m_t,\hat{\alpha}(t,x,m)) \leq \int_A f(t,x,m_t,a)\hat{q}_t(x,m)(da).
\end{aligned}
\end{equation}
It follows that $\hat{X}$ is a solution to equation
\begin{equation}\label{existence:main_theorem:opt_ctrled_eq}
    d\hat{X}_t=b(t,\hat{X}_t,\hat{\mu}_t,\hat{\alpha}(t,\hat{X}_t,\hat{\mu}))dt + d\hat{W}_t, \quad \hat{X}_0=\hat{\xi}
\end{equation}
as well, and the consistency condition \eqref{def_mean_field_sol:cons} is still satisfied.
By Lemma \ref{existence:mimicking:lemma_strong_existence}, we deduce that the solution $\hat{X}$ to equation \eqref{existence:main_theorem:opt_ctrled_eq} can be taken adapted to the $\hat{\prob}$-augmentation of the filtration $(\sigma(\hat{\mu}) \vee \sigma(\hat{\xi}) \vee \sigma(\hat{W}_s: \; s \leq t))_{t \in [0,T]}$, and therefore there exists a progressively measurable function $\Phi:\contpdue\times\R^d\times\contrd\to\contrd$ so that
\begin{equation}
    \hat{X}= \Phi(\hat{\mu},\hat{\xi},\hat{W}) \quad \hat{\prob}\text{-a.s.}
\end{equation}
Set
\begin{equation}\label{existence:main_theorem:ctrl_indotto}
\begin{aligned}
    & \begin{aligned}
        \hat{\lambda}:[0,T]\times\contpdue\times \R^d \times \contrd & \to A \\
    (t,m,x,w)&\mapsto \hat{\lambda}_t(m,x,w)=\hat{\alpha}_t(\Phi_t(m,x,w),m_t);
    \end{aligned} \\
    & \lambda =(\lambda_t)_{t \in [0,T]}=(\hat{\lambda}_t(\hat{\mu},\hat{\xi},\hat{W}))_{t \in [0,T]}.
\end{aligned}
\end{equation}
Then, the progressively measurable processes $(\hat{\alpha}_t(\hat{X}_t,\hat{\mu}))_{t \in [0,T]}$ and $(\hat{\lambda}_t(\hat{\mu},\hat{\xi},\hat{W}))_{t \in  [0,T]}$ are equal $Leb_{[0,T]}\otimes\hat{\prob}$-a.s., which implies that $\hat{X}$ solves 
\begin{equation*}
    d\hat{X}_t= b(t,\hat{X}_t,\hat{\mu}_t,\hat{\lambda}_t(\hat{\mu},\hat{\xi},\hat{W}))dt + d\hat{W}_t, \quad \hat{X}_0=\hat{\xi}
\end{equation*}
as well, and the consistency condition is still satisfied.
Set $(\Omega^0,\F^{0-},\prob^0)=(\contpdue,\boreliani{\contpdue}, \linebreak \rho)$.
By Lemma \ref{esempi:lemma_misurabile}, there exists a $\prob^0$-null set $N \subset \Omega^0$ so that the pair $(\Lambda^*,\mu*)$ defined by
\begin{equation}
\begin{aligned}
    \Lambda^*: \tonde{\contpdue,\boreliani{\contpdue},\rho} & \to (\A,\boreliani{\A}) \\
    m & \mapsto \Lambda^*(m)= \left\{ \begin{aligned}
        & (\hat{\lambda}_t(m,\cdot,\cdot))_{t \in [0,T]}, && m \in \Omega^0 \setminus N, \\
        & a_0  && m \in N,
        \end{aligned} \right. \\
    \mu^*=\text{Id}:\tonde{\contpdue,\boreliani{\contpdue},\rho} & \to \tonde{\contpdue,\boreliani{\contpdue},\rho}
\end{aligned}
\end{equation}
is a correlated flow, where $a_0$ is an arbitrary point in $A$.
Let $X^*$ be the solution of \eqref{dinamica_MF_deviation} on the product probability space $(\Omega,\F,\mathbb{F},\prob)$ defined in \eqref{mf:condizione_ammissibilita}.
Note that the strategy associated to the admissible recommendation $\Lambda^*$ strategy $\lambda^*$ is equal to $\hat{\lambda}_t(\mu^*,\xi,W)$ $Leb_{[0,T]}\otimes\prob$-almost surely.
Since uniqueness in law holds by Theorem \ref{teorema_di_unicita_legge}, it follows that
\begin{equation}\label{existence:main_thm:uguaglianza_leggi}
    \prob\circ(X^*,(\delta_{\lambda^*_t}(da))_{t \in [0,T]},\mu^*)^{-1}= \hat{\prob}\circ(\hat{X},(\delta_{\hat{\alpha}(t,\hat{X}_t,\hat{\mu})}(da))_{t \in [0,T]},\hat{\mu})^{-1},
\end{equation}
which implies that the consistency condition \eqref{def_mean_field_sol:cons} is satisfied.

\smallskip
Finally, we verify that the correlated flow just defined satisfies the optimality condition \eqref{def_mean_field_sol:opt}.
For any $\beta \in \A$, let $\Sigma \in \mathcal{Q}$ be as in point \ref{existence:prop_relazione_mfg:deviazioni} of Proposition \ref{existence:prop_relazione_mfg}.
Then, by \eqref{existence:main_thm:uguaglianza_leggi}, \eqref{existence:main_theorem:strict_ctrl} and \eqref{existence:main_thm:condizione_min_max}, for every $\Sigma \in \mathcal{Q}$ it holds
\begin{equation*}
\begin{aligned}
    \J & (\Lambda^*,\mu^*) = \E^{\hat{\prob}}\quadre{\int_0^T f(t,\hat{X}_t,\hat{\mu}_t,\hat{\alpha}(t,\hat{X}_t,\hat{\mu})) dt + g(\hat{X}_T,\hat{\mu}_T)} \\
    & \leq \E^{\hat{\prob}}\quadre{\int_0^T \int_A f(t,\hat{X}_t,\hat{\mu}_t,a)\hat{q}_t(\hat{X}_t,\hat{\mu}) dt + g(\hat{X}_T,\hat{\mu}_T)} = \int \mathfrak{F} (y,q,m) \hat{\Gamma}^*(dy,dq,dm) \\
    & \leq \int \mathfrak{F} (y,q,m) \Sigma(dy,dq,m)\rho(dm) = \J(\beta,\mu^*),
\end{aligned}
\end{equation*}
which proves that $(\Lambda^*,\mu^*)$ satisfies the optimality condition and therefore is a mean field CCE.

\begin{remark}\label{existence:rmk_linear_programming}
As described in this section, the first step to find a coarse correlated solution is to find an optimal strategy for the maximizing player in the auxiliary zero-sum game in Definition \ref{existence:def_zerosum}, which exists by Theorem \ref{existence:thm_esistenza_strategia_ottima}. Since this optimization problem is given by a linear payoff functional defined over a convex set of probability measures, linear programming methods could be employed to find approximate solutions.
The study of computational methods for finding coarse correlated solutions to the \mbox{MFG} is beyond the scope of this work. We refer to \cite{jiangleyton2015,paparough2008} for a linear programming approach to the computation of (true) correlated equilibria in $N$-player games.
\end{remark}

\section{An example of coarse correlated solution to the mean field game}\label{sezione_esempio}
Taking inspiration from the work of Bardi and Fischer \cite{bardi_fischer2019} and Lacker's papers \cite{lacker2016,lacker2020convergence}, we provide a simple example of a mean field game possessing mean field CCEs with non-deterministic flow of measures $\mu$.
Consistently with \cite{bardi_fischer2019,lacker2016,lacker2020convergence}
but differently from the rest of the paper, we consider a payoff, to be maximized, instead of a cost.

The MFG is as follows: We consider $d=1$, $A=[a,b]$, with $a < 0 < b$, and $\nu=\delta_0$.
For $m \in \probmeasures{\R}$, denote by $\bar{m}$ its mean $\int_{\R} y m(dy)$.
Consider the following coefficients and profit functions:
\begin{equation*}
    \begin{aligned}
    & b(t,x,m,a)=a, && f(t,x,m,a)=0, &&& g(x,m)=cx\bar{m},
    \end{aligned}
\end{equation*}
with $c>0$ a positive constant.
Observe that they satisfy the requirements of Assumptions \ref{standing_assumptions}.
We want to find a coarse correlated solution for the mean field game whose payoff functional, to be maximized, is given by
\begin{equation}\label{esempio_payoff}
    \J(\Lambda,\mu)=\E[cX_T \bar{\mu}_T],
\end{equation}
under the constraint
\begin{equation}\label{esempio_dinamiche}
    X_t = \int_0^t \lambda_s ds + W_t, \quad 0\leq t \leq T,
\end{equation}
where $\lambda$ is the strategy associated to an admissible recommendation $\Lambda$ in the sense of \eqref{mf:controllo_indotto}.

The rest of the section is organised as follows: In Section \ref{sec:example:computations}, we show that there exist infinitely many coarse correlated solutions to the MFG with non-deterministic flow of measures. Such solutions are neither classical MFG solutions nor a randomization (or a mixture, in the language of \cite{lacker2016,lacker2020convergence}) of the classical MFG solutions.
In Section \ref{sec:comparions_weak_mfg}, we compare coarse correlated solutions to the MFG with Lacker's weak solutions to MFG without common noise, as they may feature a random flow of measures as well.
By the study of this specific MFG, we show that there exist infinitely many mean field CCEs which are not weak solutions to the MFG without common noise.
Moreover, we show that weak MFG solutions which satisfy an additional measurability constraint are indeed coarse correlated solution to the MFG as well.

\subsection{Exhibiting explicit coarse correlated solutions}\label{sec:example:computations}
Set $\Omega^0=\{1,2\}^2$, $\F^{0-}=2^{\Omega^0}$ the power set and, given some probability measure $\prob^0 \in \mathcal{P}(\Omega^0,\F^{0-})$, we set $\prob^0((i,j))=p_{i,j}$, so that $p_{i,j}\geq 0$ for all $i,j$ and $\sum_{i,j=1}^2 p_{i,j}=1$.
Consider the following open loop strategies and flows of measures:
\begin{equation}
    \begin{aligned}
        u^+_t(\omega_*)\equiv b, \qquad \mu^+ = (\prob^* \circ (tb + W_t)^{-1})_{t \in [0,T]}; \\
        u^-_t(\omega_*)\equiv a, \qquad \mu^- = (\prob^* \circ (ta + W_t)^{-1})_{t \in [0,T]}.
    \end{aligned}
\end{equation}
It was shown in \cite{bardi_fischer2019} that the pairs $(u^+,\mu^+)$ and $(u^-,\mu^-)$ are two non-equivalent open-loop solutions of the mean field game, with initial distribution $\nu=\delta_{0}$, where by ``non-equivalent'' we mean that the flows of measures $\mu^+$ and $\mu^-$ do not coincide.
We point out that this result holds for more general initial distributions $\nu\in\mathcal{P}(\R)$, see \cite[Definition 3.1 and Theorem 3.1]{bardi_fischer2019}.
Choose $a_1,a_2 \in [0,1]$, and set:
\begin{equation}\label{esempio:def_strategie_flussi}
    \begin{aligned}
        \mu^1=\tonde{a_1 \mu^+_t + (1-a_1)\mu^-_t}_{t \in [0,T]}, \\
        \mu^2=\tonde{a_2 \mu^+_t + (1-a_2)\mu^-_t}_{t \in [0,T]}.
    \end{aligned}
\end{equation}
Define $(\Lambda,\mu)$ in the following way:
\begin{equation}\label{esempio_corr_sol}
    (\Lambda,\mu)\tonde{(i,j)}=\begin{cases}
    (u^+,\mu^1) \quad (i,j)=\tonde{1,1}, \\
    (u^+,\mu^2) \quad (i,j)=\tonde{1,2}, \\
    (u^-,\mu^1) \quad (i,j)=\tonde{2,1}, \\
    (u^-,\mu^2) \quad (i,j)=\tonde{2,2}.
    \end{cases}
\end{equation}
We claim that, as long as $a<0<b$, for every $T,c>0$ there exists a probability measure $(p_{i,j})_{i,j=1,2}$ and a suitable choice of the parameters $(a_i)_{i=1,2}$ so that the tuple $((\Omega^0,\F^{0-},\prob^0),\Lambda,\mu)$ is a coarse correlated solution of the mean field game according to Definition \ref{def_mean_field_sol}.

\smallskip
First of all, as shown in Example \ref{mf:example:admissible_recommendations} in Section \ref{sezione_formulazione_mfg}, since $\Lambda$ takes only two values, it is admissible.
Therefore, the tuple $((\Omega^0,\F^{0-},\prob^0),\Lambda,\mu)$ is a correlated flow.

\smallskip
Let us begin with the consistency condition.
We first observe that, when the state equation is controlled by $u^+$ (respectively, $u^-$), the law of the state process at time $t$, $X_t$, is exactly $\mu^+_t$ (respectively, $\mu^-_t$), for every time $t \in [0,T]$.

Suppose that $p_{1,1}+p_{2,1}$ and $p_{1,2}+p_{2,2}$ are both strictly positive.
Then, observe that
\begin{equation*}
    \prob(X_t \in \cdot \; \vert \; \mu)(\omega_0)=\prob(X_t \in \cdot \; \vert \; \mu)(\omega_0)=\begin{cases}
    \prob(X_t \in \cdot \; \vert \;  \mu=\mu^1) \quad   \text{if } \mu(\omega_0)=\mu^1, \\
    \prob(X_t \in \cdot \; \vert \;  \mu=\mu^2) \quad  \text{if } \mu(\omega_0)=\mu^2.
    \end{cases}
\end{equation*}
We can compute explicitly such a conditional probability. Fix $A\in \boreliani{\R^d}$:
\begin{equation*}
    \begin{cases}
    \prob(X_t \in A \; \vert \; \mu=\mu^1) \\
    \prob(X_t \in A \; \vert \; \mu=\mu^2)
    \end{cases}  = \quad \begin{cases}
    \frac{p_{1,1}}{p_{1,1}+p_{2,1}}\prob(X^+_t \in A) +  \frac{p_{2,1}}{p_{1,1}+p_{2,1}}\prob(X^-_t \in A) \quad \text{if } \mu(\omega_0)=\mu^1, \\
    \frac{p_{1,2}}{p_{1,2}+p_{2,2}}\prob(X^+_t \in A) +  \frac{p_{2,2}}{p_{1,2}+p_{2,2}}\prob(X^-_t \in A) \quad  \text{if } \mu(\omega_0)=\mu^2.
    \end{cases}
\end{equation*}
In order to satisfy the consistency condition, it must hold
\begin{equation*}
    \begin{cases}
    \frac{p_{1,1}}{p_{1,1}+p_{2,1}}\prob(X^+_t \in A) +  \frac{p_{2,1}}{p_{1,1}+p_{2,1}}\prob(X^-_t \in A)=\mu^1_t(A) \\
    \frac{p_{1,2}}{p_{1,2}+p_{2,2}}\prob(X^+_t \in A) +  \frac{p_{2,2}}{p_{1,2}+p_{2,2}}\prob(X^-_t \in A) = \mu^2_t(A)
    \end{cases}
\end{equation*}
for every $A\in \boreliani{\R^d}$.
By definition of $\mu^1$ and $\mu^2$,
\begin{equation*}
    \begin{cases}
    \frac{p_{1,1}}{p_{1,1}+p_{2,1}}\mu^+_t(A) +  \frac{p_{2,1}}{p_{1,1}+p_{2,1}}\mu^-_t(A)=a_1\mu^+_t(A) + (1-a_1)\mu^-_t(A), \\
    \frac{p_{1,2}}{p_{1,2}+p_{2,2}}\mu^+_t(A) +  \frac{p_{2,2}}{p_{1,2}+p_{2,2}}\mu^-_t(A) = a_2 \mu^+_t(A) + (1-a_2)\mu^-_t(A)
    \end{cases}
\end{equation*}
which holds if and only if
\begin{equation}\label{esempio_condizione_flussi_misure}
    \begin{cases}
    \frac{p_{1,1}}{p_{1,1}+p_{2,1}} = a_1, \\
    \frac{p_{1,2}}{p_{1,2}+p_{2,2}} = a_2.
    \end{cases}
\end{equation}
We can regard \eqref{esempio_condizione_flussi_misure} as the consistency condition.

\smallskip
We now turn our attention to the optimality condition.
Consider $\gamma = \prob \circ (\Lambda,\mu)^{-1} = \prob^0 \circ (\Lambda,\mu)^{-1} \in \probmeasures{\A \times \contpdue}$ and $\rho=\prob\circ\mu^{-1}=\prob^0\circ\mu^{-1} \in \probmeasures{\contpdue}$.
Since $(\Lambda,\mu)$ takes finitely many values and it is independent of $W$ under the product measure $\prob = \prob^0 \otimes \prob^*$, we can  rewrite the optimality condition using disintegration of measures as
\begin{equation*}
    \int_{\A\times\pwassspace{2}{\contrd}}  \J(\alpha,m) \gamma(d\alpha,dm) \geq \int_{\pwassspace{2}{\contrd}} \J(\beta,m) \rho(dm) \quad \forall \, \beta \in \A,
\end{equation*}
where
\[
\J(\theta,m) = \E[c X_T \bar{m}_T], \quad     X_t =  \int_0^t \theta_s ds + W_t, \quad 0\leq t \leq T,
\]
for $m$ in $\mu(\Omega^0):=\{ \mu^1,\mu^2 \}$ and $\theta=\alpha$ in $\Lambda(\Omega^0):=\{u^+,u^-\} \subseteq \A$ on the left-hand side of the inequality above and $\theta=\beta$ in $ \A$ on the right-hand side.
We rewrite explicitly the inequality as
\begin{equation}\label{esempio_funzionale2}
    \begin{aligned}
    \J(\Lambda,\mu) - \J(\beta,\mu) & =  p_{1,1}\tonde{\J(u^+,\mu^1) - \J(\beta,\mu^1)}  + p_{1,2}\tonde{\J(u^+,\mu^2) - \J(\beta,\mu^2)} \\ 
    & + p_{2,1}\tonde{\J(u^-,\mu^1) - \J(\beta,\mu^1)} + p_{2,2}\tonde{\J(u^-,\mu^2) - \J(\beta,\mu^2)} \geq 0.
    \end{aligned}
\end{equation}
Therefore, using \eqref{esempio_payoff}, we have
\begin{equation*}
    \begin{aligned}
    \J(\Lambda,\mu) - \J(\beta,\mu) & = p_{1,1}\tonde{ cT^2b(a_1b + (1-a_1)a)  - c M(\beta)T(a_1b + (1-a_1)a)}  \\ 
    & +   p_{1,2}\tonde{ cT^2 b(a_2b + (1-a_2)a)  - c M(\beta)T(a_2b + (1-a_2)a)} \\
    & + p_{2,1}\tonde{ cT^2 a(a_1b + (1-a_1)a)  - c M(\beta)T(a_1b + (1-a_1)a)} \\
    & +  p_{2,2}\tonde{ cT^2a(a_2b + (1-a_2)a)  - c M(\beta)T(a_2b + (1-a_2)a)},
    \end{aligned}
\end{equation*}
where $M(\beta):=\E[\int_0^T\beta_t dt]=\E[X^\beta_T]$.
We can set $m(\beta):=\sfrac{1}{T}M(\beta)=\sfrac{1}{T}\E[\int_0^T\beta_t dt]$. Observe that $m(\beta) \in [a,b]$, being the mean of an $A$-valued process, and $m(\A)=[a,b]$, since for every $c \in [a,b]$ the constant process $\beta\equiv c$ belongs to $\A$.
We divide by $cT^2$ to obtain the following condition:
\begin{equation}\label{esempio_funzionale3}
    \begin{aligned}
    p_{1,1} & \tonde{ b(a_1b + (1-a_1)a)  - m(\beta)(a_1b + (1-a_1)a)}  \\ 
    &+  p_{1,2}\tonde{ b(a_2b + (1-a_2)a)  - m(\beta)(a_2b + (1-a_2)a)} \\
    &+ p_{2,1}\tonde{a(a_1b + (1-a_1)a)  - m(\beta) (a_1b + (1-a_1)a)} \\
    &+  p_{2,2}\tonde{ a(a_2b + (1-a_2)a)  - m(\beta) (a_2b + (1-a_2)a)} \\
    & \geq 0.
    \end{aligned}
\end{equation}
The condition above can be seen as a positivity condition for a real affine function of $g(m)$, $m \in [a,b]$, i.e.
\begin{equation}\label{esempio_condizione_inf}
    \begin{aligned}
        & \inf_{m \in [a,b]} g(m)=\inf_{m \in [a,b]} h((p_{i,j})_{i,j=1,2} , (a_i)_{i=1,2};a,b) m + k((p_{i,j})_{i,j=1,2} , (a_i)_{i=1,2};a,b) \geq 0 \\
        & \begin{cases}
        h((p_{i,j})_{i,j=1,2} , (a_i)_{i=1,2};a,b)= & - \left\{p_{1,1}(a_1b+(1-a_1)a) + p_{1,2}(a_2 b+(1-a_2)a) \right.\\ 
        & \left. + p_{2,1}(a_1b+(1-a_1)a) + p_{2,2}(a_2b+(1-a_2)a)  \right\}, \\
        k((p_{i,j})_{i,j=1,2} , (a_i)_{i=1,2};a,b) =  & p_{1,1}b(a_1b+(1-a_1)a) + p_{1,2}b(a_2 b+(1-a_2)a)\\ 
        & + p_{2,1}a(a_1b+(1-a_1)a) + p_{2,2}a(a_2b+(1-a_2)a) . \\
        \end{cases}
    \end{aligned}
\end{equation}
We now impose the consistency condition \eqref{esempio_condizione_flussi_misure} to get:
\begin{equation}\label{esempio_coefficienti}
    \begin{aligned}
    h((p_{i,j})_{i,j=1,2} ;a,b)  & = - b \tonde{\frac{p_{1,1}^2 + p_{2,1}p_{1,1}}{p_{1,1} + p_{2,1}} + \frac{p_{1,2}^2 + p_{1,2}p_{2,2}}{p_{1,2} + p_{2,2}} }  \\
    &\quad  -a \tonde{\frac{p_{2,1}^2 + p_{2,1}p_{1,1}}{p_{1,1} + p_{2,1}} + \frac{p_{2,2}^2 + p_{1,2}p_{2,2}}{p_{1,2} + p_{2,2}} }, \\
    k((p_{i,j})_{i,j=1,2} ;a,b) & =  b^2 \tonde{\frac{p_{1,1}^2}{p_{1,1} + p_{2,1}} + \frac{p_{1,2}^2}{p_{1,2} + p_{2,2}} } +a^2 \tonde{\frac{p_{2,1}^2}{p_{1,1} + p_{2,1}} + \frac{p_{2,2}^2 }{p_{1,2} + p_{2,2}} } \\
    & \quad + 2ab\tonde{\frac{p_{1,1}p_{2,1}}{p_{1,1}+p_{2,1}} + \frac{p_{1,2}p_{2,2}}{p_{1,2}+p_{2,2}}}.
    \end{aligned}
\end{equation}
Observe that imposing the consistency condition \eqref{esempio_condizione_flussi_misure} reduces the number of parameters but makes the problem nonlinear in the probabilities $(p_{i,j})_{i,j=1,2}$.

\medskip
Looking at \eqref{esempio_condizione_inf} and \eqref{esempio_coefficienti}, we observe that it implies that every randomization of the open loop MFG solutions $(u^+,\mu^+)$ and $(u^-,\mu^-)$ is a coarse correlated solution to the MFG as well, and it covers the case treated in \cite{bardi_fischer2019}, for any choices of $a < 0 < b$.
To see this, consider a probability measures $\prob^0=(p_{i,j})_{i,j=1,2}$ so that $p_{1,2}=p_{2,1}=0$ and, therefore, $p_{2,2}=1-p$ and $p_{1,1} = p \in [0,1]$.
For any $p$, such probability measure is a randomization of the MFG solutions $(u^+,\mu^+)$ and $(u^-,\mu^-)$.
Equations \eqref{esempio_coefficienti} take the simpler form
\begin{equation}\label{esempio_coefficienti_diagonale}
\begin{aligned}
    h((p,0,0,1-p); a,b) & = - bp -a(1-p),  \\
    k((p,0,0,1-p); a,b) &= b^2 p +a^2(1-p).
\end{aligned}
\end{equation}
If $h((p,0,0,1-p);a,b) \geq 0$, then condition \eqref{esempio_condizione_inf} becomes
\begin{equation*}
    \inf_{ m \in [a,b]} -bm + b^2 = a^2(1-p) -ab(1-p) \geq 0,
\end{equation*}
which is satisfied if $a<0$ and $b>0$, for any $p \in [0,1]$.
If instead $h((p,0,0,1-p);a,b) < 0$, we have
\begin{equation*}
    \inf_{ m \in [a,b]} -bm + b^2 = b^2(1-p) - ab(1-p) \geq 0,
\end{equation*}
which is again satisfied if $a<0$ and $b>0$, for every $p \in [0,1]$.
Observe that, when $p=1$, $\mu\equiv\mu^1=\mu^+$, while, when $p=0$, $\mu\equiv\mu^2=\mu^-$.
This shows that the deterministic correlated flows $(\Lambda,\mu)\equiv (u^+,\mu^+)$ and $(\Lambda,\mu)\equiv (u^-,\mu^-)$ are indeed mean field CCE in the sense of Definition \ref{def_mean_field_sol}.

\smallskip
Turning to more interesting cases, consider $[a,b]=[-1,1]$.
The choice of a symmetric interval is not necessary, but it has been made to ease the comparison with previous results in the literature (see the next subsection).
Figure \ref{fig:esempio_immagine} shows the existence of coarse correlated mean field equilibria as the probability measure $(p_{i,j})_{i,j=1,2}$ varies.
Yellow spots in Figure \ref{fig:esempio_immagine} refer to those probability measures on $(\Omega^0,\F^{0-})$ so that $(\Lambda,\mu)$ is indeed a mean field CCE.
In particular, it shows the existence of infinitely many mean field CCEs for the system.
Observe that there exist infinitely many coarse correlated solutions of the mean field game so that $\Lambda$ is not a deterministic function of $\mu$, i.e., they are not a randomization of the solutions $(u^+,\mu^+)$ and $(u^-,\mu^-)$: they correspond to those probability measures $(p_{i,j})_{i=1}^2$ so that $p_{1,2} \cdot p_{2,1} > 0$.
Referring to Figure \ref{fig:esempio_immagine}, they correspond to yellow points in the interior of any of the three pseudo-simplexes.

\begin{figure}
\centering
\includegraphics[scale=0.2]{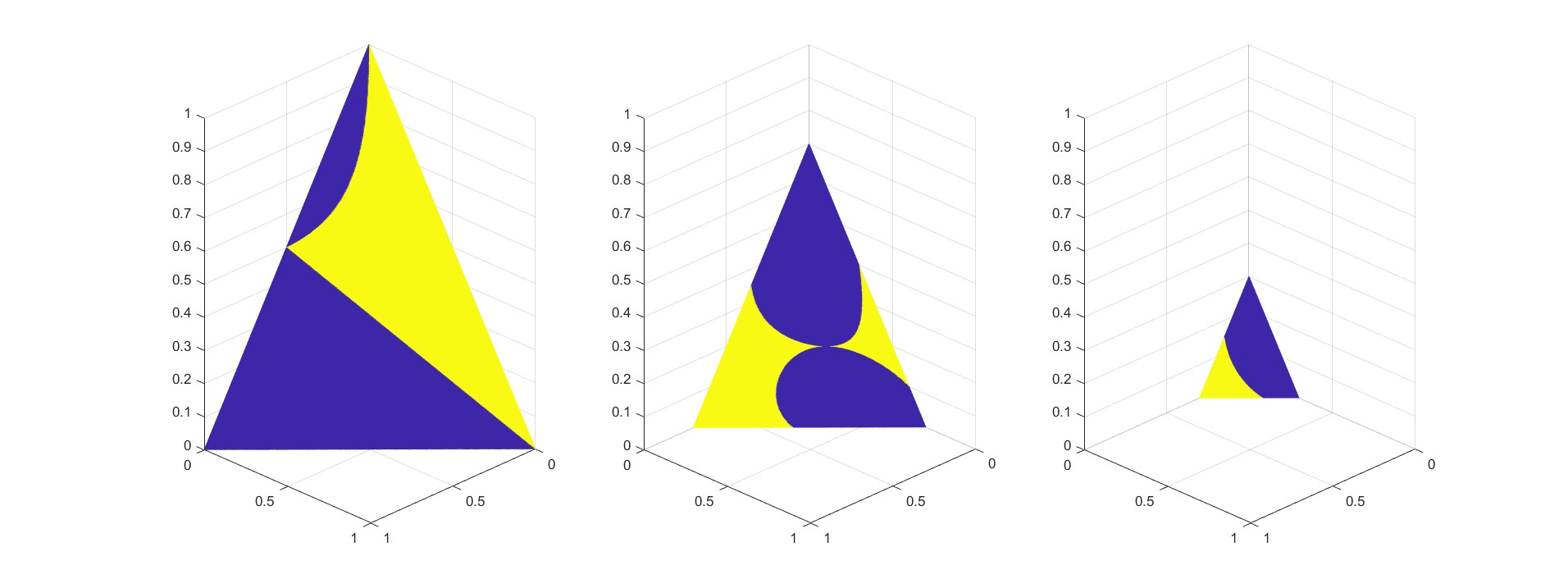}
    \caption{Existence of mean field CCEs as the probability measure $(p_{i,j})_{i,j=1,2}$ varies.
    Yellow points correspond to the values of $(p_{i,j})$ associated to coarse correlated solutions.
    The complement set of CCEs in each pseudo-simplex is indicated in blue.
    We fix $p_{2,1} \in [0,1)$ and let the other parameters vary, so that every point $(p_{1,1},p_{1,2},p_{2,2})$ in each pseudo 2-dimensional simplex is such that $p_{1,1} + p_{1,2} + p_{2,2} = 1-p_{2,1}$, $p_{i,j} \geq 0$.
    We consider $[a,b]=[-1,1]$ and $p_{2,1}$ equal to $0.0$, $0.3$, and $0.7$, respectively.
    } \label{fig:esempio_immagine}
\end{figure}

\subsection{Comparison with weak mean field game solutions without common noise of \cite{lacker2016}}\label{sec:comparions_weak_mfg}

Consider $A=[-1,1]$, $T=2$.
With this choice of control actions and time horizon, the example we proposed matches the setting of Lacker's \say{illuminating example} of \cite[Section 3.3]{lacker2016}.
We show that there exists a coarse correlated solution of the MFG which is not a weak MFG solution without common noise as defined in Definition 3.1 therein.
In particular, the most important feature is the fact that the recommendation $\Lambda$ can not be expressed as a deterministic function of the flow of measures.
On the other hand, we show that, under an additional assumption, weak MFG solutions induce coarse correlated solutions to the MFG.

\smallskip
To be consistent with the notation and the setup of Lacker's paper, we use the notion of relaxed controls, which are used extensively in Section \ref{sezione_existence} (see, in particular, Section \ref{existence:sezione_crtl_rilassati} for definitions, notation and some important properties).
Let $(p_{i,j})_{i,j=1,2}$ be so that $p_{1,1}+p_{2,1}$ and $p_{1,2}+p_{2,2}$ are strictly positive.
We introduce the relaxed controls $\delta^+=(\delta^+_t)_{t \in [0,T]}$ and $\delta^-=(\delta^-_t)_{t \in [0,T]}$, by setting
\begin{equation*}
\begin{aligned}
    & \delta^+_t(\omega_*;da)=\delta_{u^+_t(\omega_*)}(da) \equiv \delta_{1}(da), \quad \forall t \in [0,T], \omega_* \in \Omega^*,\\
    & \delta^-_t(\omega_*;da)=\delta_{u^-_t(\omega_*)}(da) \equiv \delta_{-1}(da), \quad \forall t \in [0,T], \omega_* \in \Omega^*.
\end{aligned}
\end{equation*}
Consider the correlated flow $(\Lambda,\mu)$ defined by \eqref{esempio_corr_sol}
and observe that the strategy $\lambda=(\lambda_t)_{t \in [0,T]}$ associated to the admissible recommendation $\Lambda$ can be rewritten as a relaxed control as
\begin{equation}\label{raccomandazione_come_lacker}
    \mathfrak{r}_t(\omega;da)= \mathfrak{r}_t(\omega_0,\omega_*;da)=\1_{\{\Lambda=u^+\}}(\omega_0)\delta^+_t(da) + \1_{\{\Lambda=u^-\}}(\omega_0)\delta^-_t(da).
\end{equation}
We point out that $\mathfrak{r}$ does not depend on $\omega_*$ since $\delta^+$ and $\delta^-$ do not depend on $\omega_*$.
Starting from $(\Lambda,\mu)$, we define a random variable $\Tilde{\mu}$ with values in $\mathcal{P}(\contrd \times \mathcal{V} \times \contrd)$ by setting
\begin{equation}\label{esempio_comparison_flusso}
    \Tilde{\mu}(\cdot)=\prob((W,\mathfrak{r},X) \in \cdot \; \vert \; \mu ).
\end{equation}
We observe that $\sigma(\mu)=\sigma(\Tilde{\mu})$: we have $\sigma(\Tilde{\mu}) \subseteq \sigma(\mu)$ since, by definition of regular conditional probability, $\Tilde{\mu}$ must be $\sigma(\mu)$ measurable; to get the opposite inclusion, for every $t \in [0,T]$, let $\Tilde{\mu}^x_t$ be the push forward of $\Tilde{\mu}$ through the map $\contrd \times \mathcal{V} \times \contrd \ni (w,q,x) \mapsto x_t \in \R^d$.
Then, by exploiting the consistency condition \eqref{def_mean_field_sol:cons}, we have 
\begin{equation*}
    \Tilde{\mu}^x_t(A)=\Tilde{\mu}(\{x \in \contrd: \; x_t \in A\})=\prob(X_t \in A \; \vert \; \mu ) = \mu_t(A),
\end{equation*}
for every $A \in \boreliani{\R^d}$, i.e. $\Tilde{\mu}^x_t=\mu_t$ $\prob$-a.s, for every $t \in [0,T]$.
Let $(\mathcal{B}_{t,\contrd})_{t \in [0,T]}$ be the natural filtration of the identity process on $\contrd$, i.e. $\mathcal{B}_{t,\contrd}=\sigma(\contrd \ni x \mapsto x_s \in \R^d: \; 0 \leq s \leq t)$, and let $(\F^{\Tilde{\mu}}_t)_{t \in [0,T]}$ be the natural filtration of $\Tilde{\mu}$, that is
\begin{equation*}
    \F^{\Tilde{\mu}}_t=\sigma(\Tilde{\mu}\tonde{C}: \; C \in \mathcal{B}_{t,\contrd}\otimes \F^{\mathcal{V}}_t \otimes \mathcal{B}_{t,\contrd}).
\end{equation*}
We observe that, for every $t \in (0,T]$, we have $\F^{\Tilde{\mu}}_t=\sigma(\mu)$.
To see this, observe that
\begin{equation*}
    \sigma(\mu) \supseteq \F^{\Tilde{\mu}}_t \supseteq \sigma(\Tilde{\mu}^x_s: \; s \leq t) = \sigma(\mu_s: \; s \leq t) = \sigma(\mu),
\end{equation*}
where the last equality holds for every $t>0$, as can be verified by explicit calculations.
Finally, for $t=0$, we have $\F^{\Tilde{\mu}}_0=\{\emptyset,\Omega^0\}$.
Having established the relations between such $\sigma$-algebras, it is straighforward to verify that the tuple $((\Omega,\F,\mathbb{F},\prob),W,\Tilde{\mu},\mathfrak{r},X)$ satisfies properties (1-4) and (6) of \cite[Definition 3.1]{lacker2016}.
Now, pick a probability measure $\prob^0$ so that $\min(p_{1,2},p_{2,1}) > 0$ and $\bar{\mu}^1_T>0$, $\bar{\mu}^2_T<0$.
Figure \ref{fig:esempio_immagine} shows that such a choice is possible (actually, there exist infinitely many measures $\prob^0$ with the desired property).
For such a choice of $\prob^0$, the relaxed control $\mathfrak{r}$ does not satisfy the optimality condition (5) of \cite[Definition 3.1]{lacker2016}, since, as shown in \cite[Section 3.3]{lacker2016}, every optimal control $\mathfrak{r}^*$ must be of the form $\mathfrak{r}^*_t(da)(\omega)=\delta_{\alpha^*_t(\omega)}(da)$ for $Leb_{[0,T]}$-a.e. $t$, with
\begin{equation*}
    \alpha^*_t = \text{sign}\tonde{\attesa{\Tilde{\mu}^x_T \; \vert \; \F^{\Tilde{\mu}}_t}}.
\end{equation*}
Here, $\text{sign}\tonde{0}=0$.
Since $\F^{\Tilde{\mu}}_0$ is trivial and $\F^{\Tilde{\mu}}_t=\sigma(\mu)$ for $t > 0$, the optimal control $\alpha^*_t$ must be equal to
\begin{equation}\label{esempio_comparison_ctrl_ottimo}
    \alpha^*_t(\omega)=\alpha^*_t(\omega_0)=\begin{cases}
    -1 \quad \text{ if } \bar{\mu}_T(\omega_0)<0, \\
    1 \qquad \text{ if } \bar{\mu}_T(\omega_0)>0,
    \end{cases} \quad 0 < t \leq T,
\end{equation}
and equal to an arbitrary value at $t=0$.
In particular, observe that such a control is a deterministic function of the measure $\Tilde{\mu}$.
For every $\prob^0$ so that $p_{1,2}+p_{2,1}>0$, this is not the case of the correlated flow $(\Lambda,\mu)$ defined in \eqref{esempio_corr_sol}, since $\Lambda$ is not a deterministic function of $\mu$.

\smallskip
The essential reason for the lack of optimality, in the sense of Lacker, of the relaxed control $\mathfrak{r}$ defined by \eqref{raccomandazione_come_lacker} resides in the differences between allowed deviations: on the one hand, for weak mean field games solutions in the sense of \cite{lacker2016}, all adapted compatible controls $\mathfrak{b}=(\mathfrak{b}_t)_{t \in [0,T]}$ are allowed, where ``compatible'' means that $\sigma(\mathfrak{b}_s: s \leq t)$ is conditionally independent of $\F^{\xi,\Tilde{\mu},W}_T$ given $\F^{\xi,\Tilde{\mu},W}_t$ for every $t$, which leads to a very rich class of controls.
On the other hand, for coarse correlated solution of the MFG, only $\mathbb{F}^*$-progressively measurable strategies are allowed as deviations.
Therefore, many more solutions exist.

\smallskip
Lastly, we show that under an additional assumption on the measurability of the random measure, it is indeed possible to define a mean field CCE starting from weak MFG solution without common noise.
This property is not a-priori granted, due to the difference between the respective consistency conditions.
Let $\Tilde{\mu}$ be a weak MFG solution without common noise.
Let $\mu_t$ be the push forward of $\Tilde{\mu}$ through the map $\contrd \times \mathcal{V} \times \contrd \ni (w,q,x) \mapsto x_t \in \R^d$.
Define a random flow of measures by setting $\mu=(\mu_t)_{t \in [0,T]}$.
Assume that the flow of measures $\mu$ carries the same information as the random measure $\Tilde{\mu}$, i.e.
\begin{equation}\label{esempio:hp_filtrazioni}
    \sigma(\mu_s: \; 0 \leq s \leq t)=\F^{\Tilde{\mu}}_t, \quad \forall t \in [0,T].
\end{equation}
If a weak MFG solution $\Tilde{\mu}$ satisfies condition \eqref{esempio:hp_filtrazioni}, then $\Tilde{\mu}$ does induce a mean field CCE.
Indeed, set $\rho=\prob\circ\mu^{-1}$.
By \eqref{esempio:hp_filtrazioni}, we have $\mu_t=\prob(X_t \in \cdot \; \; \vert \; \Tilde{\mu}) = \prob(X_t \in \cdot \; \; \vert \; \mu)$, i.e. the consistency condition \eqref{def_mean_field_sol:cons} is satisfied.
Moreover, the assumption on equality of the filtrations ensures that there exists a progressively measurable function $\varphi:[0,T]\times\contpdue \to A$ so that
\begin{equation*}
    \alpha^*_t=\text{sign}\tonde{\E\quadre{\bar \mu_T \; \vert \; \F^{\Tilde{\mu}}_t}}=\varphi\tonde{t,\mu}.
\end{equation*}
Then, we define $(\Omega^0,\F^{0-},\prob^0)$ and $(\Lambda^*,\mu^*)$ as
\begin{equation}
    \begin{aligned}
        & (\Omega^0,\F^{0-},\prob^0)=\tonde{\contpdue,\boreliani{\contpdue},\rho}, \\
        & \mu^*=\text{Id}:\tonde{\contpdue,\boreliani{\contpdue},\rho} \to \tonde{\contpdue,\boreliani{\contpdue},\rho}, \\
        & \begin{aligned}
        \Lambda^*: \tonde{\contpdue,\boreliani{\contpdue},\rho} & \to (\A,\boreliani{\A}) \\
        m & \mapsto \Lambda^*(m)=(\varphi(t,m))_{t \in [0,T]}.
        \end{aligned}
    \end{aligned}
\end{equation}
By Lemma \ref{esempi:lemma_misurabile}, the tuple $((\Omega^0,\F^{0-},\prob^0),\Lambda^*,\mu^*)$ is a correlated flow.
Let $X^*$ be the solution of \eqref{esempio_dinamiche} on the product probability space $(\Omega,\F,\mathbb{F},\prob)$ defined in \eqref{mf:condizione_ammissibilita}.
Since uniqueness in law holds by Theorem \ref{teorema_di_unicita_legge}, it follows that $(X^*,\mu^*)$ has the same joint law as $(X,\mu)$, which implies that the consistency condition \eqref{def_mean_field_sol:cons} is satisfied.
Since $\lambda^*_t=\varphi(t,\mu^*)$, $(\Lambda^*,\mu^*)$ satisfies optimality condition \eqref{def_mean_field_sol:opt} as well and therefore it is a mean field CCE.

We observe that the additional assumption on the filtrations \eqref{esempio:hp_filtrazioni} is satisfied both by the weak MFG solution exhibited in \cite[Proposition 3.7]{lacker2016} and in our case, as shown above.
We point out that this CCE has been already considered: suppose that the flow of measures as law $\rho=p\delta_{\mu^+} + (1-p)\delta_{\mu^-}$, $p \in (0,1)$, for $\mu^+$ and $\mu^-$ given by \eqref{esempio:def_strategie_flussi}.
Then, the correlated flow $(\Lambda^*,\mu^*)$ corresponds to the probability measures $\prob^0$ so that $p_{1,2}=p_{2,1}=0$, $p_{1,1}=p$ and $p_{2,2}=1-p$, which, as shown, always satisfy the condition \eqref{esempio_condizione_inf}.
Roughly speaking, it corresponds to the case when $\Lambda^*=\phi(\mu^*)$ $\prob^0$-a.s., for some deterministic measurable $\phi$.

\appendix
\addcontentsline{toc}{section}{Appendices}

\section{Relaxed controls}\label{existence:sezione_crtl_rilassati}
Here, we recall some facts on relaxed controls.
Denote by $\mathcal{V}$ the set of positive measures $q$ on $[0,T] \times A$ so that the time marginal is equal to the Lebesgue measure, i.e., $q ([s,t]\times A)=t-s$ for every $0 \leq s \leq t \leq T$.
We endow $\mathcal{V}$ with the topology of weak convergence of measures, which makes $\mathcal{V}$ a Polish space.
It is a well known fact that, when the set $A$ is compact, $\mathcal{V}$ is compact as well.
Moreover, for every measure $q \in \mathcal{V}$, there exists a measurable map $[0,T] \ni t \mapsto q_t \in \mathcal{P}(A)$ so that $q(da,dt)=q_t(da)dt$, with $(q_t)_{t \in [0,T]}$ unique up to $Leb_{[0,T]}$-a.e. equality.
We can equip the measurable space $(\mathcal{V},\boreliani{\mathcal{V}})$ with the filtration $(\F^{\mathcal{V}}_t)_{t \in [0,T]}$ defined by
\begin{equation*}
    \F^{\mathcal{V}}_t = \sigma( \mathcal{V} \ni q \mapsto q(C): \; C \in \boreliani{[0,t]\times A}).
\end{equation*}
We observe that $\F^{\mathcal{V}}_t$ is countably generated for every $t \in [0,T]$, by reasoning as in the proof of \cite[Proposition 7.25]{bertsekas_shreve}.
Finally, one can prove that there exists an $\mathcal{F}_t^{\mathcal{V}}$-predictable process $\overline{q}:[0,T]\times\mathcal{V}\to\mathcal{P}(A)$ such that, for each $q \in \mathcal{V}$, $\overline{q}(t,q)=q_t$ for a.e. $t \in [0,T]$ (see, e.g., \cite[Lemma 3.2]{lacker2015martingale}).
By an abuse of notation, we write $q_t(da)=\overline{q}(t,q)(da)$.

\smallskip
Consider a filtered probability space $(\Omega,\F,\mathbb{F},\prob)$.
A relaxed control $\mathfrak{r}$ is a $\mathcal{V}$-valued random variable.
We say that $\mathfrak{r}$ is $\mathbb{F}$-adapted if $\mathfrak{r}(C)$ is a real valued $\F_t$-measurable random variable for every $C \in \boreliani{[0,t]\times A}$.
Observe that every $A$-valued progressively measurable process $\alpha=(\alpha_t)_{t \in [0,T]}$, which is often referred to as strict control, induces a relaxed control by setting
\begin{equation*}
    \mathfrak{r}_t(da)dt=\delta_{\alpha_t}(da)dt.
\end{equation*}
Finally, using the map $\overline{q}$ described above, we can safely identify every $\mathbb{F}$-adapted relaxed control $\mathfrak{r}$ with the unique (up to $Leb_{[0,T]}$-a.e. equality) $\mathbb{F}$-progressively measurable process $(\mathfrak{r}_t)_{t \in [0,T]}$ with values in $\mathcal{P}(A)$ so that
\begin{equation*}
    \prob(\mathfrak{r}(da,dt)=\mathfrak{r}_t(da)dt)=1.
\end{equation*}
In the following, we will use mostly the notation $(\mathfrak{r}_t)_{t \in [0,T]}$ for a relaxed control and will make no distinction between a $\mathcal{V}$-valued random variable and a $\mathcal{P}(A)$-valued process.

\section{Weak and strong existence for controlled equations}\label{appendix_weak_strong_sol}

We state and prove a Yamada-Watanabe type result for stochastic differential equations with random coefficients as the ones encountered so far.
Recall from Section \ref{existence:sezione_crtl_rilassati} the definition of the space $\mathcal{V}$ and of relaxed controls.

Let $\mathfrak{U}=((\Omega,\F,\mathbb{F},\prob),\xi,W,\mu,\mathfrak{r})$ be a tuple composed by a filtered probability space satisfying usual assumptions, an $\mathbb{F}$-Brownian motion $W$, an $\R^d$-valued $\F_0$-measurable random variable $\xi$, an $\F_0$-measurable random flow of measures $\mu$ taking values in  $\contpdue$ and an $\mathbb{F}$-adapted $\mathcal{V}$-valued random variable $\mathfrak{r}$, in the sense that the random variables $\mathfrak{r}(C)$ are $\F_t$-measurable for every $C \in \boreliani{[0,t]\times A}$.
Let us consider the following stochastic differential equation:
\begin{equation}\label{weak_uniqueness:sde}
    dX_t=G(t,X_t,\mu,\mathfrak{r})dt + dW_t, \quad X_0=\xi.
\end{equation}
where $G:[0,T] \times \R^d \times \contpdue \times \mathcal{V} \to \R^d$ is jointly measurable and progressively measurable in $\mathcal{V}$; progressive measurability must be understood in the following sense: for every $q,q' \in \mathcal{V}$, for every $(t,x,m) \in [0,T] \times \R^d \times \contpdue$, it holds:
\begin{equation*}
    q(C)=q'(C) \; \;  \forall C \in \boreliani{[0,t] \times A} \; \Longrightarrow \; G(t,x,m,q)=G(t,x,m,q').
\end{equation*}

\begin{definition}[Strong solution and pathwise uniqueness]\label{weak_uniqueness:def_pathwise_uniqueness}
Let $\mathfrak{U}=((\Omega,\F,\mathbb{F},\prob),\xi,W,\mu,\mathfrak{r})$ be a tuple as above.
A strong solution to equation \eqref{weak_uniqueness:sde} on $\mathfrak{U}$ is a continuous $\mathbb{F}$-adapted process $X=(X_t)_{t \in [0,T]}$ adapted to the $\prob$-augmentation of $\mathbb{F}$ so that
\begin{equation}
    X_t=\xi+\int_0^t G(s,X_s,\mu,\mathfrak{r})ds + W_t, \quad 0 \leq t \leq T,
\end{equation}
holds $\prob$-almost surely.

Pathwise uniqueness holds for equation \eqref{weak_uniqueness:sde} if, given two strong solutions $X$ and $X'$ to \eqref{weak_uniqueness:sde} on $\mathfrak{U}$, they are indistinguishable:
\begin{equation*}
    \prob(X_t=X'_t \;\; \forall t \in [0,T])=1.
\end{equation*}
\end{definition}

\begin{definition}[Weak solution and uniqueness in law]\label{weak_uniqueness:def_uniqueness_in_law}
A weak solution to equation \eqref{weak_uniqueness:sde} is a tuple $\mathfrak{U}=((\Omega,\F,\mathbb{F},\prob),\xi,W,\mu,\mathfrak{r})$ as above so that there exists a continuous $\mathbb{F}$-adapted process $X=(X_t)_{t \in [0,T]}$ satisfying equation \eqref{weak_uniqueness:sde}.

Weak uniqueness holds for equation \eqref{weak_uniqueness:sde} if for any two weak solution of \eqref{weak_uniqueness:sde} $\mathfrak{U}^i$, $i=1,2$, so that $\prob^1\circ(\xi^1,W^1,\mu^1,\mathfrak{r}^1)^{-1}=\prob^2\circ(\xi^2,W^2,\mu^2,\mathfrak{r}^2)^{-1}$, it holds
\begin{equation*}
    \prob^1\circ(X^1,\xi^1,W^1,\mu^1,\mathfrak{r}^1)^{-1}=\prob^2\circ(X^2,\xi^2,W^2,\mu^2,\mathfrak{r}^2)^{-1},
\end{equation*}
where $X^i$ are the continuous $\mathbb{F}^i$-adapted processes that satisfy equation \eqref{weak_uniqueness:sde} on $\mathfrak{U}^i$, $i=1,2$.
\end{definition}

\begin{theorem}\label{teorema_di_unicita_legge}
Suppose pathwise uniqueness holds for equation \eqref{weak_uniqueness:sde}, in the sense of Definition \ref{weak_uniqueness:def_pathwise_uniqueness}.
Then, uniqueness in law in the sense of Definition \ref{weak_uniqueness:def_uniqueness_in_law} holds as well.
\end{theorem}

\begin{proof}
Let $\mathfrak{U}^1$ and $\mathfrak{U}^2$ be two weak solutions of equation \eqref{weak_uniqueness:sde} in the sense of Definition \ref{weak_uniqueness:def_uniqueness_in_law} above.
Since pathwise uniqueness holds for equation \eqref{weak_uniqueness:sde} by assumption, our goal is to bring together the solution on the same filtered probability space.
Let us define the following probability measures:
\begin{equation*}
\begin{aligned}
    & \hat{\Q}^i = \prob^i\circ(\xi^i,W^i,\mu^i,\mathfrak{r}^i,X^i)^{-1} \in \mathcal{P}(\R^d \times \contrd \times \contpdue \times \mathcal{V} \times \contrd), \quad i=1,2, \\
    & \Q = \prob^i\circ(\xi^i,W^i,\mu^i,\mathfrak{r}^i)^{-1} \in \mathcal{P}(\R^d \times \contrd \times \contpdue \times \mathcal{V}), \\
    & \Tilde{\Q} = \prob^i\circ(\xi^i,W^i,\mu^i)^{-1} \in \mathcal{P}(\R^d \times \contrd \times \contpdue).
\end{aligned}
\end{equation*}
Observe that $\Q$ and $\Tilde{\Q}$ are well defined, since $(\xi^i,W^i,\mu^i,\mathfrak{r}^i)$ share the same joint law by assumption.
Let us consider the following space:
\begin{equation*}
\begin{aligned}
    \Omega^{can} & =\contrd \times \contrd \times \R^d \times \contrd \times \contpdue \times \mathcal{V}; \\
    \F^{can} & =\boreliani{\contrd} \otimes \boreliani{\contrd } \otimes \boreliani{\R^d } \otimes \boreliani{ \contrd } \otimes \boreliani{\contpdue } \otimes \boreliani{\mathcal{V}}; \\
    \mathcal{G}_t^{can} &  = \mathcal{B}_{t,\contrd} \otimes \mathcal{B}_{t,\contrd} \otimes \boreliani{\R^d} \otimes \mathcal{B}_{t,\contrd} \otimes \boreliani{\contpdue} \otimes \F_t^{\mathcal{V}},
\end{aligned}
\end{equation*}
where
\begin{equation*}
\begin{aligned}
    & \mathcal{B}_{t,\contrd}=\sigma(\contrd \ni x \mapsto x_s \in \R^d: \; s \leq t), && \F^{\mathcal{V}}_t=\sigma( \mathcal{V} \ni q \mapsto q(C) \in \R: \; C \in \boreliani{[0,t]\times A}).
\end{aligned}
\end{equation*}
In order to equip the space $(\Omega^{can},\F^{can},(\mathcal{G}_t^{can})_{t \in [0,T]})$ with a probability measure, we disintegrate the measures $\hat{\Q}^i$, $i=1,2$, in the following way:
let $K^i:\boreliani{\contrd} \times \R^d \times \contrd \times \contpdue \times \mathcal{V} \to [0,1]$ be a regular conditional probability of $\hat{\Q}^i$ for $\boreliani{\contrd}$ given $(x,w,m,q)$, so that it holds
\begin{equation*}
    \hat{\Q}^i( A \times B ) = \int_B K^i(A , x,m,w,q) \Q (dx,dm,dw,dq),
\end{equation*}
for every $A \in \boreliani{\contrd}$, $B \in \boreliani{\R^d} \otimes \boreliani{\contrd} \otimes \boreliani{\contpdue} \otimes \boreliani{\mathcal{V}}$, or more briefly
\begin{equation*}
    \hat{\Q}^i(dx,dw,dm,dq,dy)=K^i(dy,x,m,q,w)\Q(dx,dw,dm,dq), \: i=1,2.
\end{equation*}
Then, we set 
\begin{equation*}
    \overline{\Q}(dy^1,dy^2,dx,dm,dw,dq)=K^1(dy^1,x,m,q,w)K^2(dy^2,x,m,q,w)\Q (dx,dm,dw,dq).
\end{equation*}
Observe that the joint law under $\overline{\Q}$ of the coordinate projections $y^1$, $x$, $m$, $w$ and $q$ is exactly $\hat{\Q}^1$, and analogously when considering the coordinate process $y^2$ instead of $y^1$.
Finally, complete the $\sigma$-algebra $\F^{can}$ with the $\overline{\Q}$-null sets $\mathcal{N}^{\overline{\Q}}$ and consider the complete right continuous filtration $(\F_t^{can})_{t \in [0,T]}$ given by
\begin{equation*}
    \F_t^{can}=\bigcap_{\eps > 0} \sigma\tonde{\mathcal{G}_{t+\eps}, \mathcal{N}^{\overline{\Q}}}.
\end{equation*}
By Lemma \ref{lemma_moto_browniano}, the coordinate process $w$ is a $(\F^{can}_t)_{t \in [0,T]}$-Brownian motion under $\overline{\Q}$.
Furthermore, it holds
\begin{equation*}
\begin{aligned}
    y^i_t=x + \int_0^t G(s,y^i_s,m,q)ds + w_t, \; \forall t \in [0,T], \; \overline{\Q}\text{-a.s.}
\end{aligned}
\end{equation*}
for $i=1,2$.
Since pathwise uniqueness in the sense of Definition \ref{weak_uniqueness:def_pathwise_uniqueness} holds by assumption, it follows that $y^1$ and $y^2$ are indistinguishable under $\overline{\Q}$, which implies $\hat{\Q}^1=\hat{\Q}^2$.
This proves the desired result.
\end{proof}

\begin{lemma}\label{lemma_moto_browniano}
In the construction of Theorem \ref{teorema_di_unicita_legge}, $w=(w_s)_{s \in [0,T]}$ is a Brownian motion under $\overline{\Q}$ with respect to the filtration $(\F^{can}_s)_{s \in [0,T]}$.
\end{lemma}
\begin{proof}
Observe that $w$ is a natural Brownian motion under $\overline{\Q}$.
In order to show that it is a Brownian motion with respect to the filtration $(\mathcal{G}_t^{can})_{t \in [0,T]}$, we just need to prove that its increments are independent, and the conclusion follows.

Fix $A_1,A_2 \in \mathcal{B}_{t,\contrd}$, $B \in \boreliani{\R^d}$, $C \in \mathcal{B}_{t,\contrd}$, $D \in \boreliani{\contpdue}$ and $F \in \F_t^{\mathcal{V}}$.
By Cauchy-Schwartz inequality, we have, for every $H \in \boreliani{\R^d}$ and $s>t$:
\begin{equation*}
\begin{aligned}
    \E^{\overline{\Q}} & \quadre{\1_H(w_s-w_t)\1_{A_1 \times A_2 \times B \times D \times C \times F}(y^1,y^2,x,m,w,q)}^2 \\
    \leq & \E^{\overline{\Q}}\Big[ \1_H(w_s-w_t)\1_{A_1 \times B \times D \times C \times F}(y^1,x,m,w,q)\Big] \\
    & \cdot \E^{\overline{\Q}}\Big[\1_H(w_s-w_t)\1_{A_2 \times B \times D \times C \times F}(y^2,x,m,w,q) \Big].
\end{aligned}
\end{equation*}
Therefore, it suffices to show that 
\begin{equation}\label{lemma_moto_browniano:indipendenza_incrementi_singoli}
\begin{aligned}
    \E^{\overline{\Q}}\Big[ \1_H(w_s-w_t)\1_{A_1 \times B \times D \times C \times F}(y^1,x,m,w,q)\Big]=0.
\end{aligned}
\end{equation}
Since the integrand does not depend upon $y^2$, we may rewrite such an expectation only with respect to $\hat{\Q}^1$:
\begin{equation*}
\begin{aligned}
    \E^{\overline{\Q}} & \Big[ \1_H(w_s-w_t)\1_{A_1 \times B \times D \times C \times F}(y^1,x,m,w,q)\Big] \\
    = & \int \1_H(w_s-w_t) \1_{A_1 \times B \times D \times C \times F}(y^1,x,m,w,q)\hat{\Q}^1(dy^1,dx,dm,dw,dq).
\end{aligned}
\end{equation*}
Then, we introduce another disintegration of the measure $\hat{\Q}^1$: let $\Theta^1$ be a regular conditional probability for $\boreliani{\contrd}\otimes \boreliani{\mathcal{V}}$ given $(x,w,m)$:
\begin{equation}\label{lemma_moto_browniano:dinsintegrazione_senza_controllo}
    \hat{\Q}^1\tonde{ A \times B \times C \times D \times F } = \int_{ B \times C \times D } \Theta^1(A\times F, x,m,w)\Tilde{\Q}(dx,dm,dw),
\end{equation}
for every $A \in \boreliani{\contrd}$, $B \in \boreliani{\R^d}$, $C \in \boreliani{\contrd}$, $D \in \boreliani{\contpdue}$ and $F \in \boreliani{\mathcal{V}}$, or more briefly
\begin{equation*}
    \hat{\Q}^1(dy^1,dq,dx,dw,dm)= \Theta^1(dy^1,dq,x,m,w)\Tilde{\Q}(dx,dw,dm).
\end{equation*}
As in \cite[Lemma IV.1.1]{ikeda_watanabe1981sdes}, it can easily be shown that,
for every $A\times F \in \mathcal{B}_{s,\contrd} \otimes \F_s^{\mathcal{V}}$, the map
\begin{equation*}
    (x,m,w) \mapsto \Theta^1(A\times F,x,m,w)  
\end{equation*}
is $\boreliani{\R^d} \otimes \boreliani{\contpdue}\otimes \mathcal{B}_{s,\contrd}$-measurable, for every $s \in [0,T]$.
Therefore, we can compute the left-hand side of \eqref{lemma_moto_browniano:indipendenza_incrementi_singoli}:
\begin{equation*}
\begin{aligned}
    \E^{\overline{\Q}} & \Big[ \1_H(w_s-w_t)\1_{A_1 \times B \times D \times C \times F}(y^1,x,m,w,q)\Big] \\
    = & \int \1_H(w_s-w_t) \1_{A_1 \times B \times D \times C \times F}(y^1,x,m,w,q)\hat{\Q}^1(dy^1,dx,dm,dw,dq) \\
    = & \int \1_H(w_s-w_t) \Theta^1(A_1\times F, x,m,w)\1_{B \times D \times C}(x,m,w)\Tilde{\Q}^1(dx,dm,dw) \\
    = & \E^{\prob^{1}}\quadre{ \1_H(W^1_s-W^1_t) \Theta^1(A_1\times F, \xi^1,\mu^1,W^1)\1_{B \times D \times C}(\xi^1,\mu^1,W^1)} = 0,
\end{aligned}
\end{equation*}
since $\Theta^1(A_1\times F, \xi^1,\mu^1,W^1)\1_{B \times D \times C}(\xi^1,\mu^1,W^1)$ is $\F^1_s$-measurable and $W^1$ is an $\mathbb{F}^1$-Brownian motion under $\prob^1$ by assumption.
\end{proof}

\section{On admissible recommendations}\label{appendix_recommendations}

\begin{lemma}\label{esempi:lemma_misurabile}
Let $(\Omega^0,\F^{0-},\prob^0)$ be a complete probability spaces and $(\Omega^*,\F^*,\mathbb{F}^*,\prob^*)$ be a filtered probability space satisfying the usual assumptions.
Fix a bounded $A$-valued process $(\lambda_t)_{t \in [0,T]}$ defined on the completion of the product space $(\Omega^0\times\Omega^*,\F^{0-}\otimes\F,\prob^0\otimes\prob^*)$.
Assume that it is progressively measurable with respect to the filtration $\mathbb{F}=(\F_t)_{t \in [0,T]}$, where $\mathbb{F}$ is the $\prob^0\otimes\prob^*$-augmentation of the filtration $(\F^{0-}\otimes\F^*_t)_{t \in [0,T]}$.
Define a function $\Lambda:\Omega^0\to \A$ by setting
\begin{equation}\label{appendix:raccomandazione_indotta}
\begin{aligned}
    \Lambda(\omega_0)=\left \{ \: \begin{aligned}
        & \begin{aligned}
            (\lambda_{t}(\omega_0,\cdot))_{t \in [0,T]} &  :\space [0,T] \times \Omega^* \to A \\
        & (t,\omega_*) \to \lambda_t(\omega_0,\omega_*), 
        \end{aligned} &&  \omega^0 \in \Omega^0\setminus N, \\
        & a_0 && \omega_0 \in N.
    \end{aligned} \right.
\end{aligned}
\end{equation}
where $N\subset \Omega^0$ is a $\prob^0$-null set and $a_0$ is some point in $A$.
The function $\Lambda$ defined in \eqref{appendix:raccomandazione_indotta} is an admissible recommendation.
\end{lemma}
\begin{proof}
Take any bounded $(\F)_{t \in [0,T]}$-progressively measurable process $(\lambda_t)_{t \in [0,T]}$ defined on the product space $(\Omega^0\times\Omega^*,\F^{0-}\otimes\F^*,\prob^0\otimes\prob^*)$ taking values in $\R$, and not necessarily in $A$.

\smallskip
Observe that it is always possible to define a function $\Lambda$ from $(\Omega^0,\F^{0-},\prob^0)$ to $(L^2([0,T]\times \Omega^*;\mathcal{R}^*,Leb_{[0,T]}\otimes\prob^*),\boreliani{L^2})$ as in \eqref{appendix:raccomandazione_indotta}, where $\mathcal{R}^*$ denotes the progressive $\sigma$-algebra associated to the filtration $\mathbb{F}^*$.
Indeed, by construction of the filtration $\mathbb{F}$, since $\lambda$ is $\mathbb{F}$-progressively measurable, there exists a set $N \subset \Omega^0$, $\prob^0(N)=0$, so that the section $(\lambda_t(\omega_0,\cdot))_{t \in [0,T]}$ is $\mathcal{R}^*$-measurable for every $\omega_0 \in \Omega^0 \setminus N$.
Set $\Lambda(\omega_0)=(\lambda_t(\omega_0,\cdot)_{t \in [0,T]})$ for $\omega_0 \in \Omega^0 \setminus N$ and $\Lambda(\omega_0) \equiv a_0$ for $\omega_0 \in N$, where $a_0$ is any point in $\R$, which is exactly \eqref{appendix:raccomandazione_indotta}.

\smallskip
Let $\mathcal{H}$ be the set of bounded progressively measurable processes $\lambda$ so that the function $\Lambda$ defined according to \eqref{appendix:raccomandazione_indotta} is a $\F^{0-}\setminus\boreliani{L^2}$ measurable random variable.
We show that $\mathcal{H}$ is a monotone class which contains the set $\mathcal{E}$ of progressively measurable processes $\lambda:[0,T]\times\Omega^0\times\Omega^* \to \R $ of the form
\begin{equation}\label{esempi:processi_semplici}
    \lambda_t=\sum_{i=1}^n \zeta^i\1_{[t_i,t_{i+1})}(t),
\end{equation}
where $n \geq 1$, $t_i \in [0,T]$, $t_i<t_{i+1}$ for every $i=1,\dots,n$, $\zeta^i$ are bounded $\F_{t_i}$-measurable random variables.
Having established such properties, we apply monotone class theorem (as stated, e.g., in \cite[Theorem II.3.1]{rogerswilliams_vol1}) to conclude that $\mathcal{H}$ contains the set of $\mathbb{F}$-progressively measurable bounded processes defined on the product space $\Omega^0\times\Omega^*$.

\smallskip
To see that $\mathcal{H}$ is a monotone class, observe that $\mathcal{H}$ is clearly a vector space and contains all processes $\lambda$ so that $\lambda_t\equiv c$ for every $t \in [0,T]$, for all $c \in \R$.
Let $(\lambda^n)_{n \geq 1} \subseteq \mathcal{H}$, with $\lambda^n \uparrow \lambda$ as $n$ goes to infinity, $\lambda^n$ positive and bounded by the same constant $C\geq 0$ for every $n$.
By monotone convergence, $\lambda$ is bounded and $\F^{0-} \otimes \mathcal{R}^*$-measurable as well, so that we can define $\Lambda$ as in \eqref{appendix:raccomandazione_indotta}, as previously discussed.
Let $\Lambda^n$ be the $L^2$-valued random variables defined starting from $\lambda^n$ according to \eqref{appendix:raccomandazione_indotta}, which are $\F^{0-} \setminus \boreliani{L^2}$ measurable since $\lambda^n$ belongs to $\mathcal{H}$ for every $n \geq 1$, by assumption.
Without loss of generality, we can suppose that the $\prob^0$-null set $N$ appearing in the definition of $\Lambda^n$ and $\Lambda$ is the same for every $n \geq 1$.
Notice that, for every $\omega_0 \in \Omega^0 \setminus N$, the sections $(\lambda^n_t(\omega_0,\cdot)_{t \in [0,T]}) \uparrow (\lambda_t(\omega_0,\cdot)_{t \in [0,T]})$ for every $(t,\omega_*) \in [0,T] \times \Omega^*$.
Therefore, by monotone convergence, it holds 
\begin{equation}\label{esempi:approssimazione_raccomdandazione}
\begin{aligned}
    \norm{\Lambda^n(\omega_0)-\Lambda(\omega_0)}^2_{L^2} & = \norm{(\lambda^n_t(\omega_0,\cdot)_{t \in [0,T]}) -(\lambda_t(\omega_0,\cdot)_{t \in [0,T]})}^2_{L^2} \\
    & =\E^{\prob^*}\quadre{\int_0^T \abs{\lambda^n(t,\omega_0,\omega_*)-\lambda(t,\omega_0,\omega_*)}^2 dt}\to 0
\end{aligned}
\end{equation}
for every $\omega_0 \in \Omega^0\setminus N$, i.e. $\Lambda=\lim_{n\to\infty}\Lambda^n$ $\prob^0$-a.s., which implies that $\Lambda$ is $\F^{0-}\setminus\boreliani{L^2}$ measurable, since the probability space is complete and $L^2([0,T]\times\Omega^*,\mathcal{R}^*,Leb_{[0,T]}\otimes\prob^*)$ is a complete norm space.
Finally, $\Lambda$ is admissible, since the process $\lambda$ obviously satisfies \eqref{mf:uguaglianza_ammissibilita}, choosing the same $\prob^0$-null set $N$ used in the definition $\Lambda$.

\smallskip
To see that $\mathcal{E} \subseteq \mathcal{H}$, suppose first that $\lambda$ is of the form 
\begin{equation*}
    \lambda_t=\sum_{i=1}^n \1_{A_i}(\omega_0)\1_{B_i}(\omega_*)\1_{[t_i,t_{i+1})}(t),
\end{equation*}
where $n \geq 1$, $t_i \in [0,T]$, $t_i<t_{i+1}$ for every $i=1,\dots,N$, $A_i\in\F^{0-}$ and $B_i\in\F^*_{t_i}$.
We can regard each variable $\1_{B_i}(\omega_*)\1_{[t_i,t_{i+1})}(t)$ as a bounded progressively measurable process $\alpha^i$.
Therefore, \eqref{appendix:raccomandazione_indotta} takes the following form:
\begin{equation*}
    \Lambda(\omega_0)=\begin{cases}
    \alpha^i \; & \omega_0 \in  A_i, \qquad i=1,\dots,N, \\
    0 \; & \omega_0 \in \tonde{\cup_{i=1}^n A_i}^c
    \end{cases}
\end{equation*}
which shows that $\Lambda$ is $\F^{0-}\setminus\boreliani{L^2}$-measurable.
By Dynkin Lemma, conclusion holds true for progressively measurable simple processes of the form
\begin{equation}\label{esempi:processo_intermedio}
    \lambda_t=\sum_{i=1}^n \1_{C_i}(\omega_0,\omega_*)\1_{[t_i,t_{i+1})}(t),
\end{equation}
where $n \geq 1$, $t_i \in [0,T]$, $t_i<t_{i+1}$ for every $i\in\insieme{1,\dots,n}$, $C_i\in\F^{0-}\otimes\F^*_{t_i}$.
Finally, let $\lambda$ be of the form \eqref{esempi:processi_semplici}.
Thanks to the boundedness assumption on $\zeta^i$, we can find a sequence of simple processes $(\lambda^n)_{n\geq 1}$ of the form \eqref{esempi:processo_intermedio} so that $\abs{\lambda^n}\leq \abs{\lambda}$ and $\lambda^{n}_t(\omega_0,\omega_*)\to\lambda_t(\omega_0,\omega_*)$ pointwise for every $(t,\omega_0,\omega_*)$.
Let $\Lambda^n$ and $\Lambda$ be defined according to \eqref{appendix:raccomandazione_indotta} starting by the processes $\lambda^n$.
Due to point 2.b) above, conclusion holds true for each $\Lambda^n$.
Using dominated convergence, we can prove that \eqref{esempi:approssimazione_raccomdandazione} holds for $\prob^0$-a.e. $\omega_0 \in \Omega^0$, so that $\Lambda$ is the a.s. pointwise limit  of $\Lambda^n$, which implies that $\Lambda$ is $\F^{0-}\setminus \boreliani{L^2}$ measurable. 
\end{proof}

\begin{proposition}\label{esempi:unicita_strategia_associata}
\begin{enumerate}[label=\roman*), wide]
Let $(\Omega^0,\F^{0-},\prob^0)$ be a complete probability space.
\item \label{esempi:unicita_strategia_associata:unicita_associata}  Let $\Lambda:(\Omega^0,\F^{0-},\prob^0)\to(\A,\boreliani{\A})$ be an admissible recommendation.
Let $\lambda^1$ and $\lambda^2$ be two $\mathbb{F}$-progressively measurable processes with values in $A$ so that \eqref{mf:uguaglianza_ammissibilita} holds.
Then $\lambda^1=\lambda^2$ $Leb_{[0,T]}\otimes\prob$-almost surely.
\item \label{esempi:unicita_strategia_associata:unicita_raccomandazione} Let $\Lambda,\Gamma:(\Omega^0,\F^{0-},\prob^0)\to(\A,\boreliani{\A})$ be admissible recommendations; let $\lambda,\gamma$ be the strategies associated to $\Lambda,\Gamma$, according to \eqref{mf:uguaglianza_ammissibilita}.
Suppose that $\lambda=\gamma$ $Leb_{[0,T]}\otimes \prob$-almost surely.
Then, $\Lambda=\Gamma$ $\prob^0$-a.s.
\end{enumerate}
\end{proposition}
\begin{proof}
As for point \ref{esempi:unicita_strategia_associata:unicita_associata}, let $N^i$, $i=1,2$, be two $\prob^0$-null sets so that, for every $\omega_0 \in \Omega_0\setminus N^i$ the sections $(\lambda^i_t(\omega_0,\cdot))_{t \in [0,T]}$ are $\mathbb{F}^*$-progressively measurable processes and equation \eqref{mf:uguaglianza_ammissibilita} holds true.
Without loss of generality, we can assume that $N^1=N^2=N$.
Since for every $\omega_0 \in \Omega^0 \setminus N$ it holds $\norm{\Lambda - (\lambda^i_t(\omega_0,\cdot))_{t \in [0,T]}}_{L^2}=0$, $i=1,2$, we deduce that
\begin{equation*}
    \norm{(\lambda^1_t(\omega_0,\cdot))_{t \in [0,T]} - (\lambda^2_t(\omega_0,\cdot))_{t \in [0,T]}}^2_{L^2} = 0
\end{equation*}
for every $\omega_0 \in \Omega^0 \setminus N$.
Therefore, by taking the integral with respect to $\prob^0$, we obtain
\begin{equation*}
\begin{aligned}
    0 & = \E^{\prob^0} \quadre{ \norm{(\lambda^1_t(\omega_0,\cdot))_{t \in [0,T]} - (\lambda^2_t(\omega_0,\cdot))_{t \in [0,T]}}^2_{L^2} } \\
    & = \E^{\prob^0} \quadre{ \E^{\prob^*} \quadre{ \int_0^T \vert \lambda^1_s(\omega_0,\omega_*) - \lambda^2_s(\omega_0,\omega_*) \vert^2 ds  } } = \E\quadre{ \int_0^T \vert \lambda^1_s(\omega_0,\omega_*) - \lambda^2_s(\omega_0,\omega_*) \vert^2 ds  }
\end{aligned}
\end{equation*}
by Fubini's theorem.
This is enough to conclude that $\lambda^1=\lambda^2$ $Leb_{[0,T]}\otimes\prob$-a.s.

\smallskip
As for point \ref{esempi:unicita_strategia_associata:unicita_raccomandazione}, by the same line of reasoning, if $\lambda=\gamma$ $Leb_{[0,T]}\otimes\prob$-a.s., the the sections $(\lambda_t(\omega_0,\cdot))_{t \in [0,T]}$ and $(\gamma_t(\omega_0,\cdot))_{t \in [0,T]}$ are $Leb_{[0,T]}\otimes\prob^*$-almost everywhere equal, which implies that
\begin{equation*}
\begin{aligned}
    \norm{\Lambda(\omega_0)-\Gamma(\omega_0)}^2_{L^2} & = \norm{(\lambda_t(\omega_0,\cdot))_{t \in [0,T]} - (\gamma_t(\omega_0,\cdot))_{t \in [0,T]}}^2_{L^2} \\
    & = \E^{\prob^*}\quadre{\int_0^T \vert \lambda_t(\omega_0,\omega_*) - \gamma_t(\omega_0,\omega_*) \vert^2 dt } = 0
\end{aligned}
\end{equation*}
$\prob^0$-a.s., so that $\Lambda=\Gamma$ $\prob^0$-a.s.
\end{proof}

\section{Propagation of chaos}\label{sec:propagation_of_chaos}

Here, we prove the propagation of chaos type result which is needed in the proof of Theorem \ref{thm_approssimazione}.
The probability spaces and the random variables we use here are defined in Section \ref{sec:approximation:sezione_costruzione_raccomandazioni}.

We work on the product probability space
\begin{equation*}
    \tonde{\Omega,\F,\prob}= \tonde{\overline{\Omega},\overline{\F},\overline{\prob}} \otimes \tonde{\Omega^1, \F^1, \prob^1},
\end{equation*}
with $(\overline{\Omega},\overline{\F},\overline{\prob})$ defined by \eqref{approximation:spazio_raccomandazioni} and \eqref{approximation:legge_spazio_raccomandazioni} and $(\Omega^1, \F^1, \prob^1)$ by \eqref{canonical_setup} or, equivalently, by \eqref{approximation:spazio_rumori}.
Consider the random measure flow $\mu$ defined by \eqref{approximation:approx_misura} and the recommendations $(\Lambda^i)_{i \geq 1}$ defined by \eqref{approximation:approx_raccomandazione}, which we recall are conditionally i.i.d. given $\mu$ under $\overline{\prob}$.
We endow such a probability space with the filtration $\mathbb{F}$ given by the $\prob$-augmentation of the filtration generated by $\overline{\F}$, the initial data $(\xi^i)_{i\geq 1}$ and the Brownian motions $(W^i)_{i\geq 1}$.
We observe that for every $N\geq 2$, each $\beta \in \A_N$ is also $\mathbb{F}$-progressively measurable and, for every $i \geq 1$, each strategy $\lambda^i$ associated to the admissible recommendation $\Lambda^i$ is $\mathbb{F}$-progressively measurable as well.

\smallskip
Fix $N\geq 2$, $\beta \in \A_N$ and $1 \leq i \leq N$.
Let $X=X [\Lambda^{N,-i},\beta]=\tonde{X^j[\Lambda^{N,-i},\beta]}_{j = 1}^N$ be the solution of
\begin{equation*}
    \begin{cases}
    dX^j_t=b(t,X^j_t,\mu^N_t,\lambda^{j}_t)dt + dW^j_t, \quad X^j_0=\xi^j, \quad j \neq i, \\
    dX^i_t=b(t,X^i_t,\mu^N_t,\beta_t)dt + dW^i_t, \quad X^i_0=\xi^i.
    \end{cases}
\end{equation*}
The process $X[\Lambda^{N,-i},\beta]$ is the state process of the $N$-players when every player $j \neq i$ follows the recommendation $\Lambda^{i}$ and player $i$ deviates by picking the strategy $\beta$, where $\mu^N_t$ denotes the empirical measure of the $N$-players' states at time $t$ defined in \eqref{misura_empirica}.
Let us introduce also the empirical measure of the processes $X=X[\Lambda^{N,-i},\beta]$:
\begin{equation}
    \mu^N[\Lambda^{N,-i},\beta]=\frac{1}{N}\sum_{j=1}^N\delta_{X^j[\Lambda^{N,-i},\beta]} \in \contpdue.
\end{equation}
Let us denote $X=X[\Lambda]=X[\Lambda^{N,-i},\Lambda^{i}]$ the state process of the $N$ players when every player $i=1,\dots,N$ follows the recommendation $\Lambda^{i}$.
Then, let us consider the following auxiliary processes: let $\tonde{Z^j[\Lambda^{N,-i},\beta]}_{j = 1}^N$ be the solution of
\begin{equation*}
    \begin{cases}
    dZ^j_t=b(t,Z^j_t,\mu_t,\lambda^{j}_t)dt + dW^j_t, \quad Z^j_0=\xi^j, \quad j \neq i, \\
    dZ^i_t=b(t,Z^i_t,\mu_t,\beta_t)dt + dW^i_t, \quad Z^i_0=\xi^i
    \end{cases}
\end{equation*}
and $\nu^N[\Lambda^{-i},\beta]$ be the empirical measure of the processes $Z[\Lambda^{-i},\beta]$:
\begin{equation*}
    \nu^N[\Lambda^{-i},\beta]=\frac{1}{N}\sum_{j=1}^N\delta_{Z^j[\Lambda^{N,-i},\beta]} \in \contpdue.
\end{equation*}

\begin{lemma}\label{lemma_poc}
Let $\beta$ be either an open-loop strategy in $\A_{K}$ for some $K \geq 2$, or be equal to $\lambda^{i}$, the strategy associated to the admissible recommendation $\Lambda^i$ to player $i$.
It holds:
\begin{align}
    & \sup_{t \in [0,T]}\attesa{\pwassmetric{2}{\R^d}{2}\tonde{\mu^N_t[\Lambda^{-i},\beta],\mu_t}} \overset{N \to \infty}{\longrightarrow} 0,  \label{lemma_poc:conv_wasserstein} \\
    & \max_{1 \leq j \leq N} \attesa{ \norm{X^{j,N}[\Lambda^{-i},\beta]-Z^j[\Lambda^{-i},\beta]}^2_{\contrd}} \overset{N \to \infty}{\longrightarrow} 0, \label{lemma_poc:conv_norma} \\
    & \sup_{N \geq 2} \max_{1 \leq j \leq N}\attesa{\norm{X^{j,N}[\Lambda^{-i},\beta]}^2_{\contrd} + \norm{Z^j[\Lambda^{-i},\beta]}^2_{\contrd}} < \infty. \label{lemma_poc:finitezza_momenti}
\end{align}
\end{lemma}
\begin{proof}
Because of the symmetry properties of the systems of SDEs, we can suppose $i=1$.
Throughout the proof, to make notation as simple as possible, we omit the dependence upon $[\Lambda^{-1},\beta]$.
For the same reason, define, for each $j \geq 1$, the following process $\gamma^j$:
\begin{equation*}
    \gamma^j_t=\begin{cases}
    \beta_t \qquad j=1, \\
    \lambda^j_t \qquad j \geq 2.
    \end{cases}
\end{equation*}
Obviously, in the case that $\beta$ is $\lambda^1$, we have $\gamma^j \equiv \lambda^j$ for every $j$.
Moreover, let us introduce the following auxiliary processes: let $(Y^j)_{j \geq 1}$ be the solution of
\begin{equation*}
    dY^j_t=b(t,Y^j_t,\mu_t,\lambda^j_t)dt + dW^j_t, \quad Y^j_0=\xi^j.
\end{equation*}
Let $\eta^N$ be the empirical measure of the processes $Y^j$:
\begin{equation*}
    \eta^N=\frac{1}{N}\sum_{j=1}^N\delta_{Y^j} \in \contpdue.
\end{equation*}
Denote by $X^*$ the state process resulting from the coarse correlated solution of the MFG, i.e.
\begin{equation*}
    dX^*_t=b\tonde{t,X^*_t,\mu^*_t,\lambda^*_t}dt + dW^*_t, \quad X^*_0=\xi^*.
\end{equation*}
Since, for every $j \geq 1$, $(\xi^j,W^j,\mu,\lambda^j)$ are distributed as $(\xi^*,W^*,\mu^*,\lambda^*)$, by Theorem \ref{teorema_di_unicita_legge} the processes $(Y^j)_{j \geq 1}$ are identically distributed copies of $X^*$; moreover, the joint distribution of $(Y^j,\mu)$ under $\prob$ is the same of $(X^*,\mu^*)$ under $\prob^*$, which, by marginalizing at every time $t \in [0,T]$, implies that $Y^j$ satisfies the consistency condition \eqref{def_mean_field_sol:cons} as well.

\smallskip
For every fixed $t \in [0,T]$, by the triangular inequality, it holds
\begin{equation}\label{lemma_poc:intermedio1}
\begin{aligned}
    \E\quadre{\pwassmetric{2}{\R^d}{2}(\mu^N_t,\mu_t)} & \leq C \E \quadre{ \pwassmetric{2}{\R^d}{2}\tonde{\mu^N_t,\nu^N_t} + \pwassmetric{2}{\R^d}{2}\tonde{\nu^N_t,\eta^N_t} + \pwassmetric{2}{\R^d}{2}\tonde{\eta^N_t,\mu_t}}.
\end{aligned}
\end{equation}
We start from the third term in \eqref{lemma_poc:intermedio1}: let $\prob^m$ be a version of the regular conditional probability of $\prob$ given $\mu=m$, and denote by $\E^m[\cdot]$ the expectation with respect to the measure $\prob^m$.
By construction, the strategies $(\lambda^j)_{j\geq 1}$ associated to the admissible recommendations $(\Lambda^j)_{j \geq 1}$ are i.i.d. under $\prob^m$, for $\rho$-a.e. $m \in \contpdue$.
Since $\mu$ is independent of $(W^j)_{j \geq 1}$ and $(\xi^j)_{j \geq 1}$ under $\prob$, the processes $(Y^j)_{j \geq 1}$ are independent under $\prob^m$.
Moreover, since $\mu_t(\cdot)=\prob( Y^j_t \in \cdot \; \vert \; \mu )$ $\prob$-a.s. for every $t \in [0,T]$ and $\mu_t=m_t$ $\prob^m$-a.s. for $\rho$-a.e. $m \in \contpdue$, we have
\begin{equation}\label{lemma_poc:consistency_disintegrate}
    m_t=\prob^m \circ (Y^j_t)^{-1}, \quad \rho\text{-a.e.}, \; \forall t \in [0,T],
\end{equation}
for every $j \geq 1$.
We can conclude that the processes $(Y^j_t)_{j \geq 1}$ are independent and identically distributed square integrable random variables with law $m_t$ under $\prob^m$, for every $t$ and for $\rho$-a.e. $m$.
Therefore, as ensured, e.g., by \cite[(5.19)]{librone_vol1}, it holds
\begin{equation*}
    \lim_{N \to \infty} \E^m\quadre{\pwassmetric{2}{\R^d}{2}\tonde{\eta^N_t,\mu_t}} = 0, \quad \rho\text{-a.e.}, \; \forall t \in [0,T].
\end{equation*}
We observe that there exists a function $g: \contpdue \to \R$, $g \in L^1(\rho)$, which is bigger or equal than $\E^m[\pwassmetric{2}{\R^d}{2}(\mu^N_t,\mu_t) ]$, $\rho$-a.e., for every $t$: indeed, since, under $\prob^m$, $Y^j_t$ are i.i.d with law $m_t$ and $\mu_t=m_t$ a.s., we have 
\begin{equation*}
\begin{aligned}
    \E^m & \quadre{\pwassmetric{2}{\R^d}{2}\tonde{\eta^N_t,\mu_t}} \leq 2 \E^m\quadre{\pwassmetric{2}{\R^d}{2}\tonde{\eta^N_t,\delta_0} + \pwassmetric{2}{\R^d}{2}\tonde{\delta_0,\mu_t}} \\
    & \leq 2 \tonde{\frac{1}{N}\sum_{k=1}^N\E^m\quadre{\abs{Y^k_t}^2} + \int_{\R^d}\abs{y}^2m_t(dy)}  \leq 2 \tonde{\frac{1}{N}\sum_{k=1}^N\E^m\quadre{\abs{Y^1_t}^2} + \E^m\quadre{\abs{Y^1_t}^2}} \\
    & \leq 4 \E^m\quadre{\abs{Y^1_t}^2} \leq 4 \E^m\quadre{\norm{Y^1}^2_{\contrd}}.
\end{aligned}
\end{equation*}
The function $g(m)=\E^m\quadre{\lVert Y^1 \rVert_{\contrd}}$ belongs to $L^1(\rho)$, since
\begin{equation}\label{lemma_poc:bound_uniforme}
    \int_{\contpdue} g(m)\rho(dm) = \E\quadre{ \E\quadre{\norm{Y^1}^2_{\contrd}\; \vert \; \mu}}= \E\quadre{\norm{Y^1}^2_{\contrd}} < \infty.
\end{equation}
Therefore, by dominated convergence theorem, we have 
\begin{equation}\label{lemma_poc:convergenza_wass_puntuale}
    \lim_{N \to \infty} \E\quadre{\pwassmetric{2}{\R^d}{2}\tonde{\eta^N_t,\mu_t}} = 0
\end{equation}
for every $t \in [0,T]$.
The convergence in \eqref{lemma_poc:convergenza_wass_puntuale} is actually uniform in time.
Indeed, fix $t,s \in [0,T]$: then
\begin{equation*}
\begin{aligned}
    \E & \quadre{\pwassmetric{2}{\R^d}{2}\tonde{\eta^N_t,\mu_t} - \pwassmetric{2}{\R^d}{2}\tonde{\eta^N_s,\mu_s}} \\
    & = \E \quadre{\tonde{\pwassmetric{2}{\R^d}{}\tonde{\eta^N_t,\mu_t} - \pwassmetric{2}{\R^d}{}\tonde{\eta^N_s,\mu_s}}\tonde{\pwassmetric{2}{\R^d}{}\tonde{\eta^N_t,\mu_t} + \pwassmetric{2}{\R^d}{}\tonde{\eta^N_s,\mu_s}}} \\
    & \leq C \E\quadre{\norm{Y^1}^2_{\contrd}}^\frac{1}{2} \E \quadre{\abs{\pwassmetric{2}{\R^d}{}\tonde{\eta^N_t,\mu_t} - \pwassmetric{2}{\R^d}{}\tonde{\eta^N_s,\mu_s}}^2}^{\frac{1}{2}},
\end{aligned}
\end{equation*}
where we used Cauchy-Schwartz inequality together with the uniform in time bound given by \eqref{lemma_poc:bound_uniforme}.
By triangulating with $\eta^N_s$ and $\mu_s$, we get
\begin{align*}
\E & \quadre{\big \vert \pwassmetric{2}{\R^d}{}\tonde{\eta^N_t,\mu_t} - \pwassmetric{2}{\R^d}{}\tonde{\eta^N_s,\mu_s} \big\vert^2 } \leq \E\Big[ \big \vert\pwassmetric{2}{\R^d}{}\tonde{\eta^N_t,\eta^N_s} + \pwassmetric{2}{\R^d}{}\tonde{\eta^N_s,\mu_s} \\
& + \pwassmetric{2}{\R^d}{}\tonde{\mu_s,\mu_t} - \pwassmetric{2}{\R^d}{}\tonde{\eta^N_s,\mu_s} \big\vert^2 \Big] \leq \E\quadre{\big \vert\pwassmetric{2}{\R^d}{}\tonde{\eta^N_t,\eta^N_s} - \pwassmetric{2}{\R^d}{}\tonde{\eta^N_s,\mu_s} \big\vert^2 } \\
\leq & C \tonde{\E \quadre{\pwassmetric{2}{\R^d}{2}\tonde{\eta^N_t,\eta^N_s}} + \E\quadre{\pwassmetric{2}{\R^d}{2}\tonde{\mu_t,\mu_s}}} \\
\leq & C \tonde{\E\quadre{\frac{1}{N}\sum_{k=1}^N \abs{Z^k_t - Z^k_s}^2} + \E \quadre{ \E\quadre{\pwassmetric{2}{\R^d}{2}\tonde{\mu_t,\mu_s}\; \vert \; \mu }} } \\
\leq & C \tonde{\E\quadre{\frac{1}{N}\sum_{k=1}^N \abs{Z^k_t - Z^k_s}^2} + \E\quadre{\abs{Y^1_t - Y^1_s}^2} },
\end{align*}
where in the last inequality we used \eqref{lemma_poc:consistency_disintegrate} and tower property.
By using Lipschitz continuity of $b$, the triangular inequality and $\E[\lVert Y^1 \rVert_{\contrd}]<\infty$, it is straightforward to see that $\E [ \lVert Z^k \rVert_{\contrd} ] \leq C$ for every $k \geq 1$, for some positive constant $C$ independent of $k$.
By the same arguments, we have 
\begin{equation*}
\begin{aligned}
    \E \quadre{\abs{Z^k_t - Z^k_s}^2 } & \leq C \E\quadre{\int_s^t \abs{ b(u,Z^k_u,\mu_u,\gamma^k_u) }^2 du } \\
    & \leq C \E\quadre{\int_s^t \tonde{ 1 + \abs{Z^k_u}^2 + \int_{\R^d}\abs{y}^2\mu_u(dy) + \abs{\gamma^k_u}^2} du }\\
    & \leq C \E\quadre{\int_s^t \tonde{ 1 + \norm{Z^k}^2_{\contrd} + \norm{Y^1}^2_{\contrd} + \abs{\gamma^k_u}^2} du } \leq C\abs{t-s},
\end{aligned}
\end{equation*}
where the constant $C$ depends upon $T$, $b$, $\E[\lVert Y^1 \rVert ^2 _{\contrd}] < \infty$ and $\text{diam}(A)$, which is a finite quantity since the set $A$ of actions is compact by Assumptions \ref{standing_assumptions}.
Analogously holds for $\E[\vert Y^1_t - Y^1_s \vert ]$, which implies that 
\begin{equation}\label{lemma_poc:convergenza_uniforme}
\begin{aligned}
    \abs{\E \quadre{\pwassmetric{2}{\R^d}{2}\tonde{\eta^N_t,\mu_t} - \pwassmetric{2}{\R^d}{2}\tonde{\eta^N_s,\mu_s}}} \leq C \abs{t-s}^\frac{1}{2}.
\end{aligned}
\end{equation}
This is enough to conclude, by Arzel\`a-Ascoli theorem, that the convergence in \eqref{lemma_poc:convergenza_wass_puntuale} is uniform in time.

\smallskip
Remind from Section \ref{sezione_notations_assumptions} that $\norm{x}_{t,\contrd}=\sup_{s \in [0,t]}\abs{x_s}$, $t \in [0,T]$.
To handle the second term in \eqref{lemma_poc:intermedio1}, we use the coupling of $\nu^N$ and $\eta^N$ given by $\frac{1}{N}\sum_{k=1}^N\delta_{(Z^k,Y^k)}$, together with the Lipschitz continuity of $b$:
\begin{equation*}
\begin{aligned}
    \E\quadre{\norm{Z^k-Y^k}^2_{t,\contrd}} & =\E\quadre{\sup_{0 \leq s \leq t} \tonde{\int_0^s \tonde{b\tonde{u,Z^k_u,\mu_u,\gamma^k_u} - b\tonde{u,Y^k_u,\mu_u,\lambda^k_u}}du}^2} \\
    & \leq C \left( \int_0^t \E\quadre{ \sup_{0 \leq u \leq s} \abs{Z^k_u - Y^k_u }^2}ds +  \int_0^t \E\quadre{ \abs{\lambda^k_s - \gamma^k_s}^2} ds \right).
\end{aligned}
\end{equation*}
By definition of $(\gamma^k)_{k \geq 1}$, we have
\begin{equation*}
    \int_0^T \E\quadre{\abs{\lambda^k_u-\gamma^k_u}^2}du=\begin{cases}
    & \int_0^T \E\quadre{\abs{\lambda^1_u-\beta_u}^2}du \qquad k = 1,\\
    & 0 \qquad \qquad \qquad \qquad  \qquad \quad k \geq 2.
    \end{cases}
\end{equation*}
Therefore, by Gronwall's lemma, we sum over $k=1,\dots,N$ to obtain the estimate
\begin{equation}\label{stime_poc_controlli}
\begin{aligned}
    \sup_{t \in [0,T]} & \E\quadre{\pwassmetric{2}{\R^d}{2}\tonde{\eta^N_t,\nu^N_t}} \leq \E\quadre{\pwassmetric{2}{\contrd}{2}\tonde{\eta^N,\nu^N}} \leq \frac{C}{N} \sum_{k=1}^N \E\quadre{\norm{Y^k - Z^k}_{\contrd}^2} \\
    & \leq \frac{C}{N} \sum_{j=1}^N \int_0^T \E\quadre{\abs{\lambda^j_u-\gamma^j_u}^2}du = \frac{C}{N} \int_0^T \E\quadre{\abs{\lambda^1_u-\beta_u}^2}du  \leq \frac{C}{N} \overset{N \to \infty}{\longrightarrow} 0,
\end{aligned}
\end{equation}
where the constant $C$ depends only upon $T$, $b$ and $\text{diam}(A)$.

\smallskip
Finally, for the first term of \eqref{lemma_poc:intermedio1}, we use the coupling of $\mu^N_t$ and $\nu^N_t$ given by $\frac{1}{N}\sum_{k=1}^N\delta_{(X^{k,N}_t,Z^k_t)}$, together with the Lipschitz continuity of $b$:
\begin{equation*}
\begin{aligned}
    \E & \quadre{\norm{X^{k,N}-Z^k}^2_{t,\contrd}} \leq C \int_0^t \left( \E\quadre{ \sup_{0 \leq u \leq s} \abs{X^{k,N}_u - Z^k_u}^2} + \E\quadre{ \pwassmetric{2}{\R^d}{2}(\mu^N_s,\mu_s) } \right) ds \\
    & \leq C  \int_0^t \left( \E\quadre{ \sup_{0 \leq u \leq s} \abs{X^{k,N}_u - Z^k_u}^2} + \E\quadre{ \pwassmetric{2}{\R^d}{2}(\mu^N_s,\nu^N_s)} + \E\quadre{ \pwassmetric{2}{\R^d}{2}(\nu^N_s,\mu_s) } \right) ds \\
    & \leq C \int_0^t \left( \E\quadre{ \sup_{0 \leq u \leq s} \abs{X^{k,N}_u - Z^k_u}^2} + \frac{1}{N}\sum_{j=1}^N \E\quadre{ \abs{X^{j,N}_s - Z^j_s}^2 } + \E\quadre{ \pwassmetric{2}{\R^d}{2}(\nu^N_s,\mu_s) }\right) ds \\
    & \leq  C \left( \int_0^t \max_{k=1,\dots,N} \E\quadre{ \sup_{0 \leq u \leq s} \abs{X^{k,N}_u - Z^k_u}^2} ds + \sup_{s \in [0,t]}\E\quadre{ \pwassmetric{2}{\R^d}{2}(\nu^N_s,\mu_s) }\right).
\end{aligned}
\end{equation*}
By taking the maximum over $k=1,\dots,N$ on the left-hand side and applying Gronwall's lemma, we get
\begin{equation*}
    \max_{k=1,\dots,N} \E \quadre{\norm{X^{k,N}-Z^k}^2_{\contrd}} \leq C \sup_{t \in [0,T]}\E\quadre{ \pwassmetric{2}{\R^d}{2}(\nu^N_t,\mu_t) } \to 0,
\end{equation*}
by \eqref{lemma_poc:convergenza_wass_puntuale}, \eqref{lemma_poc:convergenza_uniforme} and \eqref{stime_poc_controlli}, which proves \eqref{lemma_poc:conv_norma}.
Coming back to \eqref{lemma_poc:conv_wasserstein}, we have
\begin{multline*}
    \sup_{t \in [0,T]} \E\quadre{ \pwassmetric{2}{\R^d}{2}\tonde{\mu^N_t,\nu^N_t} } \leq \sup_{ t \in [0,T] } \frac{1}{N} \sum_{k=1}^N \E\quadre{\abs{X^{k,N}_t - Z^k_t }^2} \\
    \leq \max_{k=1,\dots,N}\E\quadre{ \norm{X^{k,N} - Z^k}_{\contrd}^2} \to 0.
\end{multline*}
This estimate together with estimates \eqref{lemma_poc:convergenza_wass_puntuale}, \eqref{lemma_poc:convergenza_uniforme} and \eqref{stime_poc_controlli} implies \eqref{lemma_poc:intermedio1} and therefore \eqref{lemma_poc:conv_wasserstein}.
Finally, \eqref{lemma_poc:finitezza_momenti} follows from the above calculations.
\end{proof}

\section{Auxiliary results for the existence of mean field CCE}\label{appendix_existence}

\subsection{ Proof of Theorem \ref{existence:thm_esistenza_strategia_ottima}.}
\label{existence:sezione_opt_strat_auxiliary}
The main instrument is the following Minimax Theorem, due to K. Fan:
\begin{theorem}[ \cite{fan1953minimax}, Theorem 2 ] \label{thm_minimax}
Let $X$ be a compact Hausdorff space and $Y$ an arbitrary set (not topologized). Let $f:X\times Y\to \R$ be a real-valued function such that, for every $y \in Y$, $x \mapsto f(x,y)$ is lower semi-continuous on $X$.
If $f(\cdot,y)$ is concave on $X$ for every $y \in Y$ and $f(x,\cdot)$ convex on $Y$ for every $x \in X$, then
\begin{equation}
    \max_{x \in X} \inf_{y \in Y} f(x,y) = \inf_{y \in Y} \max_{x \in X} f(x,y).
\end{equation}
\end{theorem}

The following results aims at verifying that the auxiliary zero-sum game in Definition \ref{existence:def_zerosum} satisfies the assumptions of Theorem \ref{thm_minimax}.
We start with some useful moment estimates for the solution to \eqref{existence:eq_processo_K}:
\begin{lemma}[Estimates]\label{lemma_stima_a_priori}
Let $\Gamma \in \mathcal{K}$, let $\mathfrak{U}=((\Omega,\F,\mathbb{F},\prob),\xi,W,\mu,\mathfrak{r})$ be the tuple associated to $\Gamma$, as in Definition \ref{existence:strategie_max}, and let $X$ be the solution to \eqref{existence:eq_processo_K}.
Then, for every $2 \leq p \leq \overline{p}$, there exists a constant $C=C(p,T,\nu,b,A)$ so that
\begin{equation}\label{lemma_stima_a_priori:stima}
    \attesa{\norm{X}_{\contrd}^p}\leq C.
\end{equation}
\end{lemma}
The proof is omitted as it is just a straightforward application of Gronwall's lemma.
We recall the following fact, which will be used extensively and whose proof can be found in \cite[Theorem 7.12]{villani2003}: given a metric space $(E,d_E)$,  a sequence $(\mu_n)_n \subset (\mathcal{P}^p(E),\pwassmetric{p}{E}{})$ is relatively compact if and only if it is tight and satisfies
\begin{equation}\label{wass:uniforme_integrabilita}
    \lim_{r \to \infty}\sup_n\int_{\insieme{x:\;d_E^p(x,x_0)\geq r}}d_E^p(x,x_0)\mu_n(dx)=0.
\end{equation}
\begin{lemma}\label{existence:lemma_tightness}
$\mathcal{K}$ is pre-compact in $(\mathcal{P}^2(\contrd \times \mathcal{V}\times \contpdue),\pwassmetric{2}{\contrd \times \mathcal{V}\times \contpdue}{})$.
\end{lemma}
\begin{proof}
Let $(\Gamma^n)_{n \geq 1}$ be a sequence in $\mathcal{K}$, let us show that it is pre-compact, which is equivalent to show that $(\Gamma^n)_{n \geq 1}$ is tight and condition \eqref{wass:uniforme_integrabilita} is satisfied. Moreover, by \cite[Lemma A.2]{lacker2015martingale}, relative compactness of the sequence $(\Gamma^n)_{n \geq 1}$ is equivalent to the relative compactness of each sequence of marginals on $\contrd$, $\contpdue$ and $\mathcal{V}$.

Since $A$ is compact by Assumptions \ref{standing_assumptions}, the space $\mathcal{V}$ is compact as well.
Then, we automatically get both tightness of the sequence of the marginals on $\mathcal{V}$ of $(\Gamma^n)_{n\geq 1}$ and property \eqref{wass:uniforme_integrabilita}.

In the following, for every $n\geq 1$, let $\mathfrak{U}^n=((\Omega^n,\F^n,\mathbb{F}^n,\prob^n),\xi^n,W^n,\mu^n,\mathfrak{r}^n)$ and $X^n$ be as in Definition \ref{existence:strategie_max}, so that $\Gamma^n=\prob^n\circ(X^n,\mu^n,\mathfrak{r}^n)^{-1}$.
Let $\Gamma^n_1$ be the law of $X^n$ under $\prob^n$.
We prove the tightness by means of Kolmogorov-\v{C}entsov criterion, as stated, e.g., in \cite[Corollary 16.9]{kallenberg_foundations}.
Let $2< p \leq \overline{p}$, $0\leq s < t \leq T$. We have:
\begin{equation*}
\begin{aligned}
    \E^n&\quadre{\abs{X^n_t - X^n_s}^p} \leq C\E^n\quadre{\int_s^t \int_A \abs{b(u,X^n_u,\mu^n_u,a)}^p \mathfrak{r}^n_u(da)du + \abs{W_t - W_s}^p} \\
    & \leq C\tonde{\abs{t-s}^{p-1}\int_s^t \E^n\quadre{ \int_A \abs{b(u,X^n_u,\mu^n_u,a)}^p \mathfrak{r}^n_u(da)}du + \abs{t-s}^\frac{p}{2}},
\end{aligned}
\end{equation*}
for some positive constant $C$ which is updated from line to line.
For every $u \in [0,T]$, we have
\begin{equation}\label{lemma_tightness:bound_uniforme}
\begin{aligned}
    \E^n&\quadre{ \int_A \abs{b(u,X^n_u,\mu^n_u,a)}^p \mathfrak{r}^n_u(da)} \\
    & \leq C\E^n\quadre{\abs{X^n_u}^p + \tonde{\int_{\R^d}\abs{y}^2 \mu^n_u(dy)}^\frac{p}{2} + \int_A \abs{a-a_0}^p \mathfrak{r}^n_u(da) + \abs{b(u,0,\delta_0,a_0)}^p } \\
    & \leq C\tonde{1 + \E^n\quadre{\abs{X^n_u}^p + \int_{\R^d}\abs{y}^p \mu^n_u(dy) } } = C\tonde{ 1 + \E^n\quadre{\abs{X^n_u}^p +\E^n\quadre{ \abs{X^n_u}^p \big \vert \mu^n }}} \\
    & = C\tonde{ 1 + 2\E^n\quadre{ \abs{X^n_u}^p } } \leq C \tonde{ 1 + \E^n\quadre{\sup_{u \in [0,T]} \abs{X^n_u}^p } } \leq  C,
\end{aligned}
\end{equation}
where the last inequality follows from Lemma \ref{lemma_stima_a_priori}, with $C$ independent of $n \geq 1$.
Such a uniform bound implies that
\begin{equation*}
\begin{aligned}
    \E^n&\quadre{\abs{X^n_t - X^n_s}^p} \leq C\tonde{\abs{t-s}^{p-1}\int_s^t \E^n\quadre{ \int_A \abs{b(u,X^n_u,\mu^n_u,a)}^p \mathfrak{r}^n_u(da)}du + \abs{t-s}^\frac{p}{2}} \\
    & \leq C\tonde{\abs{t-s}^{p-1}\abs{t-s}C(p,T,\nu,b,A) + \abs{t-s}^\frac{p}{2}} \leq C\tonde{\abs{t-s}^p + \abs{t-s}^{\frac{p}{2}}} \leq C \abs{t-s}^{\frac{p}{2}}.
\end{aligned}
\end{equation*}
Set $\beta=\sfrac{p}{2}-1$, so that we get
\begin{equation}\label{existence:lemma_tightness:stima_momenti}
    \E^n\quadre{\abs{X^n_t-X^n_s}^p}\leq C \abs{t-s}^{1+\beta},
\end{equation}
with $p,\beta>0$.
Since $\prob^n\circ(X_0^n)^{-1}= \nu \in \mathcal{P}^p(\R^d)$ for every $n\geq 1$, we have the tightness of the initial laws as well.
This concludes of the proof of the tightness of $(\Gamma^n_1)_{n \geq 1}$.
As for condition \eqref{wass:uniforme_integrabilita}, we have:
\begin{equation*}
    \begin{aligned}
        \lim_{r \to \infty} & \sup_n \int_{\insieme{y: \; \norm{y}_{\contrd}^2 > r}} \norm{y}_{\contrd}^2 \Gamma^n_1(dy)=\lim_{r \to \infty} \sup_n \E^n\quadre{\norm{X^n}_{\contrd}^2\1_{\insieme{\norm{X^n}_{\contrd}^2 > r}}} \\
        & \leq \lim_{r \to \infty} \sup_n \tonde{\E^n\quadre{\norm{X^n}_{\contrd}^{4}}}^\frac{1}{2}\prob^n\tonde{\norm{X^n}_{\contrd}^2 > r}^\frac{1}{2}\leq C \lim_{r \to \infty} \sup_n \prob^n\tonde{\norm{X^n}_{\contrd}^2 > r}^\frac{1}{2}
    \end{aligned}
\end{equation*}
for some positive constant $C$ independent of $n$.
By Markov's inequality and estimate \eqref{lemma_stima_a_priori:stima} again, we get
\begin{equation*}
    \lim_{r \to \infty} \sup_n \int_{\insieme{y: \; \norm{y}_{\contrd}^2 > r}} \norm{y}_{\contrd}^2 \Gamma^n_1(dy) \leq C \lim_{r \to \infty} \sup_n \E^n\quadre{\norm{X^n}_{\contrd}^2}^\frac{1}{2}r^{-\frac{1}{2}}=0.
\end{equation*}

Finally, we turn to the sequence $(\rho^n)_{n \geq 1}$, where $\rho^n=\prob^n\circ(\mu^n)^{-1}$.
Let $\prob^{n,m}(\cdot)=\prob^n( \cdot \; \vert \; \mu = m)$ be the regular conditional distribution of $\prob^n$ given $\mu^n=m$.
Then, $\mu^n_t=m_t$ $\prob^{n,m}$-a.s. and $\prob^{n,m} \circ (X^n_t)^{-1}= m_t$ $\rho$-a.e. for every $t \in [0,T]$, which implies that, for every $s,t \in [0,T]$, we have 
\begin{equation*}
    \E^{n,m}\quadre{\pwassmetric{2}{\R^d}{p}(\mu^n_t,\mu^n_s)}\leq \E^{n,m} \quadre{\abs{ X^n_t - X^n_s}^p }
\end{equation*}
for $\rho$-a.e. $m \in \contpdue$.
Integrating with respect to $\rho$ yields
\begin{equation*}
    \E^{n}\quadre{\pwassmetric{2}{\R^d}{p}(\mu^n_t,\mu^n_s)}\leq \E^{n} \quadre{\abs{ X^n_t - X^n_s}^p } \leq C \abs{t-s}^{1 + \beta}
\end{equation*}
where the last inequality follows from \eqref{existence:lemma_tightness:stima_momenti} with $\beta=\sfrac{p}{2}-1$.
Since $\prob^{n}\circ(\mu^n_0)^{-1}=\delta_{\nu}$, it is enough to apply again Kolmogorov-\v{C}entsov criterion and deduce the tightness of $(\rho^n)_{n \geq 1}$.
Finally, we verify condition \eqref{wass:uniforme_integrabilita}.
To this extent, we note that, for every $n \geq 1$, there exists a continuous modification of the process $(\E^n[\vert X^n_t \vert ^2\; \vert \; \mu^n])_{t \in [0,T]}$, so that it holds
\begin{equation*}
    \sup_{t \in [0,T]} \int_{\R^d} \abs{y}^2 \mu^n_t(dy) = \sup_{t \in [0,T]}\E^n\quadre{\abs{X^n_t}^2 \; \big\vert \; \mu^n_t } \quad \prob^n\text{-a.s.}
\end{equation*}
Indeed, estimate \eqref{existence:lemma_tightness:stima_momenti} implies that the process $(\E^n[\vert X^n_t \vert ^2\; \vert \; \mu^n])_{t \in [0,T]}$ satisfies
\begin{equation*}
\begin{aligned}
    \E^n & \quadre{\abs{    \E^n\quadre{\abs{X^n_t}^2\; \vert \; \mu^n} - \E^n\quadre{\abs{X^n_s}^2\; \vert \; \mu^n}}^p} \leq \E^n \quadre{ \abs{ \abs{X^n_t}^2 - \abs{X^n_s}^2 } ^ p } \\
    & = \E^n \quadre{ \abs{X^n_t - X^n_s}^p \abs{X^n_t + X^n_s}^p } \leq \E^n \quadre{  \abs{X^n_t - X^n_s}^{2p}}^\frac{1}{2} \E^n \quadre{ \abs{X^n_t + X^n_s}^{2p} }^\frac{1}{2} \leq C\abs{t-s}^\frac{p}{2},
\end{aligned}
\end{equation*}
where we have used Cauchy-Schwartz inequality, \eqref{lemma_stima_a_priori:stima} and  \eqref{existence:lemma_tightness:stima_momenti} to bound $\E^n [ \vert X^n_t - X^n_s \vert ^{2p} ]^{\sfrac{1}{2}}$ and $\E^n [ \vert X^n_t + X^n_s \vert ^{2p} ]^{\sfrac{1}{2}}$, respectively.
Therefore, by choosing $2 < p < \sfrac{\overline{p}}{2}$ and $\beta=\sfrac{p}{2}-1$ as above, we deduce from \cite[Theorem 3.3]{kallenberg_foundations}, that there exists a continuous modification of $(\E^n[\vert X^n_t \vert ^2\; \vert \; \mu^n])_{t \in [0,T]}$.
Then, observe that
\begin{equation*}
    \int_{\R^d} \abs{y}^2 \mu^n_t (dy) = \E^n\quadre{\abs{X^n_t}^2\; \vert \; \mu^n} \quad \forall t \in [0,T] \cap \Q, \; \prob^n\text{-a.s.}
\end{equation*}
Since both processes are almost surely continuous, we can take the supremum over every $t \in [0,T]$ to conclude that
\begin{equation}\label{existence:lemma_tightness:stima_sup}
    \sup_{t \in [0,T]}\int_{\R^d} \abs{y}^2 \mu^n_t (dy) = \sup_{t \in [0,T]} \E^n\quadre{\abs{X^n_t}^2\; \vert \; \mu^n} \leq \E^n\quadre{\sup_{t \in [0,T]} \abs{X^n_t}^2  \; \bigg \vert \; \mu^n }  \quad \prob^n\text{-a.s.}
\end{equation}
We are now ready to show that \eqref{wass:uniforme_integrabilita} holds for $(\rho^n)_{n \geq 1}$: by applying \eqref{existence:lemma_tightness:stima_sup} in the first inequality, Cauchy-Schwartz and Markov inequalities, we have
\begin{equation*}
\begin{aligned}
    & \lim_{r \to \infty} \sup_n \int_{\insieme{m: \; \sup_{t \in [0,T]} \int_{\R^d} \abs{y}^2 m_t(dy) > r}} \sup_{t \in [0,T]} \int_{\R^d} \abs{y}^2 m_t(dy) \rho^n(dm) \\
    & \leq \lim_{r \to \infty} \sup_n \E^n\quadre{ \E^n \quadre{ \norm{X}_{\contrd}^2 \; \vert \; \mu^n } \1_{\insieme{\E^n \quadre{ \norm{X}_{\contrd}^2 \; \vert \; \mu^n } > r}}} \\
    & \leq \lim_{r \to \infty} \frac{1}{r^\frac{1}{2}} \sup_n \E^n\quadre{\norm{X^n}_{\contrd}^{4}}^\frac{1}{2}\E^n\quadre{\norm{X^n}_{\contrd}^{2}}^\frac{1}{2}  \leq \lim_{r \to \infty} C r^{-\frac{1}{2}}=0, 
    \end{aligned}
\end{equation*}
since the suprema over $n \geq 1$ are finite by Lemma \ref{lemma_stima_a_priori}.
\end{proof}

\begin{lemma}\label{existence:lemma_closedness}
$\mathcal{K}$ is closed in $(\mathcal{P}^{2}(\contrd \times \mathcal{V}\times \contpdue),\pwassmetric{2}{\contrd \times \mathcal{V}\times \contpdue}{})$.
\end{lemma}
\begin{proof}
It is enough to prove that, for every sequence $(\Gamma^n)_{n \geq 1}\subseteq \mathcal{K}$ converging to $\Gamma$ as $n \to \infty$ in $\pwassmetric{2}{\contrd \times \mathcal{V} \times \contpdue}{}$, we have $\Gamma \in \mathcal{K}$.
We work on the following canonical space: let $(\overline{\Omega},\overline{\mathcal{G}})$ be given by 
\begin{equation*}
    (\overline{\Omega},\overline{\mathcal{G}})=\tonde{\contrd \times \contpdue \times \mathcal{V} ,\boreliani{\contrd} \otimes \boreliani{\mathcal{V}} \otimes \boreliani{\contpdue}}.
\end{equation*}
We equip such a space with the filtration $\mathbb{G}=(\mathcal{G}_t)_{t \in [0,T]}$ given by
\begin{equation*}
    \mathcal{G}_t=\mathcal{B}_{t,\contrd} \otimes \F_t^{\mathcal{V}} \otimes \boreliani{\contpdue},
\end{equation*}
where $\mathcal{B}_{t,\contrd}=\sigma(\contrd \ni x \mapsto x_s: \; s \leq t)$.
Let $x$, $m$ and $q$ denote the projection from $\overline{\Omega}$ in $\contrd$, $\contpdue$ and $\mathcal{V}$, respectively.
Define the process $w=(w_t)_{t \in [0,T]}$ as
\begin{equation}\label{existence:lemma_closedness:brownian_motion}
    w_t=w_t(x,q,m)=x_t-x_0-\int_0^t\int_A b(s,x_s,m_s,a)q_s(da)ds.
\end{equation}
Observe that $w$ is a continuous process on $(\overline{\Omega},\overline{\F})$ and, by \cite[Corollary A.5]{lacker2015martingale}, for every $t \in [0,T]$ $w_t$ is a continuous with at most linear growth function of $(x,q,m)$.

For every $n\geq 1$, let $\mathfrak{U}^n=((\Omega^n,\F^n,\mathbb{F}^n,\prob^n),\xi^n,W^n,\mu^n,\mathfrak{r}^n)$ and $X^n$ be as in Definition \ref{existence:strategie_max}, so that $\Gamma^n=\prob^n\circ(X^n,\mu^n,\mathfrak{r}^n)^{-1}$.
Since $\Gamma^n \circ (x_0,w,m,q,x)^{-1}=\prob^n\circ(\xi^n,W^n,\mu^n,\mathfrak{r}^n,X^n)^{-1}$, we have that the tuple $\mathfrak{U}^n=((\overline{\Omega},\overline{\F},\overline{\mathbb{F}}^{\Gamma^n},\Gamma^n),x_0,w,m,q)$ satisfies the requirements of Definition \ref{existence:strategie_max}, where $\overline{\mathbb{F}}^{\Gamma^n}$ denotes the $\Gamma^n$-augmentation of the filtration $\mathbb{G}$.
We show that the tuple $\mathfrak{U}=((\overline{\Omega},\overline{\F},\overline{\mathbb{F}}^{\Gamma},\Gamma),x_0,w,m,q)$ satisfies the requirements of Definition \ref{existence:strategie_max}, which implies $\Gamma \in \mathcal{K}$.

We start by the independence property of $w$, $m$ and $q$ under $\Gamma$.
Let $(t^i)_{i=1}^{k} \subseteq [0,T]$, $\varphi^i \in \cbounded{\R^d}$ for $i=1,\dots,n$, $\psi \in \cbounded{\contpdue}$, $\phi \in \cbounded{\R^d}$ be bounded continuous functions.
Since $W^n$, $\mu^n$ and $\xi^n$ are independent under $\prob^n$ and $\Gamma^n \to \Gamma$ weakly, we have
\begin{equation*}
\begin{aligned}
    \E^{\Gamma^n} & \quadre{\prod_{i=1}^k\varphi^i\tonde{w_{t^i}(x,q,m)}\psi(m)\varphi\tonde{x_0}} \\
    & = \E^{\prob^n}\quadre{\prod_{i=1}^k\varphi^i\tonde{W^n_{t^i}}\psi(\mu^n)\varphi\tonde{\xi^n}}= \E^{\prob^n}\quadre{\prod_{i=1}^k\varphi^i\tonde{W^n_{t^i}}}\E^{\prob^n}\quadre{\psi(\mu^n)}\E^{\prob^n}\quadre{\varphi\tonde{\xi^n}}  \\
    & = \E^{\Gamma^n}\quadre{\prod_{i=1}^k\varphi^i\tonde{w_{t^i}(x,q,m)}}\E^{\Gamma^n}\quadre{\psi(m)}\E^{\Gamma^n}\quadre{\varphi\tonde{x_0}}
\end{aligned}
\end{equation*}
where the first equality holds since $w_t(X^n,\mathfrak{r}^n,\mu^n)=W^n_t$ for every $t \in [0,T]$ $\prob^n$-a.s.
Then, since $\phi^i \circ w_{t^i}$ is a continuous function of $(x,q,m)$ for every $i$, weak convergence implies that
\begin{equation}\label{existence:lemma_closedness:convergenze}
\begin{aligned}
    \lim_{n \to \infty} & \E^{\Gamma^n}\quadre{\prod_{i=1}^k\varphi^i\tonde{w_{t^i}(x,q,m)}\psi(m)\varphi\tonde{x_0}} = \E^{\Gamma}\quadre{\prod_{i=1}^k\varphi^i\tonde{w_{t^i}(x,q,m)}\psi(m)\varphi\tonde{x_0}}, \\
    \lim_{n \to \infty} & \E^{\Gamma^n}\quadre{\prod_{i=1}^k\varphi^i\tonde{w^n_{t^i}(x,q,m)}}\E^{\Gamma^n}\quadre{\psi(m)}\E^{\Gamma^n}\quadre{\varphi\tonde{\xi^n}} \\
    & = \E^{\Gamma}\quadre{\prod_{i=1}^k\varphi^i\tonde{w_{t^i}(x,q,m)}}\E^{\Gamma}\quadre{\psi(m)}\E^{\Gamma}\quadre{\varphi\tonde{x_0}}.
\end{aligned}
\end{equation}
This is enough to ensure the mutual independence under $\Gamma$ of $(w_{t^i})_{i=1,\dots,k}$, $x_0$ and $m$ for every $(t^i)_{i=1}^{k}\subset [0,T]$, which yields the independence of $w$, $x_0$ and $m$.
Moreover, by taking $\psi$ and $\phi$ identically equal to $1$, equation
\eqref{existence:lemma_closedness:convergenze} implies that $w$ is natural Brownian motion under $\Gamma$, since the finite dimensional distributions of $w$ coincide with the ones of a Brownian motion. 
Let us verify the independence of increments properties.
Let $s > t$, $\varphi \in \cbounded{\contrd}$ $\mathcal{B}_{t,\contrd}$-measurable, $\chi \in \cbounded{\mathcal{V}}$ $\F_t^{\mathcal{V}}$-measurable, $\psi \in \cbounded{\contpdue}$ and $\phi \in \cbounded{\R^d}$.
Then, we have:
\begin{equation*}
\begin{aligned}
    \E^{\Gamma} & \quadre{\phi\tonde{w_s-w_t}\varphi\tonde{x}\chi\tonde{q}\psi(m)} = \E^{\Gamma}  \quadre{\phi\tonde{w_s(x,q,m)-w_t(x,q,m)}\varphi\tonde{x}\chi\tonde{q}\psi(m)} \\
    & =\lim_{n \to \infty}\E^{\Gamma^n}\quadre{\phi\tonde{w_s(x,q,m)-w_t(x,q,m)}\varphi\tonde{x}\chi\tonde{q}\psi(m)} \\
    & = \lim_{n \to \infty}\E^{\prob^n}\quadre{\phi\tonde{w_s(X^n,\mathfrak{r}^n,\mu^n)-w_t(X^n,\mathfrak{r}^n,\mu^n)}\varphi(X^n)\chi\tonde{\mathfrak{r}^n}\psi(\mu^n)} \\
    & = \lim_{n \to \infty}\E^{\prob^n}\quadre{\phi\tonde{W^n_s-W^n_t}\varphi(X^n)\chi\tonde{\mathfrak{r}^n}\psi(\mu^n)} = 0,
\end{aligned}
\end{equation*}
where the last equality holds since $W^n$ is a $\mathbb{F}^n$-Brownian motion under $\prob^n$, $\mu^n$ is $\F_0^n$-measurable and $X^n$ and $\mathfrak{r}^n$ are both $\mathbb{F}^n$-adapted.
By working with an approximating sequence, this holds also for bounded measurable $\varphi$, $\chi$, $\psi$ and $\phi$, which is enough to conclude the independence of increments.
Finally, since $w$ is $\mathbb{G}$-Brownian motion, it remains so under the $\Gamma$-augmentation of $\mathbb{G}$.

Since $\Gamma^n \circ x_0^{-1}\equiv \nu$, we have that $\Gamma \circ x_0 = \nu$ as well.
Moreover, since $w$ is a $\overline{\mathbb{F}}^{\Gamma}$-Brownian motion and $x$ is $\overline{\mathbb{F}}^{\Gamma}$-adapted by definition of the filtration, equation \eqref{existence:lemma_closedness:brownian_motion} implies that $x$ is a solution to \eqref{existence:eq_processo_K}.

As for the consistency condition, observe that, for every $t \in [0,T]$, $\varphi \in \cbounded{\R^d}$, $\psi \in \cbounded{\contpdue}$, we have 
\begin{align*}
    \E^{\Gamma^n}\quadre{\int_{\R^d}\varphi(y)m_t(dy)\psi(m)} & = \E^{\prob^n}\quadre{\int_{\R^d}\varphi(y)\mu^n_t(dy)\psi(\mu^n)}\\
     &=\E^{\prob^n}\quadre{\E^{\prob^n}\quadre{\varphi(X^n_t)\psi(\mu^n)\; \vert \;  \mu^n}}\\
     & =\E^{\prob^n}\quadre{\varphi(X^n_t)\psi(\mu^n)}= \E^{\Gamma^n}\quadre{\varphi\tonde{x_t}\psi(m)},
\end{align*}
since $\mu^n_t$ is a version of the conditional law under $\prob^n$ of $X^n_t$ given $\mu^n$.
Therefore, by weak convergence we have both
\begin{equation*}
\begin{aligned}
    & \lim_{n \to \infty}\E^{\Gamma^n}\quadre{\varphi\tonde{x_t}\psi(m)} = \E^{\Gamma}\quadre{\varphi\tonde{x_t}\psi(m)}, \\
    & \lim_{n \to \infty}\E^{\Gamma^n}\quadre{\int_{\R^d}\varphi(y)m_t(dy)\psi(m)} = \E^{\Gamma}\quadre{\int_{\R^d}\varphi(y)m_t(dy)\psi(m)}
\end{aligned}
\end{equation*}
where the second limit holds since the function $m \mapsto \int_{\contrd}\varphi(y)m_t(dy) \in \cbounded{\contpdue}$, which implies
\begin{equation*}
    \E^{\Gamma}\quadre{\int_{\R^d}\varphi(y)m_t(dy)\psi(m)}=\E^{\Gamma}\quadre{\varphi\tonde{x_t}\psi(m)}.
\end{equation*}
This is enough to conclude that $m_t = \Gamma(x_t \in \cdot \; \vert \; m)$ $\Gamma$-a.s for every $t \in [0,T]$, since the random element $(x_t,m)$ takes values in a Polish space.
\end{proof}

\begin{lemma}[Convexity]\label{existence:lemma_cvx_strategie}
$\mathcal{K}$ and $\mathcal{Q}$ are convex.
\end{lemma}
\begin{proof}
We start by proving that $\mathcal{K}$ is convex.
Let $\Gamma^i$, $i=1,2$, be in $\mathcal{K}$, and let $\alpha \in (0,1)$.
Let $\mathfrak{U}^i=((\Omega^i,\F^i,\mathbb{F}^i,\prob^i),\xi^i,W^i,\mu^i,\mathfrak{r}^i)$ be as in Definition \ref{existence:strategie_max}, so that $\Gamma^i=\prob^i\circ(X^i,\mathfrak{r}^i,\mu^i)^{-1}$.
Set $\Xi^i=(\xi^i,W^i,\mu^i,\mathfrak{r}^i,X^i)$.
Without loss of generality, we can suppose that the tuples are defined on the same probability space $(\Omega,\F,\mathbb{F},\prob)$ which supports also a Bernoulli random variable $\eta \sim B(\alpha)$, so that $\eta$ and $(\Xi^i)_{i=1,2}$ are mutually independent.
If needed, we can enlarge the filtration so that $\eta$ is $\F_0$-measurable.
Let us consider the following random variables:
\begin{equation}\label{existence:lemma_cvx_strategie:combinazioni_variabili}
    \begin{aligned}
        & \xi^\alpha=\eta\xi^1 + \tonde{1-\eta}\xi^2, && \mu^\alpha=\eta\mu^1 + \tonde{1-\eta}\mu^2, \\
        & W^\alpha=\eta W^1 + \tonde{1-\eta}W^2, \; && \mathfrak{r}^\alpha=\eta \mathfrak{r}^1 + \tonde{1-\eta}\mathfrak{r}^2, \\
        & X^\alpha=\eta X^1 + \tonde{1-\eta}X^2.
    \end{aligned}
\end{equation}
Set $\Xi^\alpha=(\xi^\alpha,W^\alpha,\mu^\alpha,\mathfrak{r}^\alpha,X^\alpha)$ and $\Gamma^\alpha=\prob\circ(X^\alpha,\mathfrak{r}^\alpha,\mu^\alpha)^{-1}$.
Observe that the law of $\Xi^1$ under $\prob$ is the same as the law of $\Xi^\alpha$ conditionally to $\eta=1$, as the two tuples coincide on the set $\{\eta=1\}$, and analogously for $\eta=0$.
Therefore, for every Borel set $B$, we have
\begin{equation}\label{existence:lemma_cvx:combinazione_cvx}
\begin{aligned}
    \prob\tonde{\Xi^\alpha\in B}
    & = \prob\tonde{\Xi^\alpha \in B \big \vert  \eta =1}\prob\tonde{\eta=1} + \prob\tonde{ \Xi^\alpha \in B \big \vert  \eta =0}\prob\tonde{\eta=0} \\
    & = \prob\tonde{ \Xi^1 \in B }\prob\tonde{\eta=1} + \prob\tonde{\Xi^2 \in B }\prob\tonde{\eta=0} \\
    & = \alpha\prob\tonde{ \Xi^1 \in B } + (1-\alpha)\prob\tonde{ \Xi^2 \in B }.
\end{aligned}
\end{equation}
In particular, \eqref{existence:lemma_cvx:combinazione_cvx} implies that $\Gamma^\alpha=\alpha\Gamma^1 + (1-\alpha)\Gamma^2$.
Let us show that the tuple $\Xi^\alpha$ satisfies the requirements of Definition \ref{existence:strategie_max}.
By \eqref{existence:lemma_cvx:combinazione_cvx}, $\xi^\alpha$ has law $\nu$ and $W^\alpha$ is a natural Brownian motion.
To see that it is an $\mathbb{F}$-Brownian motion, let $s<t$, $G \in \F_s$, $B \in \boreliani{\R^d}$: then
\begin{equation*}
\begin{aligned}
    \E\quadre{\1_{B}\tonde{W^\alpha_t-W^\alpha_s}\1_G}  = & \E\quadre{\1_{B}\tonde{W^\alpha_t-W^\alpha_s}\1_G \1_{\insieme{\eta =1}}} + \E\quadre{\1_{B}\tonde{W^\alpha_t-W^\alpha_s}\1_G\1_{\insieme{\eta =0}}} \\
     = &  \E\quadre{\1_B\tonde{\eta \tonde{W^1_t-W^1_s} + \tonde{1-\eta}\tonde{W^2_t-W^2_s} \in B}\1_G\1_{\insieme{\eta =0}}} \\
    & + \E\quadre{\1_B\tonde{\eta \tonde{W^1_t-W^1_s} + \tonde{1-\eta}\tonde{W^2_t-W^2_s} \in B}\1_G\1_{\insieme{\eta =1}} }\\
     = &  \E\quadre{\1_B\tonde{W^1_t-W^1_s}\1_{G \cap\insieme{\eta =1}}} + \E\quadre{\1_B\tonde{W^1_t-W^1_s}\1_{G \cap\insieme{\eta =0}} }= 0,
\end{aligned}
\end{equation*}
since $\eta$ is $\F_0$-measurable by assumption.
As for the mutual independence of $\xi^\alpha$, $\mu^\alpha$ and $W^\alpha$, we have that the joint law factorizes in the product of the marginals: by using \eqref{existence:lemma_cvx:combinazione_cvx}, since $(\xi^i,W^i)_{i=1,2}$ share the same joint law, one gets 
\begin{equation*}
\begin{aligned}
    \prob (\mu^\alpha & \in A, W^\alpha \in B,\xi^\alpha \in C) \\
    = & \alpha\prob(\mu^1 \in A, W^1  \in B, \xi^1 \in C)  + (1-\alpha)\prob(\mu^2 \in A, W^2  \in B, \xi^2 \in C) \\
    = & \alpha\prob(\mu^1 \in A)\prob( W^1  \in B)\prob(\xi^1 \in C)  + (1-\alpha)\prob(\mu^2 \in A)\prob(W^2\in B)\prob(\xi^2 \in C) \\
    = & \left(\alpha\prob(\mu^1 \in A)  +(1-\alpha) \prob(\mu^2 \in A) \right)\wienermeasure(B)\nu(C) = \prob(\mu^\alpha \in A)\prob(W^\alpha\in B)\prob(\xi^\alpha \in C).
\end{aligned}
\end{equation*}
With similar arguments, one can show that for every $t \in [0,T]$, $g:\R^d \to \R$, $f:\contpdue \to \R$ bounded and measurable, it holds
\begin{equation*}
\begin{aligned}
    \E \quadre{ g\tonde{X^\alpha_t} f\tonde{\mu^\alpha}} = \E\quadre{ \int_{\contrd} g\tonde{ y } \mu^\alpha_t (dy) f\tonde{\mu^\alpha}  },
\end{aligned}
\end{equation*}
which implies that $\mu^\alpha_t$ is a version of the conditional distribution of $X^\alpha_t$ given $\mu^\alpha$.
Finally, consider the set
\begin{equation*}
    \Omega^1=\insieme{X^1_t = \xi^1 + \int_0^t \int_A b(s,X^1_s,\mu^1_s,a)\mathfrak{r}^1_s(da)ds + W^1_t \;\; \forall t \in [0,T]}\cap \insieme{\eta=1},
\end{equation*}
and define analogously $\Omega^2$.
We have that $\Omega^1 \cap \Omega^2 = \emptyset$ and $\prob(\Omega^1)=\alpha$, since $X^1$ satisfies the equation above $\prob$-a.s., and analogously $\prob(\Omega^2)=1-\alpha$, so that $\prob(\Omega^1 \cup \Omega^2)=1$.
On such a set, $X^\alpha$ satisfies the equation
\begin{equation*}
    X^\alpha_t = \xi^\alpha + \int_0^t \int_A b(s,X^\alpha_s,\mu^\alpha_s,a)\mathfrak{r}^\alpha_s(da)ds + W^\alpha_t, \; t \in [0,T].
\end{equation*}
Since $X^\alpha$ is $\mathbb{F}$-adapted, $X^\alpha$ is a solution to equation \eqref{existence:eq_processo_K}, which concludes this part of the proof.

\smallskip
Let us turn to the convexity of the set $\mathcal{Q}$. Let $\Sigma^i$, $i=1,2$, be in $\mathcal{Q}$, and $\alpha \in (0,1)$.
Let $\mathfrak{U}^i=((\Omega^i,\F^i,\mathbb{F}^i,\prob^i)$, $\xi^i,W^i,\mathfrak{r}^i)$ be as in Definition \ref{existence:strategie_min} so that $\Sigma^i(\cdot,m)=\prob^i((X^{m,i},\mathfrak{r}^i)\in \cdot )$, where $X^{m,i}$ is the solution to equation \eqref{existence:eq_processo_Q} on $(\Omega^i,\F^i,\mathbb{F}^i,\prob^i)$ when $b$ is evaluated at $m \in \contpdue$.
Let $\Theta^i=\prob^i\circ(\xi^i,W^i,\mathfrak{r}^i)^{-1}$, and consider the maps $\mathcal{I}_{\Theta^i}$ defined by 
\begin{equation}
\begin{aligned}
    \mathcal{I}_{\Theta^i}: \contpdue  & \longrightarrow \mathcal{P}(\R^d \times \contrd \times \contrd \times \mathcal{V}) \\
    m & \longmapsto \mathcal{I}_{\Theta^i}(m)=\prob^i\circ(\xi^i,W^i,X^{m,i},\mathfrak{r}^i)^{-1}.
\end{aligned}
\end{equation}
Similarly as for the set $\mathcal{K}$, suppose that the tuples are defined on the same probability space $(\Omega,\F,\mathbb{F},\prob)$ supporting also a Bernoulli random variable $\eta \sim B(\alpha)$, so that $\eta$ and $(\xi^i,W^i,\mathfrak{r}^i)_{i=1,2}$ are mutually independent.
If needed, we can enlarge the filtration so that $\eta$ is $\F_0$-measurable.
Let $\xi^\alpha$, $W^\alpha$ and $\mathfrak{r}^\alpha$ be as in \eqref{existence:lemma_cvx_strategie:combinazioni_variabili}, and, for every $m \in \contpdue$, define
\begin{equation*}
    \begin{aligned}
        X^{\alpha,m}=\eta X^{1,m} + \tonde{1-\eta}X^{2,m}.
    \end{aligned}
\end{equation*}
Let $\Theta^\alpha=\prob^\alpha \circ ( \xi^\alpha,W^\alpha,\mathfrak{r}^\alpha)^{-1}$ and consider the map $\mathcal{I}_{\Theta^\alpha}$, defined analogously to above.
By point \ref{existence:lemma_kernels:kernel_ben_definito} of Lemma \ref{existence:lemma_kernels}, it induces a stochastic kernel $\Sigma^\alpha \in \mathcal{Q}$.
By working in the same way as in the case of $\mathcal{K}$, we can show that 
$\mathcal{I}_{\Theta^\alpha}(m) = \alpha\mathcal{I}_{\Theta^1}(m) + (1-\alpha)\mathcal{I}_{\Theta^2}(m)$ for each $m \in \contpdue$, which implies that $\Sigma^\alpha=\alpha \Sigma^1 + (1-\alpha)\Sigma^2 \in \mathcal{Q}$.
\end{proof}

\begin{proposition}\label{existence:prop_payoff_continuo}
The map $\mathcal{K} \times \mathcal{Q} \ni (\Gamma,\Sigma) \mapsto \mathfrak{p}(\Gamma,\Sigma)$ is bilinear.
Moreover, $\mathcal{K} \ni \Gamma \mapsto \mathfrak{p}(\Gamma,\Sigma)$ is continuous for every $\Sigma \in \mathcal{Q}$.
\end{proposition}
\begin{proof}
Bilinearity is clear, hence we focus on the continuity of $\mathfrak{p}(\cdot,\Sigma)$ for fixed $\Sigma$.
Take $(\Gamma^n)_{n\geq 1}$, $\Gamma$ in $\mathcal{K}$ and suppose $\Gamma^n \to \Gamma$ in the 2-Wasserstein distance.
We treat separately the term depending just upon $\Gamma \in \mathcal{K}$ and the term depending also upon $\Sigma \in \mathcal{Q}$ in \eqref{existence:payoff_functional}.

\smallskip
By \cite[Theorem 7.12]{villani2003}, $\Gamma^n \to \Gamma$ in 2-Wasserstein metrics if and only if
\begin{equation}
\int_{\contrd \times\mathcal{V} \times \contpdue}\psi(y,q,m)\Gamma^n(dy,dq,dm) \to \int_{\contrd \times\mathcal{V} \times \contpdue}\psi(y,q,m)\Gamma(dy,dq,dm),
\end{equation}
for every $\psi$ continuous with at most quadratic growth; hence, we just need to show that the functional $\mathfrak{F}$ defined in \eqref{existence:funzione_f} is continuous with at most quadratic growth.
By Assumptions \ref{standing_assumptions} and \cite[Corollary A.5]{lacker2015martingale}, we have that $\mathfrak{F}(y,q,m)$ is continuous.
It is straightforward to verify that $\mathfrak{F}$ has at most quadratic growth, in the sense that
\begin{equation*}
    \mathfrak{F}(y,q,m) \leq C\tonde{1 + \norm{y}_{\contrd}^2 + \sup_{t \in [0,T]} \int_{\R^d} \abs{y}^2 m_t(dy) + \int_0^T \int_A \abs{a-a_0}^2 q_s(da)ds }.
\end{equation*}
Therefore, we get continuity of the term depending only upon $\Gamma$.

\smallskip
Denote by $\rho^n$ and $\rho$ the marginal of $\Gamma^n$ and $\Gamma$ on $\contpdue$.
We can manipulate the term depending both upon $\Gamma$ and $\Sigma$ as
\begin{equation*}
    \begin{aligned}
        & \int_{\contrd\times \mathcal{V}\times \contpdue}  \mathfrak{F} (y,q,m) \Sigma(dy,dq,m)\rho(dm) \\
        & \: = \int_{\contpdue}\tonde{\int_{\contrd\times \mathcal{V}} \mathfrak{F} (y,q,m) \Sigma(dy,dq,m)}\rho(dm)= \int_{\contpdue} g(m)\rho(dm)
\end{aligned}
\end{equation*}
where we set
\begin{equation*}
    g(m)=\int_{\contrd\times \mathcal{V}} \mathfrak{F} (y,q,m) \Sigma(dy,dq,m).
\end{equation*}
We must show that $g: \contpdue\to\R$ is continuous with at most quadratic growth with respect to the 2-Wasserstein distance.
As for the growth condition, estimate \eqref{existence:lemma_continuita:bound_uniformi} in Lemma \ref{existence:lemma_kernels} proves that $g$ has at most quadratic growth in $m \in \contpdue$.
As for the continuity, let $(m^n)_{n \geq 1},m \in \contpdue$ so that $m^n \to m$ in $\pwassmetric{2}{\contrd}{}$.
Note that $\Sigma(dy,dq,m^n) \to \Sigma(dy,dq,m)$ in $\pwassmetric{2}{\contrd \times \mathcal{V}}{}$, as implied by Lemma \ref{existence:lemma_kernels}.
Define $\phi^n(y,q)=\mathfrak{F}(y,q,m^n)$.
Since the cost functions are locally Lipschitz, we have that $\phi^n$ converges to $\phi$ uniformly on bounded sets of $\contrd \times \mathcal{V}$.
This is enough to conclude that
\begin{equation*}
    g(m^n)=\int_{\contrd\times \mathcal{V}} \phi^n (y,q) \Sigma(dy,dq,m^n) \to \int_{\contrd\times \mathcal{V}} \phi (y,q) \Sigma(dy,dq,m)=g (m)
\end{equation*}
as $m^n \to m$.
\end{proof}

We can now prove points \ref{existence:thm_esistenza_strategia_ottima:existence_value} and \ref{existence:thm_esistenza_strategia_ottima:optimal_strategy} of Theorem \ref{existence:thm_esistenza_strategia_ottima}: take $X=\mathcal{K}$, $Y=\mathcal{Q}$ and $f(x,y)=\mathfrak{p}(\Gamma,\Sigma)$ in the statement of Theorem \ref{thm_minimax}.
By Lemmata \ref{existence:lemma_tightness} and \ref{existence:lemma_closedness}, $\mathcal{K}$ is compact with the topology of convergence in 2-Wasserstein distance and both sets $\mathcal{K}$ and $\mathcal{Q}$ are convex by Lemma \ref{existence:lemma_cvx_strategie}.
By Proposition \ref{existence:prop_payoff_continuo}, the payoff $\mathfrak{p}$ is both concave and continuous in $\Gamma$ for every fixed $\Sigma \in \mathcal{Q}$ and convex in $\Sigma$ for every fixed $\Gamma$.
Therefore, Theorem \ref{thm_minimax} yields the existence of both the value $v$ of the auxiliary zero-sum game and an optimal strategy for player A. 
The next proposition proves point \ref{existence:thm_esistenza_strategia_ottima:positive_value}, concluding the proof of Theorem \ref{existence:thm_esistenza_strategia_ottima}.

\begin{proposition}[Positivity of the value of the auxiliary zero-sum game]\label{existence:prop_positive_value}
Let $v$ be the value of the zero-sum game defined in Definition \ref{existence:def_zerosum} has a value $v$.
Then $v \geq 0$.
\end{proposition}
\begin{proof}
We show that, for every $\Sigma \in \mathcal{Q}$ there exists a strategy $\Gamma_{\Sigma} \in \mathcal{K}$ so that $\mathfrak{p}(\Gamma_{\Sigma},\Sigma)=0$.
Fix $\Sigma \in \mathcal{Q}$, let $\mathfrak{U}=((\Omega,\F,\mathbb{F},\prob),\xi,W,\mathfrak{b})$ be a tuple as in Definition \ref{existence:strategie_min} so that $\Sigma(\cdot,m)=\prob((X^m,\mathfrak{b})\in\cdot)$, for every $m \in \contpdue$.
On this probability space, consider the following stochastic differential equation of McKean-Vlasov type:
\begin{equation}\label{existence:positivity_MKV_SDE}
   \begin{cases}
   dY_t=\int_A b(t,Y_t,p_t,a)\mathfrak{b}_t(da)dt + dW_t, \; t \in [0,T], \quad Y_0=\xi; \\
   \mathcal{L}(Y_t)=p_t, \: t \in [0,T], \quad p=(p_t)_{t \in [0,T]} \in \contpdue.
   \end{cases}
\end{equation}
Under Assumptions \ref{standing_assumptions}, there exists a unique pair $(Y,p)$ satisfying \eqref{existence:positivity_MKV_SDE}, where $Y=(Y_t)_{t \in [0,T]}$ is an $\mathbb{F}$-adapted continuous process so that $\E[\sup_{t \in [0,T]}\vert Y_t \vert ^2]<\infty$, as ensured by, e.g., \cite[Theorem 4.21]{librone_vol1}, which implies that $p$ actually belongs to $\contpdue$.
Define the deterministic flow of measures $\mu$ by setting $\mu=p$.
Define $\Gamma_{\Sigma}$ as $\prob\circ(Y,\mu,\mathfrak{b})^{-1}$.
Since $\mu=p$ is deterministic and $(Y,p)$ is a solution to \eqref{existence:positivity_MKV_SDE}, $\mu$ is $\F_0$-measurable and independent of $\xi$ and $W$, and consistency condition holds trivially.
This implies that $\Gamma_{\Sigma}$ belongs to $\mathcal{K}$.
By writing the integrals in $\mathfrak{p}$ as expectations, we have:
\begin{equation*}
\begin{aligned}
    & \mathfrak{p} (\Gamma_{\Sigma},\Sigma) = \int_{\contrd\times \mathcal{V}\times \contpdue}  \mathfrak{F} (y,q,m) \Sigma(dy,dq,m)\rho_{\Sigma}(dm) - \int_{\contrd\times \mathcal{V}\times \contpdue} \mathfrak{F} (y,q,m)  \Gamma_{\Sigma}(dy,dq,dm) \\
    & =  \E \bigg[\int_0^T \int_A f(t,Y_t,p_t,a)\mathfrak{b}_t(da)dt + g(Y_T,p_T)\bigg] - \E\bigg[ \int_0^T \int_A f(t,Y_t,p_t,a)\mathfrak{b}_t(da)dt + g(Y_T,p_T) \bigg] \\
    & = 0,
\end{aligned}
\end{equation*}
where $\rho_{\Sigma}(\cdot)=\delta_{p}(\cdot)$ denotes the marginal law of $\Gamma_{\Sigma}$ on $\contpdue$.
Since such a construction holds for every $\Sigma \in \mathcal{Q}$, we have 
\begin{equation*}
    \sup_{ \Gamma \in \mathcal{K}} \mathfrak{p}(\Gamma,\Sigma)  \geq \mathfrak{p}(\Gamma_{\Sigma},\Sigma) = 0 \quad \forall \Sigma \in \mathcal{Q}.
\end{equation*}
Taking the infimum with respect to $\Sigma \in \mathcal{Q}$,  we have 
\begin{equation*}
    \inf_{\Sigma \in \mathcal{Q}} \sup_{ \Gamma \in \mathcal{K}} \mathfrak{p}(\Gamma,\Sigma)  \geq \mathfrak{p}(\Gamma_{\Sigma},\Sigma)  \geq 0,
\end{equation*}
which shows that $v^B$ is non-negative.
Since $v^A=v^B=v$, this proves that the value of the auxiliary zero-sum game is non-negative.
\end{proof}

\subsection{Further technical Lemmata}\label{appendix_auxiliary:technical}
In this section, we state and prove some auxiliary results that were used in Section \ref{sezione_existence} to prove the existence of a mean field CCE.
In particular, Lemmata \ref{existence:lemma_kernels} and \ref{existence:lemma_decomposizione} provide the technical instruments we used in Proposition \ref{existence:prop_relazione_mfg} to show that, for every deviating strategy $\beta \in \A$ and random flow of measures $\mu$, we can represent the joint law of $\mu$, $\beta$ and deviating player's state process in terms of a strategy for player B in the zero-sum game \ref{existence:def_zerosum} and the the law of $\mu$.
Lemmata \ref{existence:lemma_mimicking} and \ref{existence:mimicking:lemma_strong_existence} were needed in the proof of Theorem \ref{existence:main_theorem} in order to define a mean field CCE starting from an optimal strategy for player A in the zero-sum game \ref{existence:def_zerosum}.

\smallskip
Consider any tuple $\mathfrak{U}=((\Omega,\F,\mathbb{F},\prob),\xi,W,\mu,\mathfrak{r})$, composed of a filtered probability space satisfying usual assumptions, a $d$-dimensional $\mathbb{F}$-Brownian motion, an $\R^d$-valued $\F_0$-measurable random variable, an $\F_0$-measurable
random continuous flow of measures in $\mathcal{P}^2(\R^d)$ and an $\mathbb{F}$-progressively measurable $\mathcal{P}(A)$-valued process.
Assume that $\mu$, $W$ and $\xi$ are mutually independent.
Let us consider the following equations:
\begin{align}
    dX_t & =\int_A b(t,X_t,\mu_t,a)\mathfrak{r}_t(da)dt + dW_t, \quad X_0=\xi, \label{existence:eq_misura_aleatoria} \\
    dX^m_t & =\int_A b(t,X^m_t,m_t,a)\mathfrak{r}_t(da)dt + dW_t, \quad X_0=\xi, \label{existence:eq_misura_deterministica}
\end{align}
where $m$ is a point of $\contpdue$.
In order to stress the dependence upon the deterministic flow of measures $m$, we write $X^m$ for the solution of \eqref{existence:eq_misura_deterministica}.

By Assumptions \ref{standing_assumptions}, on any such tuple $\mathfrak{U}$ there exists a solution to equation \eqref{existence:eq_misura_aleatoria} and pathwise uniqueness holds.
If needed, we can suppose that the filtration $\mathbb{F}$ on $(\Omega,\F,\prob)$ is the $\prob$-augmentation of the filtration $\mathbb{F}^{\xi,W,\mu,\mathfrak{r}}$, given by
\begin{equation}\label{existence:filtrazione_forte_ctrl}
    \F_t^{\xi,W,\mu,\mathfrak{r}}=\sigma(\xi)\vee\sigma(\mu)\vee\sigma(W_s: s\leq t)\vee\sigma(\mathfrak{r}(C): C \in \boreliani{[0,t]\times A}).
\end{equation}
By Theorem \ref{teorema_di_unicita_legge}, uniqueness in law holds.
Analogous reasoning holds for equation \eqref{existence:eq_misura_deterministica} as well, for every $m \in \contpdue$.

\begin{lemma}\label{existence:lemma_kernels}
Let $\mathfrak{U}=((\Omega,\F,\mathbb{F},\prob),\xi,W,\mu,\mathfrak{r})$ be as above, let $\Theta \in \mathcal{P}(\R^d \times \contrd \times \mathcal{V})$ be the joint law of $\xi$, $W$ and $\mathfrak{r}$.
Let us define the map
\begin{equation}\label{existence:mappa_strat_min_continua}
\begin{aligned}
    \mathcal{I}_{\Theta}: \contpdue  & \longrightarrow \mathcal{P}(\R^d \times \contrd \times \contrd \times \mathcal{V}) \\
    m & \longmapsto \mathcal{I}_{\Theta}(m)=\prob\circ(\xi,W,X^m,\mathfrak{r})^{-1},
\end{aligned}
\end{equation}
where $X^m$ is the solution to equation \eqref{existence:eq_misura_deterministica}.
\begin{enumerate}[label=(\roman*)]
    \item \label{existence:lemma_kernels:mappa_continua}
    The map $\mathcal{I}_{\Theta}$ is continuous, in the sense that
    \begin{equation*}
        \sup_{t \in [0,T]}\pwassmetric{2}{\R^d}{}(m^n_t,m_t) \to 0 \text{ as } n \to \infty \; \Longrightarrow \; \mathcal{I}_{\Theta}(m^n) \overset{n \to \infty}{\longrightarrow} \mathcal{I}_{\Theta}(m) \text{ in } \pwassmetric{2}{\R^d \times \contrd \times  \mathcal{V} \times \contrd}{}.
    \end{equation*}
    
    \item \label{existence:lemma_kernels:kernel_ben_definito}
    The map $\mathcal{I}_{\Theta}$ induces a stochastic kernel $\Sigma$ from $\contpdue$ to $\contrd \times \mathcal{V}$, by setting
    \begin{equation*}
        \Sigma(B,m)=\prob((X^m,\mathfrak{r}) \in B)=\mathcal{I}_{\Theta}(m)(\R^d \times \contrd \times B) \quad \forall m \in \contpdue, \; B \in \boreliani{\contrd}\otimes\boreliani{\mathcal{V}}.
    \end{equation*}
    $\Sigma$ is a strategy for player B, as described in Definition \ref{existence:strategie_min}.
\end{enumerate}
\end{lemma}
\begin{proof}
Note that, by Theorem \ref{teorema_di_unicita_legge}, $\mathcal{I}_{\Theta}(m)$ is the unique weak solution of \eqref{existence:eq_misura_deterministica} when the joint law of $\xi$, $W$ and $\mathfrak{r}$ is given by $\Theta$ and $b$ is evaluated at $m \in \contpdue$.
Let $\Sigma \in \mathcal{Q}$, let $(m^n)_{n\geq 1} \subset \contpdue$ so that $m^n \to m$.
For every $n\geq 1$, denote by $X$ and $X^n$ the solution to equation \eqref{existence:eq_misura_deterministica} when $b$ is evaluated at $m$ and $m^n$, respectively.
For every $2 \leq p \leq \overline{p}$, by Lipschitz continuity of $b$, we have:
\begin{align}
    \E\quadre{\sup_{0 \leq s \leq T}\abs{X^n_s-X_s}^p}\leq C \sup_{t \in [0,T]}\pwassmetric{p}{\R^d}{2}(m^n_t,m_t), \label{existence:lemma_continuita:stime_norme} \\ 
    \sup_{n \geq 1}\attesa{\sup_{t \in [0,T]}\abs{X^n_t}^p}\leq C\tonde{1 + \sup_{n \geq 1} \sup_{t \in [0,T]} \tonde{\int_{\R^d}\abs{y}^2 m^n_t(dy)}^{\frac{p}{2}}}. \label{existence:lemma_continuita:bound_uniformi}
\end{align}
Therefore, for $p=2$, we have that $m^n \to m$ in $\contpdue$ implies that $\lVert  X^n-X \rVert_{\contrd}^2 \to 0$ in expectation, which in turns implies that $(\xi,W,\mathfrak{r},X^n) \longrightarrow (\xi,W,\mathfrak{r},X)$ in distribution.
In order to have convergence in 2-Wasserstein metrics, it is enough to check uniform integrability, according to  \eqref{wass:uniforme_integrabilita}.
Since $(\mathcal{I}_{\Theta}(m^n))_n$ have the same marginals on $\R^d\times\contrd\times\mathcal{V}$, we just need to check \eqref{wass:uniforme_integrabilita} for the laws of $(X^n)_n$: for every $n \geq 1$, $r > 0$, we have
\begin{equation*}
\begin{aligned}
    \E\quadre{\norm{X^n}_{\contrd}^2\1_{\insieme{\norm{X^n}^2_{\contrd}>r}}} & \leq  \tonde{\E\quadre{\norm{X^n}_{\contrd}^{4}}}^\frac{1}{2}\tonde{\E\quadre{\1_{\insieme{\norm{X^n}^2_{\contrd}>r}}}}^\frac{1}{2}  \\
    & \leq \tonde{\E\quadre{\norm{X^n}_{\contrd}^{4}}}^\frac{1}{2} \E\quadre{\norm{X^n}_{\contrd}^2}^\frac{1}{2}r^{-\frac{1}{2}} \leq C r^{-\frac{1}{2}}
\end{aligned}
\end{equation*}
by using Cauchy-Schwartz inequality, Markov inequality, \eqref{existence:lemma_continuita:stime_norme} and \eqref{existence:lemma_continuita:bound_uniformi}.
By taking the limit as $r \to \infty$, we get condition \eqref{wass:uniforme_integrabilita} satisfied and so point \ref{existence:lemma_kernels:mappa_continua} is proved.

\smallskip
As for point \ref{existence:lemma_kernels:kernel_ben_definito}, let $\pi:\R^d\times\contrd\times\contrd\times\mathcal{V}\to\contrd\times\mathcal{V}$ be the projection on the last two components.
Note that $(\mathcal{I}_{\Theta}\circ \pi^{-1})(m)=\prob \circ (X^m,\mathfrak{r})^{-1}$, which shares the same continuity properties of the map $\mathcal{I}_{\Theta}$.
Therefore, in particular, it is Borel measurable, where $\mathcal{P}(\contrd \times \mathcal{V})$ is endowed with the usual Borel $\sigma$-algebra associated with the topology of weak convergence.
Then, the thesis follows from the fact that, for a Polish space $E$, the usual Borel $\sigma$-field on $\mathcal{P}(E)$ coincide with the $\sigma$-field generated by the maps $\mathcal{P}(E) \ni m \mapsto m(S)$, with $S \in \boreliani{E}$, (see, e.g., \cite[Corollary 7.29.1]{bertsekas_shreve}).
\end{proof}

\begin{lemma}\label{existence:lemma_decomposizione}
Let $\mathfrak{U}=((\Omega,\F,\mathbb{F},\prob),\xi,W,\mu,\mathfrak{r})$ be a tuple so that $(\xi,W,\mathfrak{r})$ and $\mu$ are independent.
Denote by $\rho \in \mathcal{P}(\contpdue)$ the law of $\mu$ under $\prob$.
Suppose without loss of generality that $\mathbb{F}$ is the $\prob$-augmentation of the filtration $(\F^{\xi,W,\mu,\mathfrak{r}}_t)_t$ defined by \eqref{existence:filtrazione_forte_ctrl}.
Let $X$ be the unique solution of \eqref{existence:eq_misura_aleatoria} on the tuple $\mathfrak{U}$.
Then, the following decomposition of measure holds
\begin{equation*}
    \prob((X,\mathfrak{r},\mu) \in B \times S)= \int_S \Sigma(B,m)\rho(dm), \quad \forall B \in \boreliani{\contrd \times \mathcal{V}}, \; S \in \boreliani{\contpdue}.
\end{equation*}
In particular, $\Sigma(B,m)=\prob((X,\mathfrak{r}) \in B \; \vert \; \mu=m)=\prob((X^m,\mathfrak{r}) \in B )$ for every $B \in \boreliani{\contrd \times \mathcal{V}}$, $\rho$-a.e. $m \in \contpdue$.
\end{lemma}
\begin{proof}
Let $\prob(\cdot \; \vert \; \mu)$ denote the regular conditional probability of $\prob$ given $\mu$.
Set $\prob^m(\cdot)=\prob(\cdot \; \vert \; \mu=m)$.
Since $(\xi,W,\mathfrak{r})$ and $\mu$ are independent by assumption, we have that $\prob^m\circ(\xi,W,\mathfrak{r})^{-1}=\prob\circ(\xi,W,\mathfrak{r})^{-1}$ for $\rho$-a.e. $m \in \contpdue$.
Therefore, it is enough to prove that $X$ is a solution to \eqref{existence:eq_misura_deterministica} on the tuple $\mathfrak{U}=((\Omega,\F,\mathbb{F},\prob^m),\xi,W,\mathfrak{r})$ for $\rho$-a.e. $m \in \contpdue$.
Then, since uniqueness in law holds for \eqref{existence:eq_misura_deterministica}, we deduce that $\prob^m\circ(\xi,W,\mathfrak{r},X)^{-1}=\mathcal{I}_{\Theta}(m)$ $\rho$-a.s.
Observe that, since the joint law of  $(\xi,W,\mathfrak{r})$ is the same under $\prob$ and $\prob^m$ for $\rho$-a.e $m$, $W$ is a natural Brownian motion under $\prob^m$ as well.
By definition of the filtration $\mathbb{F}$, it can be easily verified that
\begin{equation*}
    \E^\prob[\1_A(W_t-W_s)g \; \vert \;  \mu]=0 \quad \prob\text{-a.s}
\end{equation*}
for every $0 \leq s < t \leq T$, $A \in \boreliani{\R^d}$, $g$ bounded and $\F_s$-measurable.
This implies that
\begin{equation*}
    E^{\prob^m}[\1_A(W_t-W_s)g]=\E^{\prob}[\1_A(W_t-W_s)g\; \vert \;  \mu=m]=0
\end{equation*}
$\rho$-a.s., for every $g$ bounded and $\F_s$-measurable.
By working with a countable measure determining class of sets, which is possible since the $\sigma$-algebra $\F^{\xi,W,\mu,\mathfrak{r}}_t$ is countably generated for every $t \in [0,T]$, the equality holds for every $g$ bounded and $\F_s$-measurable, for $\rho$-a.e. $m \in \contpdue$, which in turn implies that $W$ remains an $\mathbb{F}$-Brownian motion under $\prob^m$ as well.
Under $\prob^m$ one has 
\begin{equation*}
    \prob^{m}\tonde{\int_A b(t,x,\mu_t,a)\mathfrak{r}_t(da)=\int_A b(t,x,m_t,a)\mathfrak{r}_t(da) \quad \forall x \in \R^d}=1 \quad Leb_{[0,T]}\text{-a.e. } t \in [0,T]
\end{equation*}
and therefore $X$ solves \eqref{existence:eq_misura_deterministica} for $b$ evaluated at $m \in \contpdue$.
The thesis follows from marginalizing as in the proof of point \ref{existence:lemma_kernels:kernel_ben_definito} in \ref{existence:lemma_kernels}.
\end{proof}

\smallskip
We turn our attention to the mimicking result, needed in the proof of the existence result in Theorem \ref{existence:main_theorem}:
\begin{lemma}\label{existence:lemma_mimicking}
Let $\Gamma \in \mathcal{K}$. There exists a measure $\hat{\Gamma} \in \mathcal{K}$ so that the following holds:
\begin{itemize}
    \item The marginal distributions of $\Gamma$ and $\hat{\Gamma}$ on $\contpdue$ are the same: $\Gamma(\contrd \times \mathcal{V} \times \cdot)=\hat{\Gamma}(\contrd \times \mathcal{V} \times \cdot)$.
    
    \item Let $(X,\mathfrak{r},\mu)$ be such that $\hat{\Gamma}=\prob\circ(X,\mathfrak{r},\mu)^{-1}$.
    Then $\mathfrak{r}$ is of the form $\mathfrak{r}_t=\hat{q}_t(X_t,\mu)$, where $\hat{q}:[0,T]\times\R^d\times\contpdue \to \mathcal{P}(A)$ is a measurable function.
    \item For every $\Sigma \in \mathcal{Q}$, it holds
    \begin{equation*}
        \mathfrak{p}(\Gamma,\Sigma)=\mathfrak{p}(\hat{\Gamma},\Sigma).
    \end{equation*}
\end{itemize}
\end{lemma}
\begin{proof}
In the following, for a metric space $(E,d_E)$, $\phi:E \to \R$ continuous and bounded and $m \in \probmeasures{E}$, we set $\langle \phi, m \rangle = \int_E \phi(e) m(de)$.

Let $\mathfrak{U}=((\Omega,\F,\mathbb{F},\prob),\xi,W,\mu,\mathfrak{r})$ be as in Definition \ref{existence:strategie_max}, so that $\Gamma=\prob\circ(X,\mathfrak{r},\mu)^{-1}$.
As ensured by \cite[Lemma C.2]{lacker2020convergence}, we have that, by choosing $Y_t=(X_t,\mu)$ taking values in $\R^d \times \contpdue$ as conditioning process, there exists a jointly measurable function $\hat{q}:[0,T] \times \R^d \times \contpdue \to \mathcal{P}(A)$ so that, for every $\phi:[0,T] \times \R^d \times \contpdue \times A \to \R$ bounded and measurable it holds
\begin{equation}\label{existence:lemma_mimicking:markovian_ctrl}
    \int_{A}\phi(X_t,\mu,a)\hat{q}_t(X_t,\mu)(da)=\attesa{\int_{A}\phi(X_t,\mu,a)\mathfrak{r}_t(da) \big\vert X_t,\mu} \quad \prob\text{-a.s.}, \; \text{a.e. } t \in [0,T],
\end{equation}
which we abbreviate as
\begin{equation*}
    \hat{q}_t(X_t,\mu)(da)=\attesa{\mathfrak{r}_t(da) \big\vert X_t,\mu} \quad \prob\text{-a.s.}, \; \text{a.e. } t \in [0,T].
\end{equation*}
Next, we manipulate the term of the functional $\mathfrak{p}$ in \eqref{existence:payoff_functional} which depends only upon $\Gamma$:
\begin{equation}\label{existence:lemma_mimicking:uguaglianze1}
\begin{aligned}
    \int \: &  \mathfrak{F} (y,q,m) \Gamma(dy,dm,dq) = \attesa{\int_0^T\int_A f(t,X_t,\mu_t,a)\mathfrak{r}_t(da)dt + g(X_T,\mu_T)} \\
    = & \int_0^T \attesa{\attesa{ \int_A f(t,X_t,\mu_t,a)\mathfrak{r}_t(da) \Big \vert X_t,\mu }}dt + \attesa{g(X_T,\mu_T)} \\
    = & \int_0^T\attesa{ \int_A f(t,X_t,\mu_t,a)\hat{q}_t(X_t,\mu)(da) }dt + \attesa{g(X_T,\mu_T)} \\
    = & \int_0^T\attesa{\attesa{ \int_A f(t,X_t,\mu_t,a)\hat{q}_t(X_t,\mu)(da) \Big \vert \mu }}dt + \attesa{\attesa{g(X_T,\mu_T) \Big \vert \mu }} \\
    = & \int_0^T\attesa{ \left \langle \int_A f(t,\cdot,\mu_t,a)\hat{q}_t(\cdot,\mu)(da), \mu_t \right \rangle } dt + \attesa{\left \langle g(\cdot,\mu_T),  \mu_T \right \rangle }.
\end{aligned}
\end{equation}
Second equality holds by Fubini's theorem and tower property of conditional expectation,
third equality holds by definition of the control \eqref{existence:lemma_mimicking:markovian_ctrl}, third and fourth equalities hold by tower property again, and fifth equality holds since, by consistency condition, $\mu_t(\cdot)=\prob(X_t \in \cdot \; \vert \; \mu)$.

\smallskip
We observe that, by choosing $\phi(t,x,m,a)=b(t,x,m_t,a)$ in \eqref{existence:lemma_mimicking:markovian_ctrl}, we have 
\begin{equation*}
    \int_A b(t,X_t,\mu_t,a)\hat{q}_t(X_t,\mu)(da)=\attesa{\int_A b(t,X_t,\mu_t,a)\mathfrak{r}_t(da) \big\vert X_t,\mu} \quad \prob\text{-a.s.}, \; \text{a.e. } t \in [0,T].
\end{equation*}
This is enough to apply \cite[Theorem~3.6]{brunick_shreve_mimicking}: indeed, in its terminology, we can take $\mathcal{E}=\R^d \times  \contpdue$, $\Phi: \mathcal{E} \times \contrd_0  \to \conttraj{\mathcal{E}}$ defined by $\Phi_t(x,m,y)=(x_t + y,m) \in \R^d \times \contpdue$, where $\contrd_0=\{ y \in \contrd: \; x_0=0\}$.
Set $Z_t=\Phi(X_t-X_0,X_0,\mu)=(X_t,\mu)$, 
where we note that the second component of $Z$ is constant in time as it is equal to the whole flow $\mu=(\mu_s)_{s \in [0,T]}$.
Then, such a result ensures that there exists a probability space $(\hat{\Omega},\hat{\F},\hat{\mathbb{F}},\hat{\prob})$, with $\hat{\Omega}$ Polish and $\hat{\F}$ its corresponding Borel $\sigma$-algebra, supporting an $\hat{\mathbb{F}}$-Brownian motion $\hat{W}$, a continuous $\mathcal{E}$-valued process $\hat{Z}$ so that there exists an $\hat{\mathbb{F}}$-adapted process $\hat{X}$ that satisfies
\begin{equation*}
    \hat{X}_t=\hat{X_0} + \int_0^T\int_A b(s,\hat{X}_s,\hat{\mu}_s,a)\hat{q}_s(\hat{X}_s,\hat{\mu}_s)(da)ds + \hat{W}_t, \quad \hat{Z}=\Phi(\hat{X}_t-\hat{X}_0,\hat{X}_0,\hat{\mu})
\end{equation*}
so that for every $t \in [0,T]$ it holds $\prob\circ Z_t^{-1}=\hat{\prob}\circ\hat{Z}_t^{-1}$.
This implies both that $\hat{\mu}$ and $\mu$ have the same law $\rho$ and that the consistency condition is satisfied, since $\hat{\prob}\circ(\hat{X}_t,\hat{\mu})^{-1}=\prob \circ (X_t,\mu)^{-1}=m_t(dx)\rho(dm)$.
Finally, since $Z$ is $\hat{\mathbb{F}}$-adapted, we deduce that $\hat{X}_0$ and $\hat{\mu}$ are $\F_0$-measurable and therefore $\hat{W}$, $\hat{X}_0$ and $\hat{\mu}$ are mutually independent.

Set $\hat{\Gamma}=\hat{\prob}\circ(\hat{X},\hat{\mathfrak{r}},\hat{\mu})^{-1}$.
Since the last term in the chain of equalities \eqref{existence:lemma_mimicking:uguaglianze1} depends only upon $\mu$ and $\mu$ and $\hat{\mu}$ share the same law, we can exploit the fact that $\hat{\mu}$ and $\hat{X}$ satisfy the consistency condition as well to get
\begin{equation*}
\begin{aligned}
     \: \int &  \mathfrak{F} (y,q,m) \Gamma(dy,dm,dq) =\int_0^T \E \quadre{  \left \langle \int_A f(t,\cdot,\mu_t,a)\hat{q}_t(\cdot,\mu)(da), \mu_t \right \rangle } dt + \E \quadre{ \left \langle g(\cdot,\mu_T),  \mu_T \right \rangle }\\
     = & \int_0^T \E^{\hat{\prob}} \quadre{  \left \langle \int_A f(t,\cdot,\hat{\mu}_t,a)\hat{q}_t(t,\cdot,\hat{\mu}_t,a)(da), \hat{\mu}_t \right \rangle} dt + \E^{\hat{\prob}} \quadre{ \left \langle g(\cdot,\hat{\mu}_T),  \hat{\mu}_T \right \rangle } \\
     = & \int_0^T \E^{\hat{\prob}} \quadre{ \E^{\hat{\prob}} \quadre{\int_A f(t,\hat{X}_t,\hat{\mu}_t,a)\hat{q}_t(\hat{X}_t,\hat{\mu})(da)} dt + \E^{\hat{\prob}} \quadre{ g(\hat{X}_T,\hat{\mu}_T) \Big \vert \hat{\mu}} } \\
     = & \int_0^T \E^{\hat{\prob}} \quadre{ \int_A f(t,\hat{X}_t,\hat{\mu}_t,a)\hat{q}_t(\hat{X}_t,\hat{\mu})(da) } dt + \E^{\hat{\prob}} \quadre{ g(\hat{X}_T,\hat{\mu}_T) } \\
     = & \int \mathfrak{F} (y,q,m) \hat{\Gamma}(dy,dm,dq).
\end{aligned}
\end{equation*}
Analogously, for every $\Sigma \in \mathcal{Q}$, we have
\begin{equation*}
    \int \mathfrak{F} (y,q,m)  \Sigma(dy,dq,m)\rho(dm) = \int \mathfrak{F} (y,q,m)  \Sigma(dy,dq,m)\hat{\rho}(dm),
\end{equation*}
which proves the desired statement about the payoff functional $\mathfrak{p}$.
\end{proof}

Finally, we show that it is always possible to find a strong solution to equation \eqref{existence:eq_processo_K} in the case of a feedback in state control process $\hat{q}_t(x,m)$, as given by Lemma \ref{existence:lemma_mimicking}:
\begin{lemma}[Strong solutions for feedback in state controls]\label{existence:mimicking:lemma_strong_existence}
Let $(\Omega,\F,\mathbb{F},\prob)$ be a filtered probability space satisfying the usual assumptions, with $\Omega$ Polish and $\F$ its Borel $\sigma$-algebra, supporting a $d$-dimensional $\mathbb{F}$-Brownian motion $W$, an $\F_0$-measurable $\R^d$-valued random $\xi$ with law $\nu$ and a $\F_0$-measurable random flow of measures $\mu$ in $\contpdue$ with law $\rho$.
Assume that $\xi$, $W$ and $\mu$ are mutually independent.
Let $\hat{q}:[0,T]\times \R^d \times \contpdue \to \mathcal{P}(A)$ be a measurable function, and suppose that there exists a solution of the SDE
\begin{equation}
    dX_t=\int_A b(t,X_t,\mu_t,a)\hat{q}_t(X_t,\mu)(da)dt + dW_t,
\end{equation}
so that it holds
\begin{equation*}
    \mu_t(\cdot)=\prob( X_t \in \cdot \; \vert \; \mu) \quad \prob\text{-a.s.}
\end{equation*}
for every $t \in [0,T]$.
Then, $X$ may be taken adapted to the $\prob$-augmentation of the filtration $\mathbb{F}^{\xi,\mu,W}=\sigma(\xi) \vee \sigma(\mu) \vee \mathbb{F}^W$.
In particular, there exists a progressively measurable function $\Phi:\contpdue \times \R^d \times \contrd \to \contrd$ so that $\Phi(\mu,\xi,W)=X$ $\prob$-a.s.
\end{lemma}
\begin{proof}
Set $B(t,x,m)=\int_A b(t,x,m_t,a)\hat{q}_t(x,m)(da)$.
$B$ is jointly measurable in $(t,x,m) \in [0,T] \times \R^d \times\contpdue$ with at most linear growth in $(x,m) \in  \R^d \times \contpdue$ for every $t \in [0,T]$.
The following hold:
\begin{enumerate}
    \item \label{mimicking:lemma_strong_existence:punto_eq_deterministica}
    For every $m \in \contpdue$, equation
    \begin{equation} \label{mimicking:lemma_strong_existence:eq_deterministica}
        dX^m_t=B(t,X^m_t,m)dt +dW_t, \quad X^m_0=\xi.
    \end{equation}
    admits a unique strong solution.
    Moreover, let $P^m=\prob \circ (X^m)^{-1}$.
    Then, the map $\contpdue \ni m \mapsto P^m \in \probmeasures{\contrd}$ is measurable.
    \item \label{mimicking:lemma_strong_existence:esistenza_sol_forte}
    There exist a continuous $\mathbb{F}$-adapted process $X$ solution to
    \begin{equation}\label{mimicking:lemma_strong_existence:eq_aleatoria}
        dX_t=B(t,X_t,\mu)dt + dW_t, \quad X_0=\xi.
    \end{equation}
    $X$ is adapted to the $\prob$-augmentation of the filtration $\mathbb{F}^{\xi,\mu,W}$.
    \item \label{mimicking:lemma_strong_existence:pathwise_uniqueness}
    Pathwise uniqueness holds, in the following sense: 
    suppose there exists a pair of continuous $\mathbb{F}$-adapted processes $(X^1,X^2)$ which satisfy equation \eqref{mimicking:lemma_strong_existence:eq_aleatoria} so that $(X^1_s,X^2_s)_{s \leq t}$ is conditionally independent of $\F^{\xi,\mu,W}_T$ given $\F^{\xi,\mu,W}_t$ for every $t \in [0,T]$.
    Then, $\prob(X^1_t=X^2_t, \; 0 \leq t \leq T)=1$.
    
    \item \label{mimicking:lemma_strong_existence:punto_legge_soluzione}
    The joint law of $X$ and $\mu$ is given by 
    \begin{equation*}
        \prob\circ(X,\mu)^{-1}=P^m(dx)\rho(dm).
    \end{equation*}
\end{enumerate}
This properties can be proven with the same methods of \cite[Appendix~A]{lacker2020convergence} and \cite[Appendix~A]{lacker_leflem2022}.
We just point out that the results therein do not hold automatically in our case, since $B$ is not progressively measurable in the measure flow $m$, in the sense of \cite{lacker2020convergence,lacker_leflem2022}.
Nevertheless, since we require $\mu$ to be $\F_0$-measurable, the same arguments lead to the results above.

\smallskip
Let $X$ be as in the statement of the lemma.
We first show that the joint law of $X$ and $\mu$ is given by $P^m(dx)\rho(dm)$.
Let $\prob^m(\cdot)=\prob(\cdot \vert \mu=m)$ be a version of the regular conditional probability of $\prob$ given $\mu=m$.
Then, since $\xi$, $W$ and $\mu$ are mutually independent, 
$\prob^m\circ(\xi,W)^{-1}=\prob\circ(\xi,W)^{-1}$ for $\rho$-a.e. $m$, and, by exploiting the fact the $\F^{\xi,\mu,W,X}_t$ is countably generated for every $t$, $W$ is an $\mathbb{F}^{\xi,\mu,W,X}$-Brownian motion under $\prob^m$ as well.
Therefore, $X$ satisfies equation \eqref{mimicking:lemma_strong_existence:eq_deterministica} on $(\Omega,\F,\mathbb{F}^{\xi,\mu,W,X},\prob^m)$ for $\rho$-a.e. $m \in \contpdue$.
By point \ref{mimicking:lemma_strong_existence:punto_eq_deterministica}, $\prob^m \circ X^{-1}=P^m$ for $\rho$-a.e. $m \in \contpdue$, which implies that $\prob\circ(X,\mu)^{-1}=P^m(dx)\rho(dm)$.

\smallskip
It can be shown by straightforward calculations that $(X_s)_{s\leq t}$ is conditionally independent of $\F^{\xi,\mu,W}_T$ given $\F^{\xi,\mu,W}_t$, for every $t \in [0,T]$.
Since pathwise uniqueness holds by point \ref{mimicking:lemma_strong_existence:pathwise_uniqueness}, this implies that $X$ is indistinguishable from an $\mathbb{F}^{\xi,\mu,W}$-adapted solution to equation \eqref{mimicking:lemma_strong_existence:eq_aleatoria}.
\end{proof}

\providecommand{\bysame}{\leavevmode\hbox to3em{\hrulefill}\thinspace}
\providecommand{\MR}{\relax\ifhmode\unskip\space\fi MR }
\providecommand{\MRhref}[2]{%
  \href{http://www.ams.org/mathscinet-getitem?mr=#1}{#2}
}

\textbf{Acknowledgments.} \hspace{0.25cm} 
Markus Fischer acknowledges financial support from the research project "Stochastic mean field control and the Schrödinger problem" (BIRD229791) of the University of Padua.
Luciano Campi and Markus Fischer received partial financial support from EU – Next Generation EU – PRIN2022 (2022BEMMLZ) CUP: D53D23005780006.

\end{document}